\newcommand\Modified{June 8, 2009}
\documentclass{amsart}

\usepackage{amsfonts, amsmath,amscd, amssymb, latexsym, mathrsfs, stmaryrd, verbatim, wasysym }
\usepackage{graphicx}
\usepackage[all]{xy}

\newtheorem{theorem}{Theorem}
\newtheorem{definition}{Definition}

\newtheorem{corollary}[theorem]{Corollary}

\newtheorem{lemma}[theorem]{Lemma}
\newtheorem{proposition}[theorem]{Proposition}

\numberwithin{equation}{section}
\numberwithin{theorem}{section}

 \newcommand\datver[1]{\def\datverp%
 {\par\boxed{\boxed{\text{Version: #1; Run: \today}}}}}
 
 \usepackage{color}
\definecolor{darkgreen}{cmyk}{1,0,1,.2}
\definecolor{m}{rgb}{1,0.1,1}

\renewcommand{\bar}{\overline}

\newcommand{\df}[1]{\mathfrak{#1}}
\renewcommand{\hat}[1]{\widehat{#1}}

\newcommand{\rest}[1]{\big\rvert_{#1}} 
\newcommand{\script}[1]{\textsc{#1}}
\renewcommand{\tilde}{\widetilde}
\newcommand{\wt}[1]{\widetilde{#1}}
\newcommand{\wh}{\widehat}
\newcommand\Iielaga{{}^{\mathrm{iie}}\Lambda^*_{\Gamma}} 

\newcommand{\dR}{\mathrm{dR}}
\newcommand{\Image}{\operatorname{Image}}

\newcommand\eps\varepsilon

\newcommand\pa{\partial}

\newcommand\spec{\mathrm{spec}}
\newcommand\Spec{\mathrm{Spec}}

\newcommand\ie{\operatorname{ie}}
\newcommand\iie{\operatorname{iie}}
\newcommand\pie{\operatorname{pie}}

\newcommand\Ie{{}^{\ie}}
\newcommand\Iie{{}^{\iie}}

\newcommand\CI{{\mathcal{C}}^{\infty}}
\newcommand\CIc{{\mathcal{C}}^{\infty}_c}
\newcommand\CmI{{\mathcal{C}}^{-\infty}}

\newcommand{\lrpar}[1]{\left( #1 \right)}
\newcommand{\lrspar}[1]{\left[ #1 \right]}
\newcommand\ang[1]{\left\langle #1 \right\rangle}
\newcommand{\lrbrac}[1]{\left\lbrace #1 \right\rbrace}
\newcommand{\norm}[1]{\lVert #1 \rVert}
\newcommand{\abs}[1]{\left\lvert #1 \right\rvert}

\DeclareMathOperator*{\btimes}{\times} 

\newcommand\Diff{\operatorname{Diff}}
\newcommand\Dir{\operatorname{Dir}}
\newcommand{\Dom}{\operatorname{Dom}}
\newcommand\dvol{\operatorname{dvol}}

\newcommand{\Hom}{\operatorname{Hom}}
\newcommand\Id{\operatorname{Id}}

\newcommand{\Ind}{\operatorname{Ind}}

\newcommand{\loc}{\operatorname{loc}}

\newcommand{\sign}{\operatorname{sign}}

\newcommand{\supp}{\operatorname{supp}}

\newcommand\Mand{\text{ and }}

\newcommand\Mforevery{\text{ for every }}

\newcommand\Mif{\text{ if }}

\newcommand\Mst{\text{ s.t. }}

\newcommand\Mthen{\text{ then }}

\DeclareMathAlphabet{\mathpzc}{OT1}{pzc}{m}{it}
\newcommand{\cl}[1]{\mathpzc{cl}\left( #1 \right)}

\newcommand\paperintro%
        {%
         }
\newcommand\paperbody%
        {%
         }

\newcommand\bbC{\mathbb{C}}

\newcommand\bbN{\mathbb{N}}

\newcommand\bbQ{\mathbb{Q}}
\newcommand\bbR{\mathbb{R}}

\newcommand\bbZ{\mathbb{Z}}

\newcommand\cB{\mathcal{B}}

\newcommand\cD{\mathcal{D}}
\newcommand\cE{\mathcal{E}}
\newcommand\cF{\mathcal{F}}

\newcommand\cH{\mathcal{H}}
\newcommand\cI{\mathcal{I}}

\newcommand\cK{\mathcal{K}}
\newcommand\cL{\mathcal{L}}
\newcommand\cM{\mathcal{M}}

\newcommand\cO{\mathcal{O}}

\newcommand\cU{\mathcal{U}}
\newcommand\cV{\mathcal{V}}

\newcommand\bN{\mathbf{N}}

\newcommand{\RR}{\mathbb{R}}

\newcommand{\e}{\epsilon}
\newcommand{\del}{\partial}

\newcommand{\calH}{{\mathcal H}}

\newcommand{\calU}{{\mathcal U}}
\newcommand{\calV}{{\mathcal V}}
\newcommand{\calW}{{\mathcal W}}

\newcommand{\frakS}{{\mathfrak S}}

\newcommand{\ovl}{\overline}

\datver{\Modified}

\begin{document}
\title[The signature package on Witt space, I]{The signature package on Witt spaces, I. \\ Index classes}

\author{Pierre Albin}
\address{Department of Mathematics, MIT}
\email{pierre@math.mit.edu}
\author{Eric Leichtnam}
\address{CNRS Institut de Math\'ematiques de Jussieu}
\author{Rafe Mazzeo}
\address{Department of Mathematics, Stanford University}
\email{mazzeo@math.stanford.edu}
\author{Paolo Piazza}
\address{Dipartimento di Matematica, Sapienza Universit\`a di Roma}
\email{piazza@mat.uniroma1.it}


\begin{abstract}
We give a parametrix construction for the signature operator on any compact,
oriented, stratified pseudomanifold $X$ which satisfy the Witt condition. This
construction is inductive. It is then used to show that the signature operator
is essentially self-adjoint and has discrete spectrum of finite multiplicity,
so that its index -- the analytic signature of $X$ -- is well-defined. We then show
how to couple this construction to a $C^*_r\Gamma$ Mischenko bundle 
associated to any Galois covering of $X$ with covering group $\Gamma$. 
The appropriate analogues of these same results are then proved, and it follows 
that we may define an analytic signature class as an element of the
$K$-theory of $C^*_r\Gamma$.  In a sequel to this paper we establish in this
setting the full range of conclusions for this class which sometimes
goes by the name of the signature package.
\end{abstract}

\maketitle

\section{Introduction}
Let $X$ be an orientable closed compact Riemannian manifold with fundamental group $\Gamma$. Let $X^\prime$ 
be a Galois $\Gamma$-covering and $r: X\to B\Gamma$ a classifying map for $X^\prime$. What we call `the signature 
package' for the pair $(X,r:X\to B\Gamma)$ refers to the following collection of results: 
\begin{enumerate}
\item the signature operator with values in the Mischenko bundle $r^* E\Gamma \times_\Gamma C^*_r\Gamma$
defines  a signature index class  $\Ind (\wt \eth_{\sign})\in K_* (C^*_r \Gamma)$, $* \equiv \dim X \;{\rm mod}\; 2$; 
\item the signature index class is a bordism invariant; more precisely  it defines a group homomorphism 
$\Omega^{{\rm SO}}_* (B\Gamma) \to K_* (C^*_r \Gamma)$;
\item the signature 
index class is a  homotopy invariant; 
\item  there is a  K-homology signature class $[\eth_{\sign}]\in K_* (X)$ whose Chern
 character is, rationally, the Poincar\'e dual of the L-Class;
\item the assembly map $\beta: K_* (B\Gamma)\to K_* (C^*_r\Gamma)$ sends the
class $ r_* [\eth_{\sign}]$ into $\Ind (\wt \eth_{\sign})$;
\item  if  the assembly map  is rationally injective one  can deduce from the above results the homotopy invariance of
 Novikov higher signatures.
\end{enumerate}

We label the {\em full} signature package the one decorated by the following item

\smallskip
 \begin{list}
 {(7)} \item there is a ($C^*$-algebraic) symmetric signature $\sigma_{C^*_r\Gamma} (X,r)\in K_* (C^*_r \Gamma)$,
which is topologically defined, a bordism invariant $\sigma_{C^*_r\Gamma}: \Omega^{{\rm SO}}_* (B\Gamma) \to 
K_* (C^*_r \Gamma)$ and,  in addition, equal to the signature index class.
\end{list}

\smallskip
For history and background see \cite{oberwolfach} \cite{rosenberg-anft} and for a survey we refer to \cite{Kasparov-contemporary}.

\smallskip
{\it This is the first of two papers in which the signature package will be formulated
and established for a class of stratified pseudomanifolds known as Witt spaces.}

\smallskip
In the present paper we shall concentrate on the analytic side of the signature package; the forthcoming second part 
will treat the more topological aspects. More precisely, the goal of the present paper is to prove the following

\begin{theorem}
Let $\wh{X}$ be any smoothly stratified pseudomanifold which satisfies the Witt hypothesis. Let $g$ be any
adapted Riemannian metric on the regular part of $\wh{X}$. Denote by $\eth$ either the Hodge-de Rham operator 
$d + \delta$  or the signature operator $\eth_{\sign}$ associated to $g$. Then the following is true:
\begin{itemize}
\item Let $u$ be in the maximal domain of $\eth$ as an operator on $L^2_{\iie}(X;\Iie\Lambda^*X)$. Then
\[
u \in \bigcap_{0 < \e < 1}\rho^{\eps}H^1_{\iie}(X;\Iie\Lambda^*X)
\].
\item The maximal domain $\cD_{\max}(\eth) $ is compactly embedded in $L^2_{\iie}$.
\end{itemize}
As a consequence, the minimal and maximal domains of $\eth$ are equal, the de Rham operator and the signature operator
are essentially self-adjoint and have only discrete spectrum of finite multiplicity. 
Moreover, there is a well defined signature class $[ \eth_{\sign}]\in K_* (\wh{X})$,
with $*=\dim \wh{X} \; {\rm mod}\; 2$, which is independent of the choice  of the adapted metric on 
the regular part of $\wh{X}$. In particular, in the even dimensional case, the index 
of the signature operator is well-defined.

If $\wh{X}' \to \wh{X}$ is a Galois covering with group $\Gamma$ and $r: \wh{X} \to B\Gamma$ is the
classifying map, then the signature operator $\wt \eth_{\sign}$ with coefficients in the Mishchenko bundle,
together with the $C^*_r\Gamma$-Hilbert module $L^2_{\iie,\Gamma}(X;\Iie\Lambda_\Gamma^*X)$ 
define an unbounded Kasparov $(\bbC,C^*_r\Gamma)$-bimodule and hence a class in $KK_* (\bbC, C^*_r \Gamma)$
$=$$K_* (C^*_r\Gamma)$,
which we call the index class associated to $\wt \eth_{\sign}$ and denote by $\Ind (\wt \eth_{\sign})\in K_* (C^*_r\Gamma)$.
If $[[\eth_{\sign}]]\in KK_*(C(\widehat{X})\otimes C^*_r\Gamma, C^*_r\Gamma)$ is the class obtained from $[\eth_{\sign}]\in 
KK_*(C(\widehat{X}),\bbC)$ by tensoring with  $C^*_r\Gamma$, then $\Ind (\wt \eth_{\sign})$ is equal to the 
Kasparov product of the class defined by the Mishchenko bundle $[\widetilde{C^*_r}\Gamma]\in  KK_0(\bbC,C(\widehat{X})\otimes 
C^*_r\Gamma)$ with  $[[\eth_{\sign}]]$:
\begin{equation}\label{tensor1}
\Ind (\wt \eth_{\sign})= [\widetilde{C^*_r}\Gamma)]\otimes [[\eth_{\sign}]]
\end{equation}
In particular, the index class $\Ind (\wt \eth_{\sign})$ does not depend on the choice of the adapted metric on 
the regular part of $\wh{X}$. Finally, if $\beta:  K_* (B\Gamma)\to K_* (C^*_r\Gamma)$ denotes the assembly 
map in K-theory, then 
\begin{equation}
\beta(r_*  [\eth_{\sign}])=\Ind (\wt \eth_{\sign})\text{ in }  K_* (C^*_r\Gamma)
\end{equation}
\end{theorem}

\medskip
The objects and notation in the statement of this main theorem will be gradually introduced in the remainder
of this introduction and in later parts of the paper. 

The starting point of the signature package for closed oriented manifolds is  the signature theorem, surely
one of the triumphs of mid-twentieth century mathematics.
The original version proved by Hirzebruch in the mid 1950's equates the topological signature $\sigma_{\mathrm{top}}(X)$, 
i.e.\ the signature of the intersection form with respect to cup product on middle degree cohomology, of any smooth 
closed $4k$-dimensional oriented manifold $X$ with the so-called $\mathcal{L}$-genus of $X$, $\mathcal{L} (X)$, the 
characteristic number of that manifold obtained by pairing the $L$-class of $X$, $L(X)$, which is a universal polynomial 
in the Pontrjagin classes, with the fundamental class of $X$. The signature formula is the equality
\[
\sigma_{\mathrm{top}}(X) = \mathcal{L} (X) := \langle L(X),[X] \rangle. 
\]
This formula provided one of the main inspirations for the Atiyah-Singer index theorem, and of course is a special case:
by Hodge theory, $\sigma_{\mathrm{top}}(X)$ is equal to the index of the signature operator $\eth_{\sign}$, and specializing
the Atiyah-Singer index formula to this operator one obtains Hirzebruch's formula.  Because of its fundamental
nature and the enormous range of applications, this signature formula has been generalized in many different ways.
These include the Atiyah-Patodi-Singer formula on manifolds with boundary \cite{APS}, Melrose's approach to this result 
\cite{APS Book} which has proved to be particularly well adapted for generalizations, the signature formul\ae\ for manifolds 
with $\bbQ$-rank $1$ cusps by Atiyah-Donnelly-Singer \cite{ADS} and M\"uller \cite{Mue}, and for manifolds with fibred 
cusp ends by Vaillant \cite{Vai} and the signature formula for manifolds with isolated conic singularities by Cheeger \cite{Ch}, 
cf.\ also the treatments by Lesch \cite{L} and Br\"uning-Seeley \cite{BS}. (We are omitting here any mention of the 
extensive literature on families index theorems.) If $X^\prime\to X$ is a Galois $\Gamma$-covering then there is a formula
due to Atiyah for the $L^2$-signature of $X^\prime$ \cite{atiyah-coverings} and  the higher signature formula of 
Connes-Moscovici \cite{CM} \cite{Lott}. The latter provides an alternate route to the proof of the homotopy invariance 
of the higher signatures for Gromov hyperbolic groups, independent of the use of the assembly map as it appears in the 
signature package. Generalizations of these results to $\Gamma$-coverings with boundary are due to
Vaillant \cite{vaillant-diploma} for the numeric $L^2$-signature and to  Leichtnam, Lott and Piazza \cite{LP-SMF} \cite{LLP}
for the higher signatures, see \cite{LPFOURIER} for a survey.

One particularly interesting problem is to establish the signature package for a general class of stratified pseudomanifolds, 
the definition of which is reviewed below, since this is a large and interesting class of singular space to which one 
can reasonably hope to extend many of the key features of analysis of elliptic operators and of differential topology. 
Such extensions are far from being completely understood and there are considerable challenges, both analytic and 
topological, including the proper definition of $\eth_{\mathrm{sign}}$ as a self-adjoint operator and the development of its 
mapping properties, leading to an unambiguous definition of its index, etc.

Indeed, much of this has already been accomplished. The seminal work on the analytic side is by Cheeger \cite{Ch},
with his development of elliptic theory on `spaces with cone-like singularities', and on the topological side by 
Goresky-MacPherson with their introduction of the theory of intersection homology, see \cite{GM} as well as the
work of Siegel \cite{Se}. One of Cheeger's important discoveries was that if the stratified space satisfies the 
Witt hypothesis, which is a vanishing condition for certain local middle-degree intersection cohomology groups, then 
(for a large class of metrics) the signature operator is essentially self-adjoint, with discrete spectrum, and he proved 
much of what is necessary to define the analytic signature for this general class of spaces and to equate
it with the topological signature defined using the intersection homology groups. However, for simplicity of
exposition he focused on the case where the metric is flat on each stratum.  Ultimately this does not hinder
the general applicability of his results, but one main goal of the present paper is to provide an alternate analytic 
framework to address these analytic issues, in particular one which is sufficiently robust so that one may directly 
treat arbitrary compact stratified pseudomanifolds satisfying the Witt condition and endowed with fairly general 
`iterated edge' metrics. Moreover, our methods can also be extended directly to treat the higher signature operator 
in the Mischenko-Fomenko calculus on these spaces. Thus a corollary of our main result is the existence of
an analytic signature and a higher analytic signature for an arbitrary compact Witt space: the former
is an integer while the latter is class in the K-theory of $C^*_r \Gamma$, the reduced group $C^*$-algebra of the fundamental group.  
Topological applications of this index class will be treated in a second paper.  Later in this introduction we
describe some of the other recent analytic approaches to operators and spaces of this type.

We first recall that the signature operators $\eth_{\mathrm{sign}}^\pm$ on an even dimensional oriented compact Riemannian 
manifold $(X,g)$ are the elliptic operators $d + \delta$ acting between $\Omega^*_{\pm}(X)$ and $\Omega^*_{\mp}(X)$
(i.e.\ forms in the $\pm 1$ eigenspaces of the natural involution on $\Omega^*$ induced by the Hodge star).
The analytic signature of the manifold $X$ is, by definition, the index of $\eth^+_{\mathrm{sign}}$, or equivalently,
\[
\sigma_{\mathrm{an}}(X) = \dim \ker \eth^+_{\mathrm{sign}} - \dim \ker \eth^-_{\mathrm{sign}}.
\]
We may tensor $\eth_{\mathrm{sign}}$ with any finite dimensional flat vector bundle $E$ over $X$, to define the index 
$\sigma_{\mathrm{an}}(X,E)$.  A more sophisticated construction leads to the notion of the signature index class of $X$, 
as we now explain. Let $\Gamma = \pi_1(X)$ and $\widetilde{X}$ the universal cover 
of $X$. Let $C^*_r \Gamma$ be the reduced group $C^*$-algebra for $\Gamma$ and consider the associated bundle 
$\widetilde{C^*_r}\Gamma:=\widetilde{X}\times_\Gamma C^*_r \Gamma$. This is  a bundle of finitely generated projective 
$C^*_r \Gamma$-modules of rank $1$ with a flat connection inherited from the trivial connection on $\widetilde{X}
\times C^*_r \Gamma$. Using this data we can define the signature operator twisted by $\widetilde{C^*_r}\Gamma$ ; this is an 
elliptic operator in the Mischenko-Fomenko pseudodifferential calculus and is thus invertible modulo 
$C^*_r\Gamma$-compact operators of the Hilbert $C^*_r\Gamma$-module $L^2 \Omega^* (X,\widetilde{C^*_r}\Gamma)$. In particular, 
it defines an index class in $K_0 (C^*_r \Gamma)$ when $\dim X$ is even, and an index class in $K_1 (C^*_r \Gamma)$ 
when $\dim X$ is odd, both denoted $\Ind (\widetilde{\eth}_{\sign})$. These signature index classes 
$\Ind (\widetilde{\eth}_{\sign})\in K_* (C^*_r \Gamma)$ play a fundamental role in almost all proofs of the
cases where the Novikov conjecture on the homotopy invariance of Novikov  higher signatures is known to be true,
see \cite{oberwolfach} for historical remarks and background.

As noted earlier, there are analogues of the regular and higher signature theorems on manifolds with boundary
with metric of product type near the boundary (or equivalently, with infinite cylindrical ends), but we turn directly
to the closely related setting of compact spaces with isolated conic singularities. Recall that a metric cone is a product 
$\bbR^+ \times F$ with metric (at least quasi-isometric to) $dr^2 + r^2 h$, where $(F,h)$ is a (usually compact) 
Riemannian space. A Riemannian stratified space $X$ has isolated conic singularities if its singular set consists 
of a discrete collection of points, all of which have neighbourhoods of this form. There are by now very many 
approaches to understanding elliptic operators on such spaces, but the first systematic approach directed at 
the sort of applications we are discussing here was accomplished by Cheeger \cite{Ch}, who pointed out that the 
Hodge-de Rham operator $d + \delta$ and Hodge Laplacian $\Delta$ acting on forms of all degrees are both 
essentially self-adjoint for some choice of metric on $X$ if the middle degree cohomology of the cross-section $F$ 
at each cone point vanishes. (This is automatic when $\dim X$ is even so that $\dim F$ is odd.) This is the 
simplest instance of the Witt condition. If $X$ is not Witt, then (again at least for suitable metrics) 
self-adjoint extensions are in bijective correspondence with half-dimensional subspaces of this cohomology 
group which are Lagrangian with respect to a natural symplectic pairing. In either case, any self-adjoint 
extension of $\eth_{\mathrm{sign}}$ on $L^2$ has discrete spectrum with finite multiplicity, and if $\dim X$ 
is even and $X$ is Witt, then $\sigma_{\mathrm{an}}(X)$ is equal to the signature of the intersection form on 
the middle degree intersection cohomology of $X$. (In fact, the Witt condition is used again here to
ensure that there is a unique middle perversity intersection cohomology group; in general, one
must take the signature of the intersection form on the image of the lower middle perversity
intersection cohomology into the upper middle perversity intersection cohomology of middle degree.) 
In this setting, exactly as for manifolds with boundary with the Atiyah-Patodi-Singer boundary condition,
the integral of the $L$-differential form  over $X$ is no longer a topological invariant, and there is an extra term 
in the signature formula equal to half the eta invariant of the associated signature operator on the link 
or boundary. 

Although Cheeger's analysis was specifically adapted to the conic geometry, all of this can also be deduced
using Melrose's $b$-calculus, which is a much more powerful method which can be used to analyze operators
associated to the conformally related asymptotically cylindrical metric, which takes the form $r^{-2}dr^2 + h$ 
near each conic point. We refer to Melrose's book \cite{APS Book} for a detailed explanation of this, but also 
to \cite{Mazzeo:edge} for a treatment of the $b$-calculus as a special case of the edge calculus and \cite{GM} 
which focuses on its use in the conic settng. We shall employ a similar idea in a more general setting.

A stratified space $X$ is said to have a simple edge singularity if it has only one singular stratum $Y$
and a neighbourhood of $Y$ in $X$ is identified with a bundle of truncated cones over a compact smooth 
manifold $F$, which is called the link of the edge. The metric $g$ is required to be conic on each
fibre of this cone bundle. The space is called Witt if $H^{f/2}(F) = 0$, $f := \dim F$. As before, this condition is 
vacuous if $\dim F$ is odd. The analytic techniques needed to understand $\eth_{\mathrm{sign}}$ in this setting 
are more involved; one particularly comprehensive method uses the pseudodifferential edge calculus 
\cite{Mazzeo:edge}. The paper \cite{Hunsicker-Mazzeo} uses this machinery to generalize many facts from
the conic setting to that of simple edges, in particular that for any edge metric there is a Hodge theory which 
identifies $L^2$ closed and coclosed forms with intersection cohomology classes of that space; the signature 
theorem now involves an eta form correction term. As in the isolated conic case, and importantly in the present paper
as well, the analysis in \cite{Hunsicker-Mazzeo} proceeds by relating $\eth_{\mathrm{sign}}$ to an elliptic operator associated
to the {\it complete} conformally related metric $\tilde g$ obtained by dividing $g$ by $r^2$ and then using 
the pseudodifferential edge calculus \cite{Mazzeo:edge} associated to   $\tilde g$.  A more topological approach to 
this and related results was attained by Cheeger and Dai \cite{Cheeger-Dai}.

There is an interesting class of stratified spaces which may be obtained by iterating this `coning
and edging' procedure. We define a cone $C(F)$ over any compact but possibly singular space $F$
just as above; an edge with link $F$ is then a bundle over a smooth space $Y$ with fibres $C(F)$.
Spaces with simple conic or edge singularities have already been defined above, and we say that these
spaces have depth $1$. An iterated edge space of depth $k$ is one which near any point has a neighbourhood
which is either a cone or an edge with link a compact iterated edge space of depth $k-1$. We present this 
definition more carefully in \S 2. We are omitting some very interesting 
classes of stratified spaces, however, e.g.\ those with various types of cusp singularities. There is
a natural class of incomplete iterated edge metrics which generalize the conic and edge metrics above.

An iterated edge space $X$ satisfies the Witt condition if the link $F$ at any conic point or any edge 
has vanishing (upper and lower middle perversity) intersection cohomology, $I\! H^{f/2}(F) = 0$,
$f = \dim F$. If $X$ is Witt, then the topological signature is well-defined without additional choices. 

Cheeger was the first to realize the inherent tractability of studying elliptic operators on this class of 
spaces, and \cite{Ch} describes how to set up an inductive procedure to obtain certain analytic results on 
iterated edge spaces of arbitrary depth. His technique relies heavily on analysis of the heat kernel.
Those methods certainly generalize beyond the specific results he obtained, but may not provide the 
more detailed analytic results that can be obtained in the simple edge case using pseudodifferential operators, 
e.g.\ those concerning sharp asymptotics of solutions of the signature operator near the singular set. 
Even more daunting would be the analysis of self-adjoint boundary conditions for the signature 
operator in the non-Witt case. All of this ``should'' be tractable if one were to set up an `iterated edge 
pseudodifferential calculus',  but carrying that out will involve many substantial technicalities, and in
any case has not been done yet. The present paper is meant to steer some sort of middle ground. 
We present a parametrix construction which is crude by the standards of geometric microlocal
analysis, but does allow one to obtain the necessary analytic information for the signature operator 
on iterated edge spaces with fairly general metrics. Our main goals here are the same as above: to show 
that if $X$ is a compact oriented iterated edge space satisfying the Witt condition, with an adapted 
iterated edge metric, then $\eth_{\mathrm{sign}}$ is essentially self-adjoint and has discrete spectrum with 
finite multiplicity, so that
$\sigma_{\mathrm{an}}(X)$ is well-defined. We work directly with $\eth_{\mathrm{sign}}$ itself, rather than 
its heat kernel, and some of the main work involves a justification of the perturbation theory needed
to pass from the model problems along each stratum to the actual problem.  Finally, we show how to
couple all of this with the $C^*_r\Gamma$ bundle so as to define the higher analytic signature 
of $X$ as well.

Now let us give a bit more detail of our analytic techniques. If $(X,g)$ is a space with isolated conic 
singularity, then $\eth_{\mathrm{sign}}$ can be written as $r^{-1}D$, where $D$ is an elliptic differential 
$b$-operator of order $1$; in local coordinates $r \geq 0$ and $z$ on $F$,
\[
D = A(r,z) \left(r\del_r + \eth_{\mathrm{sign}, F}\right).
\]
The final term on the right here is simply the signature operator on the link $F$. Mapping properties of
the signature operator and regularity properties for solutions of $\eth_{\mathrm{sign}}u = 0$ can be deduced 
from the corresponding properties for $D$, which are in turn direct consequences of Melrose's 
pseudodifferential $b$-calculus \cite{APS Book}. Similarly, if $(X,g)$ has simple edge singularities, 
then once again $\eth_{\mathrm{sign}} = r^{-1}D$ where $D$ is an elliptic edge operator locally of the form
\[
D = A(r,y,z) \left(r\del_r + \sum B_i(r,y,z) r\del_{y_i} + \eth_{\mathrm{sign},F}\right).
\]
Here $r$ is the radial variable in the cone fibres, $z \in F$ and $y$ are coordinates on the edge.
This is an elliptic differential edge operator and the pseudodifferential edge calculus can be
used to deduce all necessary properties of the signature operator $\eth_{\mathrm{sign}}$. 

Finally, let $(X,g)$ be an iterated edge space and $Y$ is a stratum of maximal depth, so that $Y$ is a 
compact smooth manifold without boundary and some neighbourhood of $Y$ is a cone bundle with link $F$, 
where $F$ is itself an iterated edge space of one depth less than $X$. If $r$ is the radial coordinate in
this cone bundle, then $\eth_{\mathrm{sign}} = r^{-1}D$ where $\eth_{\mathrm{sign},F}$ is again an iterated
edge operator, but on a space which is one step less singular. We do not conformally rescale in all
radial variables around the strata of $X$ of smaller depth, but instead consider the space with
metric $r^{-2}g$. The key idea is to use induction, so that we assume that we know all the 
properties of $\eth_{\mathrm{sign},F}$ at this maximal depth stratum and from there wish to deduce
them for $X$ itself. Unlike the simple edge case, however, we cannot rely on the pseudodifferential
edge calculus. In its place, we follow some of the main steps as we would using that theory,
but partially replacing the use of parametrices with a priori estimates. In other words, to
the extent possible, we treat $\eth_{\mathrm{sign},F}$ as a `black box' whose properties we only know
through induction.  This is a crude parametrix method compared to the rather detailed results
that should be available once the methods of \cite{Mazzeo:edge} are generalized to the
iterated edge setting. That generalization is work in progress by the first and third authors
and Richard Melrose, but is anticipated to be fairly intricate, so another goal of the present work 
is to find some middle ground which is sufficient to establish the results needed here, but
is not too complicated. In addition, our approach can be directly adapted when $\eth_{\mathrm{sign}}$ 
is coupled to a $C^*$ bundle, and hence the main theorem in the higher setting can be deduced with
fairly little extra effort from the `ordinary' case.  

Preliminary to this analysis, however, we present in \S 2 a fairly extensive discussion of the class
of smoothly stratified spaces. Part of the reason is to normalize notation, which is not uniform in
the various references to this material, but more significantly, we also establish an equivalence
between this class of spaces and the class of manifolds with corners with iterated fibration structures
introduced by Melrose; the maps providing this equivalence are resolution (blowup) and blowdown, 
respectively. 

The main result of this paper establishes the existence of an index, or an index class, as
quoted in the statement of the main theorem at the beginning of this introduction. The actual
signature theorem requires a substantially more topological argument, and because of this
we develop that material in a separate paper.

As noted earlier, there do exist various classes of pseudodifferential operators associated to iterated edge
spaces. We mention in particular the extensive contributions by Schulze and his collaborators. Schulze's
recent survey \cite{Sch-MSRI} contains a good description of many of the problems, methods and results
accessible by his approach, as well as good list of references of related work. We also mention the papers by
Nazaikinskii, Savin and Sternin \cite{NSS1} and \cite{NSS2}, and the work of Ammann, Lauter and Nistor \cite{ALN}.  
The latter paper describes a calculus slightly richer than the uniform calculus described in \S 4.2 below, 
but does not appear to allow one to handle the specific problems considered here; Schulze's work is geared
toward understanding rather general boundary conditions and, along with \cite{NSS1}, \cite{NSS2}, toward
applications in $K$-theory. In particular, none of these seem to directly apply to the signature operator.

P.A., R.M. and P.P. all wish to thank MSRI for hospitality and financial support during the Fall Semester of 2008
when much of the latter stages of this paper were completed. P.A. was partly supported by an NSF Postdoctoral 
Fellowship, and wishes to thank Stanford for support during visits; R.M. was supported by NSF grants 
DMS-0505709 and DMS-0805529, and enjoyed the hospitality and financial support of MIT, 
Sapienza Universit\`a di Roma and the Beijing International Center for Mathematical Research;
P.P. wishes to thank the CNRS and  Universit\'e Paris 6 for financial support during  visits to Institut 
de Math\'ematiques de Jussieu in Paris. E.L was  partially supported during visits to 
Sapienza Universit\`a di Roma by CNRS-INDAM (through the bilateral agreement GENCO (Non commutative 
Geometry)) and the Italian {\it Ministero dell' Universit\`a  e della  Ricerca Scientifica} (through the 
project "Spazi di moduli e teoria di Lie").

The authors are particularly grateful Richard Melrose for his help and encouragement at many stages of 
this project and for allowing us to include some of his unpublished ideas in section 2.
P.P. also wishes to thank M. Banagl and S. Weinberger for very helpful discussions.

\section{{Stratified spaces and resolution of singularities}} \label{Resolution}
This section contains a description of the class of smoothly stratified pseudomanifolds. We begin by 
recalling the notion of a stratified space with `control data', which is a topological space with a 
decomposition into a union of smooth strata, each with a specified tubular neighbourhood with fixed 
product decomposition, all satisfying several basic axioms. This material is taken from the paper of
Brasselet-Hector-Saralegi \cite{BHS}, but there are more detailed expositions in the monographs by 
Verona \cite{Ver} and Pflaum \cite{Pfl}.  We also refer the reader to \cite{Mather}, \cite{Hughes-Weinberger}, 
\cite{Banagl} and \cite{Kirwan-Woolf}. 
Unfortunately, definitions are not entirely consistent across those sources, so one of the purposes of 
reviewing this material is to specify the precise definitions used here. This section has another 
purpose, however, which is to prove the equivalence of this class of smoothly stratified pseudomanifolds
and of the class of manifolds with corners with iterated fibration structures, introduced by Melrose. 
The correspondence between elements in these two classes is by blowup (resolution) and blowdown, 
respectively (this was Melrose's motivation for formulating the notion of iterated fibration structure
in the first place).  We introduce the latter class in \S 2.2 and show that any manifold with corners with 
iterated fibration structure can be blown down to a smoothly stratified pseudomanifold. The converse, 
that any smoothly stratified pseudomanifold can be blown up, or resolved, to obtain a manifold with 
corners with iterated fibration structure, is proved in \S 2.3; this resolution was already defined
by Brasselet et al.\ \cite{BHS}, though those authors did not take note of the relevance of the fibration structures
on the boundaries of the resolution.  There is a subtlety in all of this regarding the proper definition of
isomorphism between these spaces. We discuss this and propose a suitable definition, which is
phrased in terms of this resolution, in \S 2.4. This alternate description of smoothly stratified pseudomanifolds
also helps to elucidate certain notions such as the natural classes of structure vector fields, metrics, etc.

\subsection{Smoothly stratified spaces}
\begin{definition}
A stratified space $X$ is a metrizable, locally compact, second countable space which admits a locally finite 
decomposition into a union of locally closed {\rm strata} $\frakS = \{Y_\alpha\}$, 
where each $Y_\alpha$ is a smooth (usually open) manifold, with dimension depending on the index $\alpha$. 
We assume the following:

\begin{itemize}
\item[i)] If $Y_\alpha, Y_\beta \in \frakS$ and $Y_\alpha \cap \overline{Y_\beta} \neq \emptyset$,
then $Y_\alpha \subset \overline{Y_\beta}$. 
\item[ii)] Each stratum $Y$ is endowed with a set of `control data' $T_Y$, $\pi_Y$ and $\rho_Y$;
here $T_Y$ is a neighbourhood of $Y$ in $X$ which retracts onto $Y$, $\pi_Y: T_Y \longrightarrow Y$ is 
a fixed continuous retraction and $\rho_Y: T_Y \to [0,2)$ is a proper `radial function' in this 
tubular neighbourhood such that $\rho_Y^{-1}(0) = Y$. Furthermore, we require that if $Z \in \frakS$ and 
$Z \cap T_Y \neq \emptyset$, then 
\[
(\pi_Y,\rho_Y): T_Y\cap Z \longrightarrow Y \times [0,2)
\]
is a proper differentiable submersion. 
\item[iii)] If $W,Y,Z \in \frakS$, and if $p \in T_Y \cap T_Z \cap W$ and $\pi_Z(p) \in T_Y \cap Z$,
then $\pi_Y(\pi_Z(p)) = \pi_Y(p)$ and $\rho_Y(\pi_Z(p)) = \rho_Y(p)$.
\item[iv)] If $Y,Z \in \frakS$, then
\begin{eqnarray*}
Y \cap \ovl{Z} \neq \emptyset & \Leftrightarrow & T_Y \cap Z \neq \emptyset, \\
T_Y \cap T_Z \neq \emptyset & \Leftrightarrow & Y \subset \ovl{Z}, \ Y=Z\ \  \mbox{or}\ Z \subset \ovl{Y}.
\end{eqnarray*}
\item[v)] For each $Y \in \frakS$, the restriction $\pi_Y: T_Y \to Y$ is a locally trivial fibration with
fibre the cone $C(L_Y)$ over some other stratified space $L_Y$ (called the link over $Y$), with atlas $\calU_Y = 
\{(\phi,\calU)\}$ where each $\phi$ is a trivialization $\pi_Y^{-1}(\calU) \to \calU \times C(L_Y)$, and the 
transition functions are stratified isomorphisms (as defined below) of $C(L_Y)$ which preserve the rays of 
each conic fibre as well as the radial variable $\rho_Y$ itself, hence are suspensions of isomorphisms of 
each link $L_Y$ which vary smoothly with the variable $y \in \calU$. 
\end{itemize}

If in addition we let $X_j$ be the union of all strata of dimensions less than or equal to $j$, and
require that 
\begin{itemize}
\item[vi)] $X = X_n \supseteq X_{n-1} = X_{n-2} \supseteq X_{n-3} \supseteq \ldots \supseteq X_0$ and
$X \setminus X_{n-2}$ is dense in $X$
\end{itemize}
then we say that $X$ is a stratified pseudomanifold. 
\end{definition}

Some of these conditions require elaboration: 

\medskip

$\bullet$ The depth of a stratum $Y$ is the largest integer $k$ such that there is a chain of strata
$Y = Y_k, \ldots, Y_0$ with $Y_j \subset \overline{Y_{j-1}}$ for $1 \leq j \leq k$. A stratum of maximal 
depth is always a closed manifold. The maximal depth of any stratum in $X$ is called the depth of $X$ as 
a stratified space. (Note that this is the opposite convention of depth from that in \cite{BHS}.)

We refer to the dense open stratum of a stratified pseudomanifold $\wh{X}$ as its regular set,
and the union of all other strata as the singular set,
\[
\mathrm{reg}(\wh{X}) := \wh{X}\setminus \mathrm{sing}(\wh{X}), \qquad \mathrm{where}\qquad
\mathrm{sing}(\wh{X}) = \bigcup_{{Y\in \mathfrak S}\atop{\mathrm{depth}\, Y > 0}} Y.
\]

\smallskip

$\bullet$ If $X$ and $X'$ are two stratified spaces, a stratified isomorphism between them is 
a homeomorphism $F: X \to X'$ which carries the open strata of $X$ to the open strata of $X'$ 
diffeomorphically, and such that $\pi_{F(Y))}^\prime \circ F = F \circ \pi_Y$, $\rho_Y^\prime 
= \rho_{F(Y)} \circ F$ for all $Y \in \frakS(X)$. (We shall discuss this in more detail below.)

\smallskip

$\bullet$ If $Z$ is any stratified space, then the cone over $Z$, denoted $C(Z)$, is the space $Z \times \RR^+$ 
with $Z \times \{0\}$ collapsed to a point. This is a new stratified space, with depth one greater than $Z$
itself. The vertex $0 := Z \times \{0\} / \sim $ is the only maximal depth stratum; $\pi_0$ is the natural
retraction onto the vertex and $\rho_0$ is the radial function of the cone. 

\smallskip

$\bullet$ There is a small generalization of the coning construction. For any $Y \in \mathfrak{S}$, 
let $S_Y = \rho_Y^{-1}(1)$. This is the total space of a fibration $\pi_Y: S_Y \to Y$ with fibre $L_Y$. 
Define the mapping cylinder over $S_Y$ by $\mathrm{Cyl}\,(S_Y,\pi_Y) = S_Y \times [0,2)\, /\sim$ where 
$(c,0) \sim (c',0)$ if $\pi_Y(c) = \pi_Y(c')$. The equivalence class of a point $(c,t)$ is sometimes 
denoted $[c,t]$, though we often just write $(c,t)$ for simplicity. Then there is a stratified isomorphism 
\[
F_Y: \mathrm{Cyl}\,(S_Y,\pi_Y)  \longrightarrow T_Y;
\]
this is defined in the canonical way on each local trivialization $\calU \times C(L_Y)$ and
since the transition maps in axiom v) respect this definition, $F_Y$ is well-defined.

\smallskip

$\bullet$ Finally, suppose that $Z$ is any other stratum of $X$ with $T_Y \cap Z \neq \emptyset$,
so by axiom iv), $Y \subset \bar{Z}$. Then $S_Y \cap Z$ is a stratum of $S_Y$.

\medskip

We have been very brief in this description since these axioms are described more carefully in
the references cited above. We do elaborate further on one point, however, which is the definition of 
stratified isomorphism given above. One problematic feature of this definition is that the condition
that such a map be a stratified isomorphism in this sense is very rigidly gauged by the control data 
on the domain and range, i.e.\ by the condition that $F$ preserve the product decomposition 
of each tubular neighbourhood. Because of this, it is nontrivial to prove that the same space $X$
with different sets of control data are isomorphic in this sense.

There are other even more rigid definitions of isomorphism in the literature. For example, the one in 
\cite{Pfl} requires that the spaces $X$ and $X'$ are differentiably embedded into some ambient Euclidean 
space, and that the map $F$ locally extends to a diffeomorphism of these ambient spaces. 
It is worth giving an example which indicates how rigid this last definition is. Let $X$ be a union
of three copies of the half-plane $\RR \times \RR^+$, as follows. The first and second are embedded as 
$\{(x,y,z): z=0, y \geq 0 \}$ and $\{(x,y,z): y=0, z \geq 0\}$, while the third is given by
$\{(x,y,z): y = r \cos \alpha(x), z = r \sin \alpha(x), r \geq 0\}$ where $\alpha: \RR \to (0,\pi/2)$
is smooth. In other words, this last sheet is obtained as the union of rays parallel to the $(y,z)$-plane 
which make an angle $\alpha(x)$ at each slice. Any condition requiring an isomorphism to extend to a 
diffeomorphism of the ambient $\RR^3$ would make many of the spaces obtained in this way inequivalent. 
Other more complicated phenomena arise if we let this third sheet become tangent to either of the
first two at some arbitrary closed set $x \in I$. 

At any rate, we contend that neither of these conditions is optimal, and that the precise
notion of a smooth stratified isomorphism, and hence the entire notion of a smoothly stratified space,
should be slightly relaxed from the first definition given above. With the definition we propose
below, the different subsets of $\RR^3$ described above arise in a perfectly legitimate way as
different embeddings of the same abstract smoothly stratified space (the space obtained by 
taking the product of a line with the union of three half-lines meeting at a common point). 
We return to all of this at the end of the section.

\subsection{Iterated fibration structures}
We now present the definition of an iterated fibration structure. This concept was formulated by Melrose in 
the late '90's as the correct boundary fibration structure in the sense of \cite{Melrose:Kyoto} precisely to describe 
the resolution of an iterated edge space, i.e. what we are calling a smoothly stratified space. This is necessary 
in order to apply the methodology of geometric microlocal analysis to develop a calculus of pseudodifferential 
iterated edge operators. Such a calculus, when it is eventually written down carefully, will yield direct proofs
of most of the analytic facts in later sections of this paper.  Since iterated fibration structures have not 
been discussed explicitly in the literature at this point, we present a brief outline here.  We are very grateful
to Richard Melrose for allowing us to describe this material here. The material here provides a necessary
initial step in the development of an iterated edge calculus from the point of view of geometric
microlocal analysis.

Let $\wt{X}$ be a manifold with corners up to codimension $k$. This means simply that any point $p \in \wt{X}$ 
has a neighbourhood $\calU \ni p$ which is diffeomorphic to a neighbourhood of the origin $\calV$ 
in the orthant $(\RR^+)^\ell \times \RR^{n-\ell}$, where $n = \dim \wt{X}$ and $p$ corresponds to the origin.
We can then use the induced local coordinates $(x_1,\ldots, x_\ell, y_1, \ldots, y_{n-\ell})$ where each 
$x_i \geq 0$ and $y_j \in (-\e,\e)$. There is an obvious decomposition of $X$ into its interior and
into its boundary faces of various codimensions. We make the additional global assumption that 
each face is an embedded manifold with corners in $\wt{X}$, or in other words, that no boundary face
intersects itself.

In the following, we shall be interested in fibrations $f: \wt{X} \to \wt{X}'$ between manifolds with corners. 
By definition, such a map $f$ is a fibration in this setting if it satisfies the following three properties:
$f$ is a `$b$-map', which means that if $\rho'$ is any boundary defining function in $\wt{X}'$, then $f^*(\rho')$ 
is a product of boundary defining functions of $\wt{X}$ multiplied by a smooth nonvanishing function;
next, each $q \in \wt{X}'$ has a neighbourhood $\calU$ such that $f^{-1}(\calU)$ is diffeomorphic to 
$\calU \times F$ where the fibre $F$ is again a manifold with corners; finally, we require that each fibre
$F$ be a `$p$-submanifold' in $\wt{X}$, which means that in terms of an appropriate adapted corner coordinate 
system $(x,y) \in (\RR^+)^\ell \times \RR^{n-\ell}$, as above, each $F$ is defined by setting some subset of 
these coordinates equal to $0$. 

The collection of boundary faces of codimension one play a special role, and is denoted $\calH = 
\{H_\alpha\}_{\alpha \in A}$ for some index set $A$. Each boundary face $G$ is the intersection of 
some collection of boundary hypersurfaces, $G = H_{\alpha_1} \cap \ldots \cap H_{\alpha_\ell}$,
which we often write as $H_{A'}$ where $A' = \{\alpha_1, \ldots, \alpha_\ell\} \subset A$. 

\begin{definition}[Melrose]
An iterated fibration structure on the manifold with corners $\wt{X}$ consists of the following data: 
\begin{itemize}
\item[a)] Each $H_\alpha$ is the total space of a fibration $f_\alpha: H_\alpha \to B_\alpha$,
where the fibre $F_\alpha$ and base $B_\alpha$ are themselves manifolds with corners.
\item[b)] If two boundary hypersurfaces meet, i.e.\ $H_{\alpha \beta} := H_\alpha \cap H_\beta \neq 
\emptyset$, then $\dim F_\alpha \neq \dim F_\beta$.
\item[c)] If $H_{\alpha \beta} \neq \emptyset$ as above, and $\dim F_\alpha < \dim F_\beta$, then the 
fibration of $H_\alpha$ restricts naturally to $H_{\alpha\beta}$ (i.e.\ the leaves of the fibration of 
$H_\alpha$ which intersect the corner lie entirely within the corner) to give a fibration of 
$H_{\alpha \beta}$ with fibres $F_\alpha$, whereas the larger fibres $F_\beta$ must be transverse to 
$H_\alpha$ at $H_{\alpha\beta}$. Writing $\del_\alpha F_\beta$ for the boundaries of these fibres at the 
corner, i.e.\ $\del_\alpha F_\beta := F_\beta \cap H_{\alpha\beta}$, then $H_{\alpha \beta}$ is also the total 
space of a fibration with  fibres $\del_\alpha F_\beta$. Finally, we assume that the fibres
$F_\alpha$ at this corner are all contained in the fibres $\del_\alpha F_\beta$, and in fact that
each fibre $\del_\alpha F_\beta$ is the total space of a fibration with fibres $F_\alpha$. 
\end{itemize}
\end{definition}

Because of condition a), if $H_{\alpha_1}, \ldots, H_{\alpha_r}$ intersect nontrivially at
the corner $H_{A'}$, $A' = \{\alpha_1, \ldots, \alpha_r\}$, then $A'$ inherits a strict ordering
from the dimensions of the corresponding fibres $F_{\alpha_j}$, and hence the entire index set
$A$ inherits a partial ordering, where the ordered chains $\alpha_1 < \ldots < \alpha_r$ in $A$ 
are in bijective correspondence with the corners $H_{\alpha_1} \cap \ldots \cap H_{\alpha_r}$.

The precise relationships between the various induced fibrations on each corner is intricate and difficult 
to state easily. Fortunately the details of these relationships are not important for the present considerations.
We do prove one fact about them which will be useful later.
\begin{lemma}
Suppose that $H_\alpha \cap H_\beta \neq \emptyset$ and $\alpha < \beta$. Then any boundary components of any of
the fibres $F_\alpha \subset H_\alpha$ are disjoint from the interior of $H_{\alpha \beta}$ and the image of the restriction of 
$f_\alpha$ to $H_{\alpha \beta}$ 
lies within a boundary component of $B_\alpha$, whereas the image of the restriction of $f_\beta$ to $H_{\alpha \beta}$ 
lies in the interior of $B_\beta$. In particular, if $\alpha$ and $\beta$ are, respectively, minimal and maximal
elements in $A$, then each $F_\alpha$ and the base $B_\beta$ are closed manifolds without boundary.
\label{le:ifs}
\end{lemma}
\begin{proof} It is perhaps easiest to see this in local coordinates. Near any $p$ in the interior of the corner 
$H_{\alpha\beta}$ we can choose adapted local coordinates $(x_\alpha,x_\beta,y_1, \ldots, y_{n-2})$ which simultaneously 
straighten out these fibrations, so each fibre $F_\alpha$ is given by $x_\alpha = 0$, $(x_\beta,y') = \mbox{const.}$ 
where $y = (y', y'')$ is some division of the variables on the corner, and each fibre $F_\beta$ is given by 
$x_\beta = 0$ and $(y',y_2'') = \mbox{const.}$ where $y'' = (y''_1,y''_2)$ is some further subdivision of the
$y''$ coordinates. Thus $y'$ and $(x_\alpha, y', y''_1)$ are local coordinate systems on $F_\alpha$ 
on $F_\beta$, respectively. (Note that this local coordinate description does not capture the fact that
$F_\alpha$ may have boundary components at other corners $H_{\gamma \alpha}$ where $\gamma < \alpha$.) 
From this we see that $(x_\beta,y'')$ and $y''_2$ are coordinates on $B_\alpha$ and $B_\beta$, respectively, 
which is equivalent to the first set of assertions; the final assertions of the lemma are direct consequences.  
\end{proof}

We can also introduce a notion of depth for the hypersurface faces $H_\alpha$; we say that $H_\alpha$
has depth $r$ if the longest chain of elements $H_\beta \in \calH$ for which $H_\alpha$ is the maximal element 
has length $r$. The depth of a manifold with corners with iteration fibration structure is the
maximum of the depth of any of its boundary hypersurfaces, and clearly this is the same as the
maximal codimension of any of its corners. 

In this setting there is a clear notion of equivalence: two spaces $\wt{X}$ and $\wt{X}'$ with iterated 
fibration structures are isomorphic precisely when there exists a diffeomorphism $\Phi$ between these 
manifolds with corners which preserves all of the fibration structures at all boundary faces. 

Unlike for smoothly stratified spaces in the last subsection, we have not included the notion of control 
data into this definition of iterated fibration structures. The reason is that the existence and uniqueness 
(up to isomorphism) of such control data follows directly from standard differential topology. Nonetheless, 
it will be useful to talk about control data in this setting, so we introduce it now.
\begin{definition}
Let $\wt{X}$ be a manifold with corners with an iterated fibration structure. Then control data for $\wt{X}$
consists of a set of triples $\{\tilde{T}_H,, \tilde{\pi}_H, \tilde{\rho}_H\}$, one for each $H \in \calH$, where
$\tilde{T}_H$ is a collar neighbourhood of the hypersurface $H$, $\tilde{\rho}_H$ is a defining 
function for $H$ and $\tilde{\pi}_H$ is a diffeomorphism from each slice $\tilde{\rho}_H = \mbox{const.}$
to $H$; in particular, the pair $(\tilde{\pi}_H,\tilde{\rho}_H)$ gives a diffeomorphism $\tilde{T}_H \to 
H \times [0,2)$, and hence determines an extension of the fibration of $H$ to all of $\tilde{T}_H$. 
This data is required to satisfy the following additional properties: for any hypersurface $H'$ which
intersects $H$ with $H' < H$, the restriction of $\tilde{\rho}_H$ to $H' \cap \tilde{T}_H$ is constant
on the fibres of $H'$; finally, near any corner $H_{A'}$, $A' = \{\alpha_1, \ldots, \alpha_r\}$, the extension
of the set of fibrations of $H_{A'}$ induced by the product decomposition 
\[
\left.(\tilde{\pi}_{H_{\alpha_j}}, \tilde{\rho}_{H_{\alpha_j}})\right|_{\alpha_j \in A'}: 
\bigcap_{j=1}^r \tilde{T}_{H_{\alpha_j}} \cong H_{A'} \times [0,2)^r
\]
preserves all incidence and inclusion relationships between the various fibres. 
\label{de:controldata}
\end{definition}

It is not hard to establish the existence of control data for an iterated fibration structure
on a manifold with corners $\wt{X}$. Indeed, we can successively choose the maps $\tilde{\pi}_H$
and defining functions $\tilde{\rho}_H$ in order of increasing depth, at each step making
sure to respect the compatibility relationships with all previous hypersurfaces. The uniqueness
up to diffeomorphism can be established in much the same way, based on the fact that there is a unique
product decomposition of a collar neighbourhood of any $H$ up to diffeomorphism.

Finally, note that if $\wt{X}$ has an iterated fibration structure, then any hypersurface face, or
corner of arbitrary codimension, inherits such a structure too (if we forget about the fibration of its
interior), with depth equal to $k$ minus its codimension.

\begin{proposition} If $\wt{X}$ is a compact manifold with corners with an iterated fibration structure, 
then there is a smoothly stratified space $\wh{X}$ obtained from $\wt{X}$ by a process of successively 
blowing down the fibres of each hypersurface boundary of $\wt{X}$ in order of increasing fibre 
dimension (or equivalently, of increasing depth). The corresponding blowdown map will be
denoted $\beta: \wt{X} \to \wh{X}$. 
\label{pr:blowdown}
\end{proposition}
\begin{proof}
Let us warm up to the general case by supposing first that $\wt{X}$ is a manifold with boundary, so $\del \wt{X}$ 
is the total space of a fibration with fibre $F$ and base space $Y$ and both $F$ and $Y$ are closed manifolds. 
Choose a (suitably scaled) boundary defining function $\rho$ and fix a product decomposition $\del \wt{X} \times 
[0,2)$ of the collar neighbourhood $\calU = \{\rho < 2\}$. This defines a retraction $\tilde{\pi}: \calU \to \del X$, 
as well as a fibration of $\calU$ over $\del \tilde{X}$ with fibre $\tilde{\pi}^{-1}(F) = F \times [0,2)$. 
Now collapse each fibre $F$ at $x=0$ to a point. This commutes with the restriction to each $F \times [0,2)$, 
so we obtain a bundle of cones $C(F)$ over $Y$. We call this space the blowdown of $\tilde{X}$ along the fibration, 
and write it as $X/F$. Denote by $T_Y$ the image of $\calU$ under this blowdown. 
The map $\tilde{\pi}$ induces a retraction map $\pi(\calU) = T_Y \to Y$, and $\rho$ also descends to $T_Y$. 
Thus $\{T_Y, \pi,\rho\}$ are the control data for the singular stratum $Y$, and it is easy to check that 
these satisfy all of the axioms in \S 2.1, hence $X/F$ is a smoothly stratified space. 

The proof in general follows an inductive scheme, but as an alternative approach to help the reader
gain intuition, we present an independent explanation of the case where $X$ has corners of codimension $2$.
Consider any corner $H_{12} = H_1 \cap H_2$, where $H_1 < H_2$, i.e.\ $\dim F_1 < \dim F_2$ and property c) 
is satisfied; note that $H_{12}$ is a closed manifold without boundary. From property c) and Lemma \ref{le:ifs}
we see that each $F_1$ is also a closed manifold without boundary. Fix  a defining function $\tilde{\rho}_1$ for $H_1$ 
and retraction $\tilde{\pi}_1$ as in the previous case and consider the blowdown 
$X/F_1$ obtained by collapsing the fibres $F_1$ to points. The image of $\{\tilde{\rho}_1 < 2\}$ in this blowdown is a 
cone bundle over the base $Y_1'$ of the fibration of $H_1$, and the link of each cone in this bundle is $F_1$. 
Note that $Y_1'$ is a manifold with boundary. The space $X/F_1$ has a codimension one boundary $\del (X/F_1)$, 
which is the image of $H_2$ under this blowdown, and the singular set of this boundary is precisely $\del Y_1'$,
so strictly speaking, we have left the class of pseudomanifolds (or, if you like, allowed pseudomanifolds with boundary). 
Choose a boundary defining function $\tilde{\rho}_2$ for $\del(X/F_1)$ and a retraction 
$\tilde{\pi}_2$ from $\{\tilde{\rho}_2 < 2\}$ in $X/F_1$ onto $\del (X/F_1)$. Now blow down along the fibres $F_2$ to 
obtain the space $(X/F_1)/F_2$; this is our smoothly stratified space $\wh{X}$. The singular strata are $Y_2$, 
the image of $H_2/F_2$ in this final blowdown, and $Y_1$, the image of $Y_1'$ when the fibres $\mbox{sing}\, (F_2)$ 
in its boundary are blown down. The radial functions $\rho_1$, $\rho_2$, the tubular neighbourhoods $T_{1}$, $T_{2}$ 
and the retractions $\pi_1$, $\pi_2$ are obtained from the corresponding data in $X$ and $X/F_1$ in an obvious way. 

We now turn to the general case, which is proved by induction on the depth. Similarly to the argument in the 
next subsection where we show how to blow up a smoothly stratified space, which is adapted from \cite{BHS}, 
we shall use a `doubling construction' to stay within the class of stratified pseudomanifolds while
applying the inductive hypothesis to reduce the complexity of the problem. More specifically, given 
$\wt{X}$, a manifold with corners with iterated fibration structure of depth $k$, 
we form a new manifold with corners and iterated fibration structure of depth $k-1$ by simultaneously 
doubling $\wt{X}$ across all of its maximal depth hypersurfaces. In other words, consider
\[
\wt{X}' = \left((\wt{X} \times {-1}) \sqcup (\wt{X} \times {+1}) \right)/\sim
\]
where $(p,-1) \sim (q,+1)$ if and only if $p = q \in H \in \calH$ where $\mbox{depth}\,(H) = k$. 
By the usual sorts of arguments in differential topology, one can give $\wt{X}'$ the structure 
of a manifold with corners up to codimension $k-1$. If $H_j \in \calH$ is any face with depth $j < k$
which intersects a face $H_k$ of depth $k$, then as in Lemma \ref{le:ifs}, the boundaries of the fibres 
$F_j \subset H_j$ only meet the corners $H_i \cap H_j$ for $i < j$, and do not meet the interior of $H_j \cap H_k$.
In terms of the local coordinate description from that Lemma, where $F_j$ is given by $x_j = 0$, $(x_2,y'') = 
\mbox{const.}$, so they can be continued smoothly to the other component of this double since this corresponds
to letting $x_k$ vary in $(-\e,\e)$ rather than just $[0,\e)$. 

The dimensional comparisons and inclusion relations at all other corners remain unchanged.
Therefore, $\wt{X}'$ has an iterated fibration structure. This new space also carries a smooth involution
which fixes the union of all depth $k$ faces, where the two copies of $\wt{X}$ are joined, as 
well as a function $\rho_k$ which is positive on one copy of $\wt{X}$, negative on the other, and which 
vanishes simply on the interface between the two copies of $\wt{X}$. For simplicity of exposition, we now
assume that there is only one depth $k$ face, $H_k$. We can also choose $\tilde{\rho}_k$ so that 
it is constant on the fibres of all other boundary faces, and a retraction $\tilde{\pi}_k$ defined on
the set $|\tilde{\rho}_k| < 2$ onto $H_k$. 

Now apply the inductive hypothesis to blow down the boundary hypersurfaces of $\wt{X}'$ in order of 
increasing fibre dimension to obtain a smoothly stratified space $\wh{X}'$ of depth $k-1$. The 
function $\rho_k$ descends to a function (which we give the same name) on this space. Consider
the open set $\wh{X}^+ := \wh{X}' \cap \{\rho_k > 0\}$, and also $\del_k \wh{X} := 
\wh{X}' \cap \{\rho_k = 0\}$. Both of these are smoothly stratified spaces; for the former this
is because (in the language of \cite{BHS}) we are restricting to a `saturated' open set of $\wh{X}'$,
though we do not need to appeal to this terminology since the assertion is clear, whereas for the latter 
it follows by induction since it is the blowdown of $H_k$, which has depth less than $k$. This space
$\del_k \wh{X}$, which we may as well denote by $\wh{H_k}$ is the total space of a fibration induced
from the fibration of the face $H_k$ in $\wt{X}$. In fact, by Lemma \ref{le:ifs}, since the $H_k$ are
maximal, the base $B_k$ has no boundary, and the fibres $\wt{F}_k$ are manifolds with corners with 
iterated fibration structures of depth less than $k$. 

Hence after the blowdown, the base of the fibration of $\del_k \wh{X}$ is 
still $B_k$ while the fibres are the blowdowns $\wh{F}_k$ of the spaces $\wt{F}_k$, which are again well 
defined by induction. Finally, using the product decomposition of a neighbourhood of $H_k$ in $\wt{X}$, we can 
identify the space from this neighbourhood by collapsing the fibres of $H_k$ with the mapping cylinder
for the fibration of $\del_k \wh{X}$. This produces the final space $\wh{X}$. 

It suffices to check that the stratification of $\wh{X}$ satisfies the axioms of a smoothly stratified space 
only near where this final blowdown takes place, since the inductive hypothesis guarantees that they hold
elsewhere. These axioms are not difficult to verify from the local description of $\wt{X}$ in a product 
neighbourhood of $H_k$.
\end{proof}

\subsection{The resolution of a smoothly stratified space}
We complete our description of the differential topology of smoothly stratified
spaces by showing that, conversely to the construction of the previous subsection, if $\widehat{X}$
is any smoothly stratified space, hence satisfies all the properties listed in \S 2.1, then one may
resolve its singularities by successively blowing up the strata in order of decreasing depth to obtain
a manifold with corners with iterated fibration structure. Combined with Proposition \ref{pr:blowdown},
this proves that there is a bijective correspondence between the elements of the class of compact manifolds 
with corners with iterated fibration structures and the elements of the class of compact smoothly stratified 
spaces.

\begin{proposition}\label{prop:BlowUp}
Let $\wh{X}$ be a smoothly stratified pseudomanifold. Then there exists a manifold with corners $\wt{X}$ 
with an iterated fibration structure, and a blowdown map $\beta: \widetilde{X} \to \widehat{X}$
which satisfies the following properties:
\begin{itemize}
\item there is a bijective correspondence $Y \leftrightarrow \wt{X}_Y$ between the strata $Y \in \mathfrak S$ 
of $\wh{X}$ and the boundary hypersurfaces of $\wt{X}$;
\item $\beta$ is a diffeomorphism between the interior of $\wt{X}$ and the regular set of $\wh{X}$;
we denote by $X$ this open set, which is dense in either $\wt{X}$ or $\wh{X}$; 
\item $\beta$ is also a smooth fibration of the interior of each boundary hypersurface $\wt{X}_Y$ 
with base the corresponding stratum $Y$ and fibre the regular part of the link of $Y$ in $\wh{X}$;
moreover, there is a compactification of $Y$ as a manifold with corners $\wt{Y}$ such
that the extension of $\beta$ to all of $\wt{X}_Y$ is a fibration with base $\wt{Y}$ and fibre
$\wt{L_Y}$; finally, each fibre $\wt{L_Y} \subset \wt{X}_Y$ is a manifold with corners with iterated 
fibration structure and the restriction of $\beta$ to it is the blowdown onto the smoothly stratified space
$\bar{Y}$. 
\end{itemize}
\end{proposition}

We shall give a fairly detailed sketch of the proof of this result in the remainder of this subsection,
adapting the construction of the ``deplissage'' from \cite{BHS}. The proof is inductive and the key 
point is to show that if $\wh{X}$ has depth $k$ and we simultaneously blow up the union of the depth 
$k$ strata to obtain a space $\wh{X}_1$, then all the control data of the stratification on $\wh{X}$ 
lifts to $\wh{X}_1$ to give $\wh{X}_1$ the structure of a smoothly stratified space of depth $k-1$. 
Iterating this $k$ times completes the proof.

Actually, this last paragraph is slightly inaccurate. We wish to stay in the category of
smoothly stratified pseudomanifolds, which by definition do not have codimension one boundaries,
so we proceed just as in the proof of Proposition \ref{pr:blowdown} (and as in \cite{BHS}) 
and instead construct a space $\wh{X}_1'$ which is a double across the boundary hypersurface of 
the blowup of $\wh{X}$ along its depth $k$ strata, and show that $\wh{X}'_1$ is a smoothly stratified 
pseudomanifold of depth $k-1$. However, $\wh{X}'_1$ comes equipped with an involution $\tau_1$ 
which interchanges the two copies of the double; the actual blowup is the closure of one component 
of the complement of the fixed point set of this involution. After iterating this $k$ times, we obtain 
a smooth compact manifold $\wh{X}'_k$ which is equipped with $k$ commuting involutions $\tau_j$, 
$j = 1, \ldots, k$, which are independent of one another, and the manifold with corners we seek is 
the closure of one of the $2^k$ components of the complement of the union of fixed point
sets for all of these involutions. 

\begin{proof}
To begin, then, fix a stratum $Y$ which has maximal depth $k$. Then $Y$ is a smooth closed manifold. Recall 
the notation from \S 2.1, and in particular the stratified isomorphism $F_Y$ from the mapping cylinder of 
$(S_Y,\pi_Y)$ to $T_Y$ and the family of local trivializations $\phi:\pi_Y^{-1}(\calU) \to \calU \times C(L_Y)$ 
for suitable $\calU \subset Y$. If $u \in T_Y \cap \pi_Y^{-1}(\calU)$, we write $\phi(u) = (y,z,t)$ where 
$y \in \calU$, $z \in L_Y$ and $t = \rho_Y(u)$; in particular, by axiom v), there is a retraction $R_Y: 
T_Y \setminus Y \to S_Y$, given on any local trivialization by $(y,z,t) \to (y,z,1)$ 
(which is well defined since $t \neq 0$). 

To construct the first blowup, assume for simplicity that there is only one stratum $Y$ of 
maximal depth $k$. Define
\begin{equation}
\wt{X}_1' := \left((\wh{X} \setminus Y) \times \{-1\}\right) \sqcup \left((\wh{X} \setminus Y) \times \{+1\}\right)
\sqcup \big(S_Y \times (-2,2) \big) / \sim
\label{first-blow-up}
\end{equation}
where (if $\e = \pm 1$), 
\begin{equation}
(p, \e) \sim (R_Y(p),\rho_Y(p)) \quad\text{if}\quad p \in T_Y\setminus Y \  \text{and}\ \e t  >0.
\label{eq-relation}
 \end{equation}
For convenience, let $\wh{X}' = (\wh{X} \times \{-1\}) \sqcup (\wh{X} \times \{+1\}) / \sim$ where
$(u,\e) \sim (u',\e')$ if and only if $u=u' \in Y$.  Note that $\wt{X}_1' \setminus S_Y 
\times\{0\}$ is naturally identified with $\wh{X}' \setminus Y$, so this construction replaces $Y$ with $S_Y$. 

There is a blowdown map $\beta_1: \wt{X}_1' \to \wh{X}'$ given by
\[
\beta_1(u,\e) = (u,\e) \quad \text{if}\quad u \notin Y, \qquad
\beta_1(u,0) = \pi_Y(u). 
\]
Clearly $\beta_1: \wt{X}_1'  \setminus S_Y \times \{0\} \to \wh{X}' \setminus Y$ is an isomorphism
of smoothly stratified spaces and $(S_Y \times (-2,2))$ is a tubular neighbourhood of $(\beta_1)^{-1}Y= 
S_Y\times \{0\}$ in $\wt{X}_1'$.

We shall prove that $\wt{X}_1'$ is a smoothly stratified space of depth $k-1$ equipped with an involution
$\tau_1$ which fixes $S_Y \times \{0\}$ and interchanges the two components of the complement of this
set in $\wt{X}_1'$, and which fixes all the control data of $\wt{X}_1'$. 
To do all of this, we must fix a stratification ${\mathfrak S}_1$ of $\wt{X}_1'$ and define all of
the corresponding control data and show that these satisfy properties i) - vi). 

\medskip

$\bullet$ Fix any stratum $Z\in \mathfrak{S}$ of $\wh{X}$ with $\text{depth}\,(Z) < k$, and define
\begin{equation}
\wt{Z}_1' := (Z \times \{\pm 1\}) \sqcup \left( (S_Y \cap Z)\times (-2,2)\right)/ \sim, 
\label{new-stratum}
\end{equation}
where $\sim$ is the same equivalence relation as in \eqref{eq-relation}. The easiest way to see that this is 
well-defined is to note that $S_Y \cap Z$ is a stratum of the smoothly stratified space $S_Y$ and that 
the restriction 
\begin{equation}
F_Y : \mathrm{Cyl}\,(S_Y \cap \bar{Z}, \pi_Y) \longrightarrow \bar{Z} \cap T_Y
\label{eq:cylstrat}
\end{equation}
is an isomorphism. (This latter assertion follows from axiom ii).)

As above, let $Z'$ be the union of two copies of $\bar{Z}$ joined along $\bar{Z} \cap Y$. 

\smallskip

$\bullet$ Now define the stratification ${\mathfrak S}_1$ of $\wt{X}_1'$
\begin{equation}\label{stratification-w-b}
{\mathfrak S}_1 := \{\wt{Z}_1': Z \in \mathfrak{S}\setminus Y\}. 
\end{equation}
We must now define the control data $\{ T_{\wt{Z}_1'}, \pi_{\wt{Z}_1'}, \rho_{\wt{Z}_1'}\}_{\wt{Z}_1'
\in \mathfrak{S}_1}$ associated to this stratification. 

\smallskip

$\bullet$ Following \eqref{new-stratum}, set
\begin{equation}\label{tube-for-z}
T_{\wt{Z}'_1}:= T_Z \times \{\pm 1\} \sqcup \left((S_Y\cap T_Z)\times (-2,2)\right)/\sim,
\end{equation}
where $(p,\e) \sim (c,t)$ if $t\e >0$ and $p=F_Y (c,|t|)$. Extending (or `thickening') (\ref{eq:cylstrat}),
by axiom iii) we also have that $F_Y$ restricts to an isomorphism between $\mbox{Cyl}\,(T_Z \cap S_Y ,\pi_Y) $ 
and $T_Z \cap T_Y$. In turn, using axiom ii) again, within the smoothly stratified space $S_Y$, $F_{T_Y \cap Z}$ 
is an isomorphism from $\mbox{Cyl}\,(S_Y \cap S_Z, \pi_Z)$ to the tubular neighbourhood of $Z \cap S_Y$ in $S_Y$, 
which is the same as $T_{S_Y \cap Z}$. Using these representations, the fact that (\ref{tube-for-z}) is well-defined 
follows just as before. 

Note that $Y$ has been stretched out into $S_Y \times \{0\}$, and $T_{\wt{Z}_1'} \cap (S_Y \times \{0\})$ is
isomorphic to $T_{\wt{Z}_1'} \cap (S_Y \times \{t\})$ for any $t \in (-2,2)$. 

\smallskip

$\bullet$ The projection $\pi_{\wt{Z}_1'}$ is determined by $\pi_Z$ on each slice $(S_Y \cap T_Z) 
\times \{t\}$, at least when $t \neq 0$, and extends uniquely by continuity to the slice at $t=0$
in $\wt{X}_1'$. A similar consideration yields the function $\rho_{\wt{Z}_1'}$. 

\smallskip

$\bullet$ One must check that the space $\wt{X}_1'$ and this control data for its stratification
satisfies axioms i) - vi). This is somewhat lengthy but straightforward, so details are left
to the reader. 

\smallskip

$\bullet$ Finally, this whole construction is symmetric with respect to the reflection $\tau_1$
defined by $t \mapsto -t$ in $T_Y$ and which extends outside of $T_Y$ as the interchange of the 
two components of $X' \setminus Y$. The fixed point set of $\tau_1$ is the slice $S_Y \times \{0\}$.

\medskip

This establishes that the space $\wt{X}_1'$ obtained by `resolving' the depth $k$ smoothly stratified
space $\wh{X}$ along its maximal depth strata via this doubling-blowup construction is a smoothly
stratified space of depth $k-1$, equipped with one extra piece of data, the involution $\tau_1$. 

This process can now be iterated. After $j$ iterations we obtain a smoothly stratified
space $\wt{X}_j'$ of depth $k-j$ which is equipped with $j$ commuting involutions $\tau_i$,
$1 \leq i \leq j$. In particular, the space $\wh{X}_k'$ is a compact closed manifold. 

We next check that these involutions are `independent' in the sense that for any point $p$ which 
lies in the fixed point set of more than one of the $\tau_i$, the $-1$ eigenspaces of the $d\tau_i$ are 
independent. Suppose that this assertion is true for all spaces of depth less than $k$; then it is true 
in particular for the resolution of the space $\wh{X}_1'$. Denote by $H_k$ the fixed point set of $\tau_1$; 
this is the boundary face of the closure of $\wh{X}_1' \setminus H_k$. Now consider any collection
of involutions $\tau_{i_1}, \ldots, \tau_{i_r}$ with fixed point sets $H_{i_1}, \ldots, H_{i_r}$
intersecting at a corner $G$ such that $G \cap H_k \neq \emptyset$. It suffices to verify that
this intersection is transversal. However, this follows from the fact that $\tau_1$ restricts to any
of the $H_{i_s}$ as a nontrivial involution.

The complement of the union of fixed point sets of the involutions $\tau_i$ is a union of $2^k$ components,
and our resolved space $\wt{X}$ is the closure of any one of these components. 

To conclude the construction, we must show that $\wt{X}$ carries the structure of a manifold
with corners with iterated fibration structure. In the argument above, we proved that $\wt{X}$ has
the local structure of a manifold with corners already, but it remains to check that the boundary
faces are embedded. For this, first note that all faces of the resolution of $\wh{X}_1'$ are embedded, 
and by its description in the resolution construction, $H_k$ is as well; finally, all corners of $\wt{X}$ 
which lie in $H_k$ are embedded since they are faces of the resolution of $S_Y$ where $Y$ is the maximal 
depth stratum and we may apply the inductive hypothesis. This proves that $\wt{X}$ is a manifold with corners. 

Now let us examine the structure on the boundary faces.  We proceed once again by induction. The case $k=1$ 
is obvious since then $\wt{X}$ is a manifold with boundary where the boundary is the total space 
of a fibration and there are no compatibility conditions
with other faces. Now, suppose we have proved the assertion for all spaces of depth less
than $k$, and let $X$ be a smoothly stratified space with depth $k$. Let $Y$ be the union
of all strata of depth $k$ and consider the doubled-blowup space $\wt{X}_1'$. This is
a stratified space of depth $k-1$, so its resolution is a manifold with corners up to
codimension $k-1$ with iterated fibration structure. Note that since $S_Y$ is again a
smoothly stratified space of depth $k-1$, its resolution $\wt{S_Y}$ is also a manifold with 
corners with iterated fibration structure. The blowdown of $\wt{S_Y}$ along the fibres
of all of its boundary hypersurfaces is a smoothly stratified space $\wh{S_Y}$ and
this is the boundary $H_k$ of $\wt{X}_1$, the `upper half' of $\wt{X}_1'$. 

Once we perform all the other blowups, we know that the compatibility conditions are satisfied 
at every corner except those which lie in $\wt{S_Y}$. The images of the other boundaries of 
$\wt{X}_1$ by blowdown into $\wh{X}_1$ are precisely the singular strata of this space. 
Furthermore, there is a neighbourhood of $H_k'$ in $\wt{X}_1$ of the form $S_Y \times [0,2)$ (using 
the variable $t$ in this initial blowup as the defining function $\rho_k$), so near $H_k$ 
$\wt{X}$ has the product decomposition $\wt{S_Y} \times [0,2)$. From this it follows that each
fibre $F_j$ of $H_j$, $j < k$, lies in the corresponding corners $H_k \cap H_j$; it also follows
that each fibre $F_k$ of $H_k$ is transverse to this corner, and has boundary $\del_j F_k$ equal 
to a union of the fibres $F_j$. This proves that all conditions a) - c) of the iterated fibration 
structure are satisfied. 
\end{proof}

\subsection{Smoothly stratified isomorphisms}

As promised earlier in this section, we return to a closer discussion of the proper definition of isomorphism 
between smoothly stratified spaces. Our point of view, following Melrose, is that these isomorphisms are
better understood through their lifts to the resolutions. 

To begin, let us state a result which is a straightforward consequence of the resolution and blowdown
constructions of the previous two subsections and their proofs.
\begin{proposition} Let $\wh{X}$ and $\wh{X}'$ be two smoothly stratified spaces and $\wt{X}$, $\wt{X}'$
their resolutions, with blowdown maps $\beta: \wt{X} \to \wh{X}$ and $\beta': \wt{X}' \to \wh{X}'$.
Suppose that $\hat f : \wh{X}\to \wh{X}'$ is a stratified isomorphism. Then there is a unique 
diffeomorphism of manifolds with corners $\tilde{f}: \wt{X}\to  \wt{X}'$ which preserves the iterated 
fibration structures and which satisfies $\hat{f}\circ \beta=\beta' \circ \tilde{f}$.
\end{proposition}
The existence of the lift is proved already in \cite{BHS}, \S 2 Prop. 3.2 and Remark 4.2, 
though of course they do not consider the issue of whether it preserves the fibrations at the boundaries.

The converse result is also true
\begin{proposition} Given $\wt{X}$, $\wt{X}'$, $\wh{X}$ and $\wh{X}'$, as above, if $\tilde{f}: 
\wt{X} \to \wt{X}'$ is a diffeomorphism of manifolds with corners which preserves the fibration
structures at the boundaries, then there exists some choice of control data on the blown down
spaces and a smoothly stratified isomorphism $\hat{f}: \wh{X} \to \wh{X}'$ such that
$\hat{f}\circ \beta=\beta' \circ \tilde{f}$.
\label{pr:diffeosss}
\end{proposition}
The subtlety here is that we must use the control data of the blown down space which is
induced from any choice of `control data' of $\wt{X}$ and the pushforward of this control
data via $\tilde{f}$ on $\wt{X}'$, cf.\ Definition \ref{de:controldata}. 

There is also a third result which completes the picture.
\begin{proposition}
Let $\wt{X}$ be a manifold with corners with iterated fibration structure, and suppose that
$\{\tilde{\pi}_H,\tilde{\rho}_H\}$ and $\{\tilde{\pi}_H',\tilde{\rho}_H'\}$ are two sets
of control data on it. Then there is a diffeomorphism $\tilde{f}$ of $\wt{X}$ which preserves the 
iterated fibration structure, and which intertwines the two sets of control data.
\end{proposition}
Combined with Proposition~\ref{pr:diffeosss}, this reproves the result that any two sets of
control data on a smoothly stratified space $\wh{X}$ are equivalent by a smoothly stratified
isomorphism.

Since it is not central to our main theme, we shall not prove this last result in detail.

As with everything else in this section, this should be done inductively, and the key idea
is that we can pull back any set of control data on $\wt{X}$ to a `universal' set of control
data defined on the union of the inward pointing normal bundles to each boundary hypersurface.

In any case, this discussion should make clear that the `correct' definition of a stratified
isomorphism $\hat{f}$ between smoothly stratified spaces is that the lift of $\hat{f}$ to
the associated resolutions is a diffeomorphism of the corresponding manifolds with
corners which preserves the iterated fibration structures. This has the advantage that
it is not inherently inductive (even though many of the arguments behind it are), and
provides a clear notion of the regularity of these isomorphisms on approach to the
singular set.

\section{Iterated edge metrics} \label{section:iterated}

We now introduce the class of Riemannian metric on smoothly stratified spaces with which we shall
work thoughout this paper. A priori, these metrics are only defined on $\mbox{reg}\,(\wh{X})$,
but the main point is their behaviour near the singular strata. These metrics were also considered
by Cheeger \cite{Ch} and also by Brasselet-Legrand \cite{BL}; they are most easily
described using adapted coordinate charts (see pp.\, 224-5 of \cite{BL}) or equivalently, on
the resolution $\wt{X}$. In the following, we freely use notation from the last section.

We begin by constructing an open covering of $\mbox{reg}\,(\wh{X})$ by sets with an iterated conic structure. 
Let $Y_1$ be any stratum. By definition, for each $q_1 \in Y_1$ there exists a neighbourhood $\calU_1$ and a 
trivialization $\pi_{Y_1}^{-1}(\calU_1) \cong \calU_1 \times C(L_{Y_1})$. Now fix any stratum $Y_2 \subset L_{Y_1}$,
and a point $q_2 \in Y_2$. As before, there is a neighbourhood $\calU_2 \subset Y_2$ and a trivialization 
$\pi_{Y_2}^{-1}(\calU_2) \cong \calU_2 \times C(L_{Y_2})$. Continuing on in this way, the process must stop in
no more than $d = \mbox{depth}\,(Y_1)$ steps when $q_s$ lies in a stratum $Y_s$ of depth $0$ in $L_{Y_{s-1}}$ (
which must, in particular, occur when $L_{Y_{s-1}}$ itself has depth $0$). We obtain in this way an open set of the form 
\begin{equation}\label{eq:chart}
	\calU_1 \times C\big( \calU_{2} \times C( \calU_{3} \times \ldots \times C(\calU_s)\,) \cdots \big),
\end{equation}
where $s \leq d$, which we denote  by $\calW = \calW_{q_1, \ldots, q_s}$. Choose a local coordinate system 
$y^{(j)}$ on  $\calU_j$, and let $r_j$ be the radial coordinate in the cone $C(L_{Y_j})$. Thus $(y^{(1)},r_1, y^{(2)},r_{2}, \ldots, y^{(s)})$ 
is a full set of coordinates in $\calW$. Clearly we may cover all of $\wh{X}$ by a finite number of sets of this form.

We next describe the class of admissible Riemannian metrics on $\mbox{reg}\,(\wh{X})$ by giving their
structure on each set of this type. 

\begin{definition}\label{def:metric}
We say that a Riemannian metric $g$ defined on $\mbox{reg}\,(\wh{X})$ is an iterated edge metric if there is a covering 
by the interiors of sets of the form $\calW_{q_1, \ldots, q_s}$ so that in each such set, 
\[
g= h_{1} +  d r_1^2 + r_{1}^2 ( h_{2} + d r_{2}^2 +  r_{2}^2 ( h_{3} +  d r_{3}^2 +  r_{3}^2 (  h_{4} +\ldots + r_{s-1}^2 h_s) ) ),
\]
with $ 0< r_j < \epsilon$ for some $\epsilon >0$ and every $j$, and where $h_j$ is a metric on $\calU_j$.  We also assume that 
for every $j = 1, \ldots, s$,  $h_j$ depends only on $y^{(1)}, r_1, y^{(2)} , r_{2}, \ldots, y^{(j)}, r_{j}$. 

If each $h_j$ is independent of the radial coordinates $r_{1}, \ldots, r_j$,  then we call $g$ a rigid iterated edge metric.
Note that this requires the choice of a horizontal lift of the tangent space of each stratum $Y$ as a subbundle of
the cone bundle $T_Y$ which is invariant under the scaling action of the radial variable on each conic fibre.
\end{definition}

Recall the manifold with corners $\wt{X}$ with iterated fibration structure which 
is the resolution of $\wh{X}$. Its interior is canonically identified with $\mbox{reg}\,(\wh{X})$, and we identify 
these spaces without comment. Each boundary hypersurface $H$ of $\wt{X}$ has a global defining function $x_H$,
so $H = \{x_H = 0\}$;  we now define the total boundary defining function
\[
\rho = \prod_{H \in \calH} x_H.
\]

\begin{definition}
Let $g$ be an admissible iterated edge metric on $\mbox{reg}\,(\wh{X})$. The associated
complete iterated edge metric $\wt{g}$ is, by definition, 
$$
\wt{g} = \rho^{-2} g.
$$
\end{definition}

\begin{proposition} \label{thm:metric} 
Let $\wh{X}$ be a smoothly stratified pseudomanifold. Then there exists a rigid iterated edge metric $g$ on 
$\mbox{reg}\,(\wh{X})$. 
\end{proposition}

\begin{proof} 
We prove this by induction. For spaces of depth $0$, there is nothing to prove, so suppose that $\wh{X}$ is a 
smoothly stratified space of depth $k \geq 1$, and assume that the result is true for all spaces with depth less than $k$.

Let $Y$ be the union of  strata  of depth $k$, each component of which is necessarily a closed manifold; for convenience 
we assume that $Y$ is connected. Consider the space $\wt{X}_1'$ obtained in the first step of the resolution
process in \S 2.3 by adjoining two copies of $\wh{X}$ along $Y$ and replacing the double of the neighbourhood
$T_Y$ by a cylinder $S_Y \times (-2,2)$. This is a space of depth $k-1$, and hence possesses a rigid
iterated edge metric $g_1$. We may in fact assume that in the cylindrical region $S_Y \times (-2,2)$,
$g_1$ has the form $dt^2 + g_{S_Y}$, where $g_{S_Y}$ is a (rigid) iterated edge metric on $S_Y$ which is
independent of $t$. Recalling that $S_Y$ is the total space of a fibration with fibre $L_Y$, we can define
a family of metrics $g_{S_Y}^r$ on $S_Y$ by scaling the metric on each fibre by the factor $r^2$. Then
$g = dr^2 + g_{S_Y}^r$ induces an admissible iterated edge metric on $\wh{X}$, which by construction
is also rigid. 
\end{proof}

\begin{proposition}\label{prop:homotopymet} 1) Any two admissible iterated edge metrics on $\wh{X}$ are homotopic 
within the class of admissible iterated edge metrics. 2) Any two rigid iterated edge metrics on $\wh{X}$ are homotopic 
within the class of rigid iterated edge metrics. 
\end{proposition}

\begin{proof}
We proceed by induction.  The result is obvious when the depth is $0$, so assume
it holds for all spaces of depth strictly less than $k$ and consider a pseudomanifold  of depth $k$ with
two admissible iterated edge metrics $g$ and $g^\prime$.  

To begin, then, fix a stratum $Y$ which has maximal depth $k$. Then $Y$ is a smooth closed manifold. Recall 
the notation from \S 2.1, and in particular the stratified isomorphism $F_Y$ from the mapping cylinder of 
$(S_Y,\pi_Y)$ to $T_Y$ and the family of local trivializations $\phi:\pi_Y^{-1}(\calU) \to \calU \times C(L_Y)$ 
for suitable $\calU \subset Y$. If $u \in T_Y \cap \pi_Y^{-1}(\calU)$, we write $\phi(u) = (y,z,t)$ where 
$y \in \calU$, $z \in L_Y$ and $t = \rho_Y(u)$; in particular, by axiom v), there is a retraction $R_Y: 
T_Y \setminus Y \to S_Y$, given on any local trivialization by $(y,z,t) \to (y,z,1)$ 
(which is well defined since $t \neq 0$). 

In any of these trivializations, the metric $g$ has the form
\begin{equation*}
	(\phi^{-1})^*g
	= g_{\calU}(y,t) + dt^2 + t^2 g_{L_Y}(t,y,z)
\end{equation*}
and the homotopy 
\begin{equation*}
	s \mapsto 
	g_{\calU}(y,s + (1-s)t) + dt^2 + t^2 g_{L_Y}(s+(1-s)t,y,z)
\end{equation*}
removes the dependence of $g_{\calU}$ and $g_{L_Y}$ on $t$ while remaining in the class of iterated edge metrics.
Since the coordinate $t=\rho_Y(u)$ is part of the control data, this homotopy can be performed consistently across all of the local trivializations $\phi$.

So without loss of generality we may assume that
\begin{equation*}
	(\phi^{-1})^*g
	= g_{\calU}(y) + dt^2 + t^2 g_{L_Y}(y,z), \Mand
	(\phi^{-1})^*g'
	= g_{\calU}'(y) + dt^2 + t^2 g_{L_Y}'(y,z).
\end{equation*}
The metrics $g_{\calU}$ and $g_{\calU}'$ are homotopic and, by inductive hypothesis, so are the metrics $g_{L_Y}$ and $g_{L_Y}'$.
Thus the metrics $(\phi^{-1})^*g$ and $(\phi^{-1})^*g'$ are homotopic within the class of iterated edge metrics on $\calU \times C(L_Y)$.
Using consistency of the trivializations $\phi$ we can patch these homotopies together and see that $g$ and $g'$ are homotopic in a neighborhood of $Y$.

So without loss of generality we can assume that $g$ and $g'$ coincide in a neighborhood of $Y$ and, in this neighborhood, are independent of $\rho_Y$.
As in the proof of Proposition \ref{prop:BlowUp} we consider the space 
\begin{equation*}
\wt{X}_1' := \left((\wh{X} \setminus Y) \times \{-1\}\right) \sqcup \left((\wh{X} \setminus Y) \times \{+1\}\right)
\sqcup \big(S_Y \times (-2,2) \big) / \sim
\end{equation*}
Define the lift $\wt g$ of $g$ to $\wt X_1'$ by $g$ on each copy of $\wh X \setminus Y$ and 
\begin{equation*}
	g_{\calU}(y)  +  g_{L_Y}(y,z) + dt^2
\end{equation*}
on each neighborhood of $S_Y \times (-2,2)$ corresponding to the trivialization $\phi$ as above, and define $\wt g'$ similarly.
Then $\wt g$ and $\wt g'$ are iterated edge metric on a space of depth $k-1$ so by inductive hypothesis are homotopic.
Moreover since they coincide in $S_Y \times (-2,2)$, the homotopy can be taken to be constant in a neighborhood of $S_Y$, and hence the homotopy descends to a homotopy of $g$ and $g'$.

If $g$ and $g'$ are rigid, the homotopies above preserve this.
\end{proof}

Cheeger also defines \cite{Cheeger-symp} (p.\ 127) a class of admissible metrics $g$ on the regular part of 
a smoothly stratified pseudomanifold $\wh{X}$. He uses a slightly different decomposition of $\wh{X}$ and 
assumes that on each `handle' of the form  $(0,1)^{n-i} \times C(N^{i-1})$, $h$ induces a metric  quasi-isometric
to one of the form
\[
(d y_1)^2 + \ldots + (d y_{n-i})^2 + (d r)^2 + r^2 g_{N^{i-1}};
\]
see \cite{Cheeger-symp} for the details. Using the proof of Proposition \ref{thm:metric} as well as 
\cite{Cheeger-symp} (page 127), we obtain the following 

\begin{proposition} \label{prop:adm}
{\item 1)} Any iterated edge  metric as in Definition \ref{def:metric}) is admissible in the sense of Cheeger.
{\item 2)} Any two admissible metrics are quasi-isometric.
\end{proposition}

As a first application, we discuss the $L^2$ cohomology of Witt spaces with respect
to iterated edge metrics.

\begin{definition} A stratified pseudomanifold $\wh{X}$ is a Witt space if for
all strata $Y \in {\mathfrak S}$, if the corresponding link $L_Y$ has even dimension,
$\dim L_Y = f_Y$, then $I\! H_m^{f_Y/2} ( L_Y)= 0$; here $m$ refers to either the lower or upper
middle perversity.
\end{definition}
In this paper, we shall consider only orientable Witt spaces.

There is a famous result concerning the $L^2$ cohomology of Witt spaces, due to Cheeger:
\begin{theorem}\label{theo:cheeger}(Cheeger) Let $\widehat{X}$ be a Witt space endowed with
an iterated edge metric $g$. Consider any stratum $Y$ with link $L_Y$ of even dimension $f_Y$, 
and denote by $H_{(2)}^{f_Y/2}$ and $\mathcal{H}_{(2)}^{f_Y/2} ( L_Y) $ its middle degree $L^2-$cohomology 
and Hodge cohomology, respectively, defined with respect to the iterated edge metric on $L_Y$
induced by $g$. Then 
\begin{equation}\label{eq:cheeger}
H_{(2)}^{f_Y/2} ( L_Y)=\calH_{(2)}^{f_Y/2} ( L_Y) = 0.
\end{equation}
\end{theorem}

\section{Iterated edge vector fields and operators} \label{Operators}

On a closed manifold,
$L^2$ and Sobolev spaces are defined using a Riemannian metric but are independent of which metric is used to define them.
A differential operator induces a bounded map between these spaces, and ellipticity is enough to guarantee that this map 
is Fredholm. All of this fails when the manifold is not closed, and in this section we will analyze how much can be recovered 
for iterated edge metrics.

In many respects, manifolds with bounded geometry (such as the manifold $X = \mbox{reg}\,(\wh{X})$ 
endowed with the complete metric $\wt{g}$) are the most tractable class of noncompact manifolds on
which to do analysis.  There are natural classes of $L^2$ and Sobolev spaces, and `uniform' differential 
operators induce bounded maps between them. The well-developed calculus of uniform pseudo-differential operators
contains parametrices of uniform elliptic operators, which leads to certain uniform elliptic regularity results.
Moreover, the compactification of $X$ as a manifold with corners provides a natural way to define weighted 
$L^2$ and Sobolev spaces and the uniform calculus extends easily to act between these.

In this section we describe this machinery and explain how it can be applied to the de Rham operator of the iterated conic metric $g$.
The uniform pseudodifferential calculus will provide us with a parametrix even after twisting by projective finitely generated modules over a $C^*$-algebra.
Although this will not be enough to establish Fredholm properties it will show that, for an elliptic operator, these depend solely on the behavior near the boundary.

\subsection{Edge vector fields on $X$}\label{complete-metric} $ $\newline

Associated to the metric $\wt g$ on $X$ is the space of `iterated edge' vector fields
\begin{equation}\label{FirstDescVie}
	\cV_{\ie}
	= \{ V \in C^\infty (\wt X, T\wt X): X \ni q \mapsto \wt g_q(V,V) \in \bbR^+ \text{ is bounded} \}
\end{equation}
which, in the notation of \S 3, on a neighbourhood of the form $\calW_{q_1, \ldots, q_s}$ are locally spanned 
by vector fields of the form 
\[
r_1 \ldots r_{s-1} \del_{r_1}, r_1 \ldots r_{s-1} \del_{y^{(1)}}, r_1 \ldots r_{s-2} \del_{r_2}, r_1 \ldots r_{s-2} \del_{y^{(2)}}, 
\ldots, \del_{y^{(s)}}.
\]
It is easy to see that $\cV_{\ie}$ forms a locally finitely generated, locally free Lie algebra with respect to the usual bracket
on vector fields, so Swan's theorem shows that there is a vector bundle ${}^{\ie}TX$ over $\wt X$ whose space of sections is $\cV_{\ie}$,
\begin{equation}\label{DefIeTX}
	\CI (\wt X, ^{\ie}T X ) = \cV_{\ie}.
\end{equation}
This bundle ${}^{\ie}TX$ coincides with the usual tangent bundle $TX$ over the interior of $\wt X$ and is isomorphic to $T\wt X$, though there is no canonical isomorphism.
Since \eqref{FirstDescVie} shows that $g(V,V)$ extends to $\wt X$ for any section $V$ of ${}^{\ie}TX$, it is easy to see that $\wt g$ defines a metric on ${}^{\ie}TX$. 

\begin{proposition} $(X, \widetilde{g})$ is a complete Riemannian manifold of bounded geometry.
\end{proposition} 
\begin{proof}
Recall the theorem of Gordon-de Rham-Borel, which states that a manifold is complete if and only if  it admits a 
nonnegative, smooth, proper function with bounded gradient. For this metric $\wt g$, such a function is $-\log (x_0 \cdots x_m)$.
To prove that $g$ has bounded geometry one needs to check that the curvature tensor of $\wt g$, and its covariant derivatives, 
are bounded and that the injectivity radius of $\wt g$ is bounded away from $0$. These can be shown as in \cite{ALN}.
 \end{proof}

The set of $\ie$-differential operators is the enveloping algebra of $\cV_{\ie}$; i.e., it consists of linear combinations (over $\CI(\widetilde{X})$) 
of finite products of elements of $\cV_{\ie}$. We denote by $\Diff_{\ie}^k(X)$ the subset of differential operators that have local descriptions involving products of at most $k$ elements of $\cV_{\ie}$.
If $E$ and $F$ are vector bundles over $\wt X$, then the space of $\ie$-differential operators acting between sections of $E$ and sections of $F$ is defined similarly, by taking linear combinations over $\CI(\wt X, \Hom(E,F) )$.

We define Sobolev spaces for $\ie$ metrics by
\begin{gather*}
	H^0_{\ie}(X) = L^2_{\ie}(X) = L^2(X, \dvol(\wt g)) \\
	H^k_{\ie}(X)
	= \{ u \in L^2_{\ie}(X) : Au \in L^2_{\ie}(X), \Mforevery A \in \Diff_{\ie}^k(X) \}, \; k \in \bbN
\end{gather*}
then define $H^t_{\ie}(X)$ using Calder\'on interpolation for $t \in \bbR^+$ and duality for $t \in \bbR^-$.
Sobolev spaces for sections of bundles over $ \widetilde{X}$ are defined similarly.

We will also allow for operators to act between sections of certain bundles of projective finitely generated modules over a $C^*$-algebra; see \cite{ST} for the basic definitions.
We assume that we have a continuous map $r_0: X \to B\Gamma$ which extends continuously to
\begin{equation*}
 r: \widehat{X}\to B\Gamma
	\end{equation*}
 where $\Gamma$ is a countable, finitely presented, group. This determines a $\Gamma$-covering, $ \widehat{X}^\prime \to 
\widehat{X}$;  and we will denote by $\wt{C^*_r}\Gamma$ the corresponding bundle, over $\widehat{X}$, of free left 
$C^*_r\Gamma$-modules of rank one: 
\begin{equation}\label{cEDef}
	\wt{C^*_r}\Gamma: = C^*_r\Gamma \times_\Gamma \widehat{X}^\prime.
\end{equation} 
Observe that this bundle induces, after pull back by the blowdown map $\wt X \rightarrow \widehat{X}$,
a bundle on $\wt X$ (for which we keep the same notation).
Given vector bundles $E$ and $F$ over $\wt X$ of rank $k$ and $\ell$, we define  bundles $\cE$ and $\cF$ over $\wt{X}$ 
by tensoring $E$ and $F$ by $\wt{C^*_r}\Gamma$; we obtain in this way 
bundles of projective finitely generated $C^*_r \Gamma$-modules   of rank $k$ and $\ell$ . 
We shall briefly refer to  $\cE$ and $\cF$
as $C^*_r\Gamma$-bundles.
An iterated edge differential operator acting between sections of $\cE$ and $\cF$ is defined as above, but allowing the coefficients to be 
$C^*_r\Gamma$-linear. The space of such operators will be denoted
\begin{equation*}
	\Diff_{\ie,\Gamma}^*(X; \cE, \cF).
\end{equation*}
Finally, we denote by  $H^t_{\ie,\Gamma}(X;\cE) $  the corresponding Sobolev $C^*_r \Gamma$-module, see \cite{Mich-Fomenko}.

\subsection{Uniform pseudodifferential operators} \label{sec:uniform} $ $\newline

We have verified in the previous subsection that $\ie$ metrics have bounded geometry;
hence we can make use of the calculus of uniform pseudo-differential operators as described in the work of Meladze-Shubin (see
\cite{Meladze-Shubin} and \cite{Kordyukov}).

Among the smooth functions on $X$, we single out the space $\cB C^\infty (X)$ of functions that are uniformly bounded with uniformly bounded derivatives. Smooth functions on $\wt X$ are in $\cB \CI (X)$, but generally the latter space will allow non-smooth behavior normal to the boundary faces of $\wt X$.
A vector bundle over $X$  
is said to be a {\em bundle of bounded geometry} if it has trivializations whose transition functions are (matrices with entries) in $\cB\CI(X)$. 
Clearly vector bundles that extend to  $\wt X$ have bounded geometry. 

By requiring the coefficients to be
in $\cB C^\infty$ we can define the space $\Diff^*_{\cB} (X;E,F)$ and, more generally, 
$\Diff^*_{\cB,\Gamma} (X;\cE,\cF)$.
Since $\cB\CI(X)$ contains $\CI(\wt X)$, these spaces of operators contain $\Diff_{\ie}^*(X;E,F)$ and $\Diff_{\ie,\Gamma}^*(X;\cE, \cF)$.

Next, using the bounded geometry of $(X,\wt g)$,
it is possible to find a countable cover by normal coordinate charts of radius $\eps>0$, $\cU_{\eps}(\zeta_i)$, such that $\cU_{2\eps}(\zeta_i)$ has uniformly bounded, finite multiplicity as a cover of $X$.
We can find partitions of unity $\wt \phi_i$, $\phi_i$ subordinate to $\{ \cU_{2\eps}(\zeta_i) \}$ and $\{ \cU_{\eps}(\zeta_i) \}$ respectively such that $\wt \phi_i$, $\phi_i$ have bounded derivatives uniformly in $i$, and such that
\begin{equation*}
	\wt\phi_i \rest{\mathrm{supp} (\phi_i)} \equiv 1.
\end{equation*}
These functions can be used to `transfer' constructions from $\bbR^n$ to $X$. 

We next recall how to `transfer' pseudodifferential operators from $\bbR^n$.
Let $E$ and $F$ be vector bundles over $\wt X.$ 
An operator $A:\CIc(X;E) \to \CIc(X;F)$ is called a {\bf uniform pseudodifferential operator} of order $s \in \bbR$, 
\begin{equation*}
	A \in \Psi^s_{\cB}(X;E,F), 
\end{equation*}
if its Schwartz kernel $\cK_A \in \CmI(X^2; \Hom(E,F))$ satisfies the three following properties: 

	{\em i)} There is a $C_A>0$ such that 
\begin{equation*}
	\cK_A(\zeta, \zeta')=0 \Mif d(\zeta, \zeta')>C_A,\\
\end{equation*} where $d$ denotes the Riemannian metric of $X$ associated to $\wt g.$

	{\em ii)} Outside of the diagonal $\cK_A$ is smooth uniformly in that for every $\delta>0$, 
	and any multi-indices $\alpha$, $\beta$ there is a constant $C_{\alpha\beta\delta}>0$ such that
\begin{equation*}
	|D^\alpha_\zeta D^\beta_{\zeta'} \cK_A(\zeta, \zeta')| \leq C_{\alpha\beta\delta}, 
	\text{ whenever } d(\zeta, \zeta') >\delta. 
\end{equation*} 

	{\em iii)} For any $i \in \bbN$, $\wt \phi_i A \phi_i$ is a pseudodifferential operator of order $s$ in $B_{2\eps}(0)$,
	whose full symbol $\sigma$ satisfies the usual symbol estimates
\begin{equation*}
	\abs{ D_\zeta^\alpha D_\xi^\gamma \sigma( \wt \phi_i A \phi_i )(\zeta, \xi) }
	\leq C_{\alpha\beta\gamma} (1+|\xi|_{{\wt g}}^2)^{\frac12(s-|\gamma|)}
\end{equation*}
	with bounds independent of $i$, where $ |\xi|_{{\wt g}}$ denotes the norm of $\xi \in T^*_\zeta X$ induced 
	by $\wt g.$

We shall always assume that the symbols are (one-step) polyhomogeneous. Uniform pseudo-differential operators form an algebra. There is a well defined principal  symbol map, with values in $\cB C^\infty (S^* X, \hom (\pi^* E,\pi^* F))$. Ellipticity is defined in a natural way (one requires the principal
symbol to be uniformly invertible, i.e. invertible with inverse in $\cB C^\infty $).  
The principal symbol  $\sigma(P)$ of a uniform pseudodifferential operator $P$ can certainly be seen as a  section on the restriction to $X$ of 
 $\Ie T^*X$ (the bundle  dual to $\Ie T X$); 
 notice however that, in general, $\sigma(P)$ will not extend as a smooth section on $\Ie T^*X \rightarrow \wt X.$

For a bundle of bounded geometry, $E$, it is straightforward to define $\cB$-Sobolev spaces for any $s \in \bbR$
\begin{multline}\label{SobSpaces}
	H^s_{\cB}(X;E) \\
	= \{ u \in \CmI(X;E) : \phi_i u \in H^t(\bbR^n;E) \text{ with norm bounded uniformly in } i \}.
\end{multline}
The same is true for $C^*_r\Gamma$-bundles and we denote by $H^s_{\cB,\Gamma} (X;\cE)$, $s\in \bbR$, the corresponding $C^*_r \Gamma$-module.
Uniform pseudodifferential operators extend to bounded operators between $\cB$-Sobolev spaces, as in the closed case.

If a map $r:\widehat{X}\to B\Gamma$ is given, then we can define uniform pseudo-differential operators between sections of $\cE$ and sections of $\cF$ by combining the above definition and the classic construction  of Mishchenko and Fomenko; we denote  by  $\Psi^*_{\cB, \Gamma}(X;\cE, \cF)$
the corresponding algebra. 
Notice that the principal symbol is in this case a $C^*_r\Gamma$-linear map between the lifts of $\cE$ and $\cF$ to the cotangent bundle.

The intersection over $s\in \bbR$ of the $\Psi_{\cB,\Gamma}^s(X;\cE,\cF)$ is denoted $\Psi_{\cB,\Gamma}^{-\infty}(X;\cE,\cF)$ and consists of smoothing operators whose integral kernel
in $X\times X$ is in $\cB C^\infty $.

Elements of the uniform calculus also define bounded maps between {\it weighted} $C^*_r\Gamma$-Sobolev spaces.

\begin{lemma}
If $A \in \Psi_{\cB,\Gamma}^s(X; \cE, \cF)$, then for any $a, t \in \bbR$, $A$ induces a bounded operator
\begin{equation*}
	A: \rho^a H^t_{\ie,\Gamma}(X;\cE) \to \rho^a H^{t-s}_{\ie,\Gamma}(X;\cF).
\end{equation*}
\end{lemma}

\begin{proof}
It is enough to check that $\rho^{-a} A \rho^{a} \in A \in \Psi_{\cB,\Gamma}^s(X; \cE, \cF)$ for any $a \in \bbR$. Also, we can assume that $A$ acts on scalar functions without any loss of generality.
The integral kernel of $\rho^{-a} A \rho^{a}$ is given by
\begin{equation*}
	\lrpar{\frac{\rho(\zeta)}{\rho(\zeta')}}^a \cK_A(\zeta, \zeta')
\end{equation*}
and the lemma will follow from noting that $\lrpar{\frac{\rho(\zeta)}{\rho(\zeta')}}^a$ is a bounded smooth function on the support of $\cK_A$.
By the triangle inequality, we see that
\begin{equation*}
	\abs{ \frac1{\rho(\zeta)} - \frac1{\rho(\zeta')} } \leq d(\zeta, \zeta') \leq C_A
	\implies
	\abs{ \frac{\rho(\zeta)}{\rho(\zeta')} - 1} \leq \rho(\zeta)C_A
\end{equation*}
so that either $\rho(\zeta)$ and $\rho(\zeta')$ are both large, or their quotient is close to $1$, and the lemma follows.
\end{proof}

For us the most important property of the uniform pseudodifferential calculus is that it has a symbolic calculus.
By standard constructions, this implies that any 
elliptic element in $\Diff_{\cB,\Gamma}^k(X; \cE, \cF)$ has a symbolic parametrix, i.e. an inverse modulo smoothing operators. 
In particular, using the above Proposition, we see that an elliptic $\ie$ operator $A \in \Diff_{\ie,\Gamma}^k(X; \cE, \cF)$ has a {\em symbolic parametrix}
\begin{equation*}
	Q \in \Psi_{\cB,\Gamma}^{-k}(X; \cF, \cE) \Mst \\
	\Id_{\cE} - QP \in \Psi_{\cB,\Gamma}^{-\infty}(X;\cE), \quad
	\Id_{\cF} - PQ \in \Psi_{\cB,\Gamma}^{-\infty}(X;\cF). 
\end{equation*}

The symbolic calculus also allows for the standard characterization of Sobolev spaces.
For instance, for $N \in \bbN$, we have
\begin{multline*}
	H^N_{\cB}(X)
	= \{ u \in \CmI(X) : Au \in L^2(X) \text{ for all } A \in \Diff_{\cB}^N(X) \} \\
	= \{ u \in \CmI(X) : Au \in L^2(X) \text{ for some uniformly elliptic } A \in \Diff_{\cB}^N(X) \}
\end{multline*}
and, in fact, given any fixed uniformly elliptic $A \in \Diff_{\cB}^N(X)$, $H^N_{\cB}(X)$ is equal to the maximal domain of $A$ as an unbounded operator on $L^2(X)$.
This characterization, applied to an elliptic operator $A \in \Diff_{\ie}^N(X)$, shows that $H^N_{\ie}(X) = H^N_{\cB}(X)$. Using Calderon interpolation and duality, we see that $H^t_{\ie}(X) = H^t_{\cB}(X)$ for all $t \in \bbR$, and the same is true for sections of bundles of bounded geometry and the corresponding $C^*_r\Gamma$-bundles.

\subsection{Incomplete iterated edge operators} $ $\newline
The set of incomplete iterated edge differential operators, $\Diff_{\iie,\Gamma}^*(X;\cE,\cF)$ is defined in terms of $\Diff_{\ie,\Gamma}^*(X;\cE,\cF)$ by
\begin{equation*}
	\Diff_{\iie,\Gamma}^k(X;\cE,\cF) = \rho^{-k} \Diff_{\ie,\Gamma}^k(X;\cE,\cF),
\end{equation*}
where $\rho$ denotes a total bdf for $X$ (e.g., $\rho = x_0\cdots x_{m-1}$).
As an operator between weighted $L^2$ spaces  with appropriate {\em different} weights, an operator $A \in \Diff_{\iie,\Gamma}^k(X;\cE,\cF)$ is unitarily equivalent to an iterated edge operator.
Thus, for instance, for any $a \in \bbR$, $A$ defines an unbounded operator
\begin{equation*}
	A: \rho^a L^2_{\ie,\Gamma}(X;\cE) \to \rho^{a-k} L^2_{\ie,\Gamma}(X;\cF)
\end{equation*}
which has a unique closed extension whose domain is $\rho^{\alpha}H^k_{\ie,\Gamma}(X;\cE)$; moreover, $A$ defines bounded operators
\begin{equation*}
	\rho^{a} H^t_{\ie,\Gamma}(X;\cE) \to \rho^{a-k} H^{t-k}_{\ie,\Gamma}(X;\cF)
\end{equation*}
for every $a$ and $t \in \bbR$.
However, it is the more complicated behavior of $A$ as an unbounded operator
\begin{equation}\label{Aop}
	A: \rho^a L^2_{\ie,\Gamma}(X;\cE) \to \rho^{a} L^2_{\ie,\Gamma}(X;\cF)
\end{equation}
that we will be concerned with.
We point out that the operator \eqref{Aop} is unitarily equivalent to the unbounded operator
\begin{equation*}
	\wt A = \rho^{k/2} A \rho^{k/2}: 
	\rho^{a-k/2} L^2_{\ie,\Gamma}(X;\cE) \to \rho^{a+k/2} L^2_{\ie,\Gamma}(X;\cF),
\end{equation*}
Since $\wt A\in \Diff_{\ie,\Gamma}^*(X;\cE,\cF)$, this  shows that
the study of incomplete iterated edge operators acting on a fixed Hilbert spaces is the same as the study of complete $\ie$-operators acting between {\em different} Hilbert spaces.

We point out that the $L^2$ spaces of the incomplete iterated edge metric $g$ and the associated complete $\ie$ metric $\wt g = \rho^{-2 }g$ are related by
\begin{equation*}
	L^2_{\ie,\Gamma}(X, \cE) = \rho^{n/2} L^2_{\iie,\Gamma}(X, \cE)
\end{equation*}
with $n$ equal to the dimension of $X$, 
so switching between them only involves a shift of the weight. 
Similarly, we introduce the spaces $H^t_{\iie, \Gamma}(X; \cE)$ for $t \in \bbR$ by
\begin{equation*}
	H^t_{\ie, \Gamma}(X; \cE) = \rho^{n/2} H^t_{\iie, \Gamma}(X; \cE).
\end{equation*}
Thus, for instance, if $N \in \bbN$ then $H^N_{\iie,\Gamma}(X, \cE)$ is the set of elements $ u \in L^2_{\iie,\Gamma}(X, \cE)$ such 
 that for any vector fields $V_1, \ldots, V_p \in \mathcal{V}_{\ie}$ where $p \leq N,$ we have 
 $V_1 \ldots V_p u \in L^2_{\iie,\Gamma}(X, \cE).$

We say that $A\in \Diff_{\iie,\Gamma}^k(X; \cE, \cF)$ is elliptic if $\wt A = \rho^k A$ is an elliptic $\ie$ operator.
Elliptic $\ie$ operators always have a symbolic parametrix (see \S\ref{sec:uniform}).
A symbolic parametrix $\wt Q$ for $\wt A$ yields a symbolic parametrix $Q = \rho^{k/2}\wt Q \rho^{k/2}$ for $A$.
As is well-known, since smoothing operators are not necessarily $C^*_r\Gamma$-compact, a symbolic parametrix 
is generally not enough to determine when an operator is Fredholm, so one also needs to know about the behavior at 
the boundary.
  
However, the uniform calculus does establish elliptic regularity in the sense that, whenever $B \in \Diff_{\ie,\Gamma}^k(X; \cE, \cF)$ 
is elliptic, we have 
 \begin{equation}\label{EllReg}
	u \in \rho^\epsilon L^2_{\iie,\Gamma}(X, \cE), \quad
	Bu \in \rho^\epsilon L^2_{\iie,\Gamma}(X, \cF)
	\implies u \in \rho^\epsilon H^N_{\iie,\Gamma}(X, \cE).
\end{equation}

\subsection{The de Rham operator} $ $\newline
We are interested in analyzing the de Rham operator of an $\iie$ metric,
\begin{equation*}
	\eth_{\dR} = d + \delta: \Omega^*X \to \Omega^*X.
\end{equation*}
As with the tangent bundle, it is convenient to replace the bundle of forms $\Omega^*(X)= \CI (X,\Lambda^* (T^*X))$ with the bundle of $\iie$-forms, 
\begin{equation*}
	\Iie \Omega^*(X) = \CI (X, \Lambda^* (\Iie T^*X) ),
\end{equation*}
where $\Iie T^*X \rightarrow \wt X $ is the rescaled bundle (cf. \cite[Chapter 8]{APS Book}) defined by  
\begin{equation*}
	\CI (\wt X,\Iie T^*X) = \rho \CI (\wt X,\Ie T^*X).
\end{equation*}
We set $\Iie \Lambda^*X = \Lambda^* (\Iie T^*X),$ and  we have
\begin{equation*}
	\eth_{\dR} \in \Diff_{\iie}^1(X; \Iie\Lambda^*(X), \Iie\Lambda^*(X))
\end{equation*}
as we now explain.

First note that whether $\eth_{\dR}$ is an element of $\Diff_{\iie}^1(X; \Iie\Lambda^*(X), \Iie\Lambda^*(X))$ can be checked 
locally in coordinate charts. There is nothing to check in the interior of the manifold. Then, with the notations of 
\S 3, we consider a distinguished neighborhood  $W$ of a point of a stratum $Y.$ Thus 
$W$ is diffeomorphic to $B \times C(Z)$ where $B$ is an open subset of $Y$ which is diffeomorphic to a vector space and 
$C(Z)$ is the cone whose base $Z$ is a stratified space. 
The `radial' coordinate of the cone will be denoted by $x$. 

As in \S\ref{section:iterated}, the fibration over $B$ extends to $W$, 
\begin{equation*}
	Z \times [0,1)_x - W \xrightarrow{\wt\phi} B,
\end{equation*}
and using $x$ and a choice of connection for this fibration we can write
\begin{equation*}
	T^*X\rest{W} = \ang{dx} \oplus T^*Y \oplus T^*Z.
\end{equation*}
With respect to this splitting the metric $g$ restricted to $W$ has the form
\begin{equation*}
	g = dx^2 + \wt\phi^*g_Y + x^2 g_Z
\end{equation*}
and the differential forms on $X$ can be decomposed as
\begin{equation}\label{SplitForms}
\begin{gathered}
	\Lambda^*X = 
	(\Lambda^*Y \wedge \Lambda^* Z)
	\oplus dx \wedge
	(\Lambda^*Y \wedge \Lambda^* Z) \\
	\Iie \Lambda^*X = 
	(\Lambda^*Y \wedge x^{\bN} \Lambda^* Z)
	\oplus dx \wedge
	(\Lambda^*Y \wedge x^{\bN} \Lambda^* Z)
\end{gathered}
\end{equation}
where $\bN$ is the `vertical number operator', i.e., the map given by multiplication by $k$ when restricted to forms of vertical degree $k$.
This allows us to split the exterior derivative into
\begin{equation*}
	d = 
	\df e_{dx} \pa_x
	\oplus d^Y
	\oplus d^Z
\end{equation*}
and correspondingly
\begin{equation*}
	\delta =
	\star^{-1} \df e_{dx} \pa_x \star
	\oplus 
	\star^{-1} d_Y \star
	\oplus 
	\star^{-1} d_Z \star
	= 
	\star^{-1} \df e_{dx} \pa_x \star
	\oplus \delta^Y_x
	\oplus \delta^Z_x
\end{equation*}
where the $x$-dependence in $\delta^Y_x$ and $\delta^Z_x$ comes from the $x$-dependence of the Hodge star operator, $\star$.
A straightforward computation shows that with respect to the splitting \eqref{SplitForms} of $\Iie\Lambda^*X$,
\begin{equation}\label{PreDNearBdy}
	\eth_{\dR}
	= \begin{pmatrix}
	\frac1x( d^Z + \delta^Z_x) + d^Y + \delta^Y_x & - \star^{-1}\pa_x \star - \tfrac1x(f-\bN) \\
	\pa_x + \tfrac1x \bN & -\frac1x( d^Z + \delta^Z_x) - d^Y - \delta^Y_x 
	\end{pmatrix}.
\end{equation}
As in \cite[(19)]{Hunsicker-Mazzeo} one can write this in terms of operators related to the fibration, 
however for our purposes it is more important to point out that the leading order term with respect to $x$ (as an $\iie$ operator) is given by
\begin{equation}\label{DNearBdy}
	\eth_{\dR} \sim
	 \begin{pmatrix}
	\tfrac1x \eth^Z_{\dR} + \eth^Y_{\dR} & - \pa_x - \tfrac1x(f-\bN) \\
	\pa_x + \tfrac1x \bN & -\tfrac1x \eth^Z_{\dR}  - \eth^Y_{\dR} 
	\end{pmatrix}.
\end{equation}
where $\eth^Y_{\dR}$ and $\eth^Z_{\dR}$ are the de Rham operators of $\wt\phi^*g_Y\rest{x=0}$ and $g^Z\rest{x=0}$, respectively. 
In effect, because of the weighting of the vertical forms, the Hodge star operator is asymptotically acting like the Hodge star operator of the product metric at $\{x=0\}$.

By induction on the depth of the stratification and using \eqref{DNearBdy} one proves without difficulties the following:

\begin{lemma}  \label{lem:iie} The operator $\eth_{\dR} $ is  in $\Diff_{\iie}^1$, i.e.,
 $\rho \eth_{\dR} $ is  in $\Diff_{\ie}^1$
\end{lemma} 
 
We are also interested in the behaviour of $\eth_{\dR}$ after twisting to get $C^*$-algebra coefficients. Thus we assume, as before, that we have a continuous  map
\begin{equation*}
	r:  \widehat{X} \to B\Gamma
\end{equation*}
We compose $r$ with the blow-down map $\beta$ and we pull-back the universal bundle $E\Gamma$
to $\widetilde{X}$ using $f\circ \beta$. We obtain a Galois $\Gamma$-covering $\widetilde{X}^\prime$
over   $\widetilde{X}$ and 
the  associated bundle 
$\wt C^*_r\Gamma \to \widetilde{X}$, with 
$$\wt C^*_r\Gamma:= C^*_r \Gamma \times_\Gamma \widetilde{X}^\prime\,.$$
 We restrict $\wt C^*_r\Gamma$ to $X$.
Endowing $C^*_r\Gamma\times \widetilde{X}^\prime$, as a trivial bundle over $\widetilde{X}^\prime$, with the trivial connection induces a (non-trivial) flat connection on the bundle $\wt C^*_r\Gamma \rightarrow \widetilde{X}$;
we also obtain a flat connection on the restriction of   $\wt C^*_r\Gamma$ to $X$
(and it is obvious that this connection will automatically extend to $\widetilde{X}$). 
Using the latter connection we can define directly $\wt \eth_{\dR}$, the twisted de Rham operator on the sections of the vector bundle 
\begin{equation*}
	\Iie\Lambda_\Gamma^*(X)  = \Iie\Lambda^*X \otimes \wt C^*_r\Gamma.
\end{equation*}
By construction  $\wt \eth_{\dR}\in \Diff^*_{\iie,\Gamma}$, i.e. $\rho \wt \eth_{\dR}$ is an element in $\Diff^*_{\ie,\Gamma}$.

\section{Statement of main theorem} 

One consequence of \eqref{DNearBdy} is that the fibre enters into the description of the de Rham operator near the
boundary only through its de Rham operator (and lower order terms). This will allow us to analyze the de Rham operator 
by induction on the depth of the Witt space $\hat X$. The base case can be taken to be a closed manifold,  which is classical.
The case of a stratification of depth one is analyzed in the work of Hunsicker and the third author \cite{Hunsicker-Mazzeo}, 
where the relationship between intersection cohomology and Hodge cohomology is treated in detail. Our result for depth 
one stratifications is implicitly contained in \cite{Hunsicker-Mazzeo} and similarly the study of this relationship generalizes 
readily to our situation. The treatment in  \cite{Hunsicker-Mazzeo} relies heavily on the edge calculus \cite{Mazzeo:edge} 
which allows refined results, such as finding conormal representatives of cohomology classes. Though we will not be able to 
use the edge calculus directly, we will often proceed by adapting arguments from \cite{Mazzeo:edge} to our context.

Our eventual goal is to establish the following two theorems for (suitably scaled) iterated edge metrics on Witt spaces. 

\begin{theorem} \label{InductiveThm1} 
Let $\widehat X$ be a Witt space endowed with a suitably scaled iterated edge metric $g$. Let $X=\mbox{reg}\,(\widehat{X})$.
{\item 1)} 
As an unbounded operator on $C_c^\infty (X,  \Iie\Lambda^*(X))\subset L^2_{\iie}(X;  \Iie\Lambda^*(X) )$, 
the de Rham operator of $g$ has a unique closed  extension and hence is essentially self-adjoint. 
{\item 2)} 
For any $\eps >0$, the domain of this unique closed extension, still denoted $\eth_{\dR}$, is contained in 
\begin{equation*}
	 \rho^{1-\eps}L^2_{\iie}(X; \Iie\Lambda^*(X)) \cap H^1_{\loc}(X;  \Iie\Lambda^*(X))
\end{equation*}
which is compactly included in $L^2_{\iie}(X;  \Iie\Lambda^*(X))$. 
{\item 3)} 
As an operator on its maximal domain endowed with the graph norm, $\eth_{\dR}$ is Fredholm.
{\item 4)} 
$\eth_{\dR}$   has discrete spectrum of finite multiplicity.
\end{theorem}

Item 1), 3) and 4)  have been proved by  Cheeger  \cite{Ch} (using a different method)
for metrics quasi-isometric to a piecewise flat ones.

Assume now that we are also given a map $r: \widehat{X }\to B\Gamma$. Then we have

\begin{theorem} \label{InductiveThm2} 
The de Rham operator with values in the flat bundle $\widetilde{C^*_r}\Gamma$, denoted
$\wt \eth _{\dR}$, has a unique self-adjoint closed extension
to  the $C^*_r\Gamma$-module 
$L^2_{\iie,\Gamma}(X; \Iielaga(X) )$. This extension is regular 
and its  domain is compactly included in $L^2_{\iie,\Gamma}(X; \Iielaga(X) )$.
In particular, it defines an index class 
$$
\Ind (\widetilde{\eth}_{\dR} ) \,\in\, K_j (C^*_r \Gamma),\quad j=\dim X \,{\rm mod}\, 2.$$
\end{theorem}

In \S\ref{sec:Sign}, we apply this to the signature operator and deduce the following.

\begin{corollary}
The signature operator of a suitably scaled iterated  conic metric $g$ on a Witt space is
essentially self-adjoint. Its closure is Fredholm and has only discrete spectrum of finite multiplicity.
Given a map $r: \widehat{X }\to B\Gamma$, the signature operator of $g$ twisted by the flat bundle 
$\widetilde{C^*_r}\Gamma\to X$ has a unique closed self-adjoint regular extension and defines an index class 
$$
\Ind (\widetilde{\eth}_{\sign} ) \,\in\, K_j (C^*_r \Gamma),\quad j=\dim X \,{\rm mod}\, 2.$$
\end{corollary}

Certainly the theorem holds for the base case of a closed manifold.
We work on a stratification of depth $m$ and  assume inductively that the theorem holds for all stratified manifolds of depth at most $m-1$.
We are principally interested in the behavior in a distinguished neighborhood $W$ of a point $q \in Y$ with associated cone $C(Z).$

\section{Analyzing the signature operator inductively}

In this section we analyze the behavior of the de Rham operator near the boundary of $X$.
More precisely, we define a model for this operator at each point of the boundary
and then we establish that these model operators are invertible when acting on the appropriate Sobolev spaces. 
Taken together, ellipticity and this asymptotic invertibility are enough to establish the Fredholm properties we seek.

The main advantage of the de Rham operator over an arbitrary $\iie$ operator lies in \eqref{DNearBdy}.
Indeed this shows that, at a given point $q$ on the boundary, the leading order behavior of $\eth_{\dR}$ involves the fibre $Z$ over $q$ only through its de Rham operator $\eth_{\dR}^Z$.
To take advantage of this structure we will `partially complete' $\eth_{\dR}$ so that we can treat the $N\pa M \oplus \phi^*TY$ directions as `complete' and the $TZ$ directions as `incomplete', thus setting up an inductive scheme.

\subsection{The partial completion of $\eth_{\dR}$}

Recall that  \eqref{DNearBdy} was written in a  distinguished neighborhood  $W$ of a point of a stratum $Y.$  
$W$ is diffeomorphic to $B \times C(Z)$ where $B$ is an open subset of $Y$ which is diffeomorphic to a vector space and 
$C(Z)$ is the cone whose link $Z$ is a smoothly stratified space. The `radial' coordinate of the cone will still be 
denoted by $x$, which we can identify with one the boundary defining functions $x_j$ and thereby extend globally 
to $\wt X.$ To take advantage of the structure of the de Rham operator in $W$, as it 
appears in \eqref{DNearBdy}, we define the `partial conformal completion' of the signature operator
\begin{equation*}
	D_0 = x^{1/2} \eth_{\dR} x^{1/2}.
\end{equation*}

The advantage of using $x^{1/2} \eth_{\dR} x^{1/2}$ over, say, $x\eth_{\dR}$ is that the former
is symmetric as an operator
\begin{equation*}
	x^{-1/2} L^2_{\iie,\Gamma}(X, \Iie\Lambda_\Gamma^*(X) ) 
	\to x^{1/2} L^2_{\iie,\Gamma}(X, \Iie\Lambda_\Gamma^*(X) )
\end{equation*}
(with respect to the natural pairing between the spaces on the right and left here), since $\eth_{\dR}$ is a symmetric 
operator on $L^2_{\iie,\Gamma}(X; \Iie\Lambda_\Gamma^*(X) )$ with core domain  $C^\infty_c$.

To analyze $\eth_{\dR}$ it is useful to consider the operator it induces on various weighted $L^2$ spaces.
For later use we point out first that $\eth_{\dR}$ satisfies  
\begin{equation}\label{DavsD0}
	\eth_{\dR}(x^a v) 
	= [\eth_{\dR}, x^a]v 
	+ x^a \eth_{\dR} v
	= x^a [ a \df e(\tfrac{dx}{x}) - a \df i(\tfrac{dx}{x}) + \eth_{\dR}] v,
\end{equation}
where $\df e$ and $\df i$ denote exterior and interior product respectively, and, second, that we have a unitary equivalence of 
unbounded operators \footnote{Note that in \cite{Hunsicker-Mazzeo}, for a stratification of depth one, $D_a$ denotes the 
de Rham operator of the complex $(x^aL^2_{\iie}, d)$ while here $D_a$ denotes the de Rham operator of the complex 
$(L^2_{\iie}, d)$ as an operator on $x^aL^2_{\iie}$.}
\begin{equation*}
\begin{gathered}
	\eth_{\dR} : x^a L^2_{\iie,\Gamma}(X, \Iie\Lambda_\Gamma^*(X) ) \to x^a L^2_{\iie,\Gamma}(X, \Iie\Lambda_\Gamma^*(X) ) \\
	\leftrightarrow
	D_a= x^{1/2-a} \eth_{\dR} x_0^{1/2+a} : 
	x^{-1/2} L^2_{\iie,\Gamma}(X, \Iie\Lambda_\Gamma^*(X) ) 
	\to x^{1/2} L^2_{\iie,\Gamma}(X, \Iie\Lambda_\Gamma^*(X) ).
\end{gathered}
\end{equation*}
 
In order to adapt arguments from \cite{Mazzeo:edge} it is more natural to work with the operator $x^{1/2-a} \eth_{\dR} x_0^{1/2+a}$ as an unbounded operator from the space 
$x^{-1/2} L^2_{\iie,\Gamma}(X, \Iie\Lambda_\Gamma^*(X) )$ {\em to itself.} Thought of in this way, we denote it as $P_a$,
\begin{equation}\label{PMap}
	P_a:
	x^{-1/2} L^2_{\iie,\Gamma}(X, \Iie\Lambda_\Gamma^*(X) ) 
	\to x^{-1/2} L^2_{\iie,\Gamma}(X, \Iie\Lambda_\Gamma^*(X) ).
\end{equation}

Our analysis of $\eth_{\dR}$ will proceed in two steps: in the first step we will analyze the behavior of $P_a$  
by adapting two model operators from \cite{Mazzeo:edge} -- the normal operator and the indicial family.
Then, in the second step, we will use the information gleaned about $P_a$ to analyze $\eth_{\dR}$.

\medskip
 \noindent
{\bf Remark.} 
{\em These two steps can be thought of in the following way. We  first analyze $x^{1/2} \eth_{\dR} x^{1/2}$ as a {\em partially complete} edge operator on $W$; complete in the
$(x,y)$ variables with values in $\iie$-operators on $Z$. Then, as a second step, we analyze it as an {\em incomplete} edge operator in the  $(x,y)$ variables with values, again, in $\iie$-operators on $Z$.}

\subsection{The normal operator of $P_a$} $ $\newline 
Recall that every point $q \in Y$ has a neighborhood $W$ which we identify with the product of $\cU\times C(Z)$, where
$\cU$ is a neighborhood of the origin in $\bbR^b \cong T_qY$. If this neighborhood is small enough that 
$\Iie\Lambda^*(X)\rest{W}$ can be identified with the pull-back of some vector bundle over $Z$ and similarly 
for $\Iielaga(X)|_{W}$, then we cal $W$  a {\bf basic neighborhood}. In such a $W$, let us fix smooth nonnegative 
cutoff functions $\chi$ and $\wt\chi$, both independent of the $Z$ variables, with supports in $W$ and equaling
one in a neighborhood of $q$, and such that $\wt\chi\chi = \chi$. We  refer to $W$, $\psi$, $\chi$, $\wt\chi$ as a 
{\bf basic setup} at $q \in Y$.

We can identify a basic neighborhood $W$ with a subset of the product of $Z$ with $T_qY^+ \cong \bbR^{+}_s \times  \bbR^{b}_u$ 
and use this identification to model the operator $P_a$ near $q$ by an operator on $Z \times T_qY^+$, the {\bf normal operator} 
of $P_a$ at $q \in Y$. 
Notice that the bundles $\Iie\Lambda^*(X)\rest{W}$, $\Iielaga(X)|_{W}$ as pull-backs of bundles over $Z$, extend naturally to $Z \times T_qY^+$, and that the dilation maps $R_t: T_qY^+ \to T_qY^+$ for any $t >0$ preserve the space of sections of these bundles.

\begin{definition}
The {\bf normal operator} 
$N_q(P_a)$ is the operator whose action on any $u \in \CIc( Z \times T_qY^+, \Iielaga (Z \times T_qY^+))$ is given by 
\begin{equation*}
	N_q(P_a)u = \lim_{r\to 0} R_r^* \; (\psi^{-1})^* \; \wt\chi \; P_a \; \psi^* \; \chi R^*_{1/r} u.
\end{equation*}
\end{definition}

Thus in local coordinates $(s,y,z)$ the action of the normal operator of $P_a$ on a section $u$ is obtained by evaluating $u$ at $(s/r, y/r, z)$, applying $P_a$, dilating back by a factor of $r$, and then letting $r \to 0$. 
It is easy to see that this procedure takes a vector field of the form $a(s, y, z) (s\pa_s) + b(s, y, z)(s\pa_y)$
to the vector field $a(0,0,z)(s\pa_s) + b(0,0,z)(s\pa_y)$, while for a vertical vector field $V$, this procedure returns $V\rest{s=0, y=0}$.
In fact, it is easy to see that this procedure replaces the metric 
\begin{equation*}
	g\rest{W} = g_{\cU}(x,y) + dx^2 + x^2 g_Z(x,y,z)
\end{equation*}
which is a submersion metric with respect to the projection $\cU \times C(Z) \to \cU$, with the product of an $\iie$ metric on $C(Z)$ and the flat metric on $\cU$,
\begin{equation*}
	g_{Z \times T_qY^+} = g_{\cU}(0,0) + ds^2 + s^2 g_Z(0,0,z).
\end{equation*}
It follows that any natural operator associated to $g_{\iie}$ is taken by this procedure to the corresponding natural operator of $g_{Z \times T_qY^+}$ in particular this is true for $\eth_{\dR}$.

\begin{lemma}
The normal operator of $P_a$ at $q \in Y$ is equal to $s^{1/2-a}\eth_{\dR}s^{1/2+a}$ where $\eth_{\dR}$ is the de Rham operator of the metric $g_{Z \times T_qY^+}$. Thus in local coordinates,
\begin{equation}\label{NormalPa}
	N_q(P_a) 
	= \begin{pmatrix}
		\eth_{\dR}^{Z} + s\eth_{\dR}^{\bbR^{b}}  & -s\pa_s - (f_0-\bN+a+1/2) \\
		s\pa_s + \bN + a +1/2 & - \eth_{\dR}^{Z} - s\eth_{\dR}^{\bbR^{b}}.
	\end{pmatrix}
\end{equation}
\end{lemma}

{\bf Remark.} {\em As explained above, this expression follows by naturality of the de Rham operator. Alternately, one can compute \eqref{NormalPa} directly from \eqref{DNearBdy}.}

\subsection{Localizing the maximal domain} $ $\newline 
The following lemma will allow us  to  ``localize the maximal domain" of $\eth_{\dR}$ near the singular locus.

\begin{lemma}\label{localize}
Let $W$, $\psi$, $\chi$, $\wt\chi$ be a basic setup at $q \in Y$.\\
Let  $u \in x^{-1/2} L^2_{\iie,\Gamma}(X;\Iielaga (X))$ be such that
 $P_a u  \in x^{1/2} L^2_{\iie,\Gamma}(X;\Iielaga (X)).$ Then  $\chi u \in x^{-1/2} L^2_{\iie,\Gamma}(X;\Iielaga (X))$ and $P_a(\chi u) 
  \in x^{1/2} L^2_{\iie,\Gamma}(X;\Iielaga (X))$.
\end{lemma}

\begin{proof}
Clearly $P_a(\chi u) = \chi (P_a u) + [P_a, \chi]u,$ and, since $\chi$ is independent of the $Z$-variables, 
\eqref{DNearBdy} allows us to see that $[P_a, \chi] = \sigma(P_a)(d\chi)= x H$ where $H$ is a multiplication operator 
by smooth bounded functions. Since $u \in x^{-1/2} L^2_{\iie,\Gamma}(X;\Iielaga (X))$ we see that 
$[P_a, \chi]u \in x^{1/2} L^2_{\iie,\Gamma}(X;\Iielaga (X))$, which establishes the lemma.
\end{proof}

\begin{proposition}\label{prop:upshot} Let $u \in x^{-1/2} L^2_{\iie}(X;\Iielaga(X))$ with compact support included in $W$  and such that $\chi =1$ 
on supp $u.$ Then  $P_a u \in x^{1/2} L^2_{\iie}(X;\Iielaga(X))$
 if and only if $N_q (P_a) (u \circ \psi^{-1}) \in s^{1/2} L^2_{\iie}(Z \times T_qY^+, \Iielaga (Z \times T_qY^+))$.
\end{proposition} 

\begin{proof} 
We prove only one implication, the other one is similar. Since we work in the distinguished chart $W$, we may identify $u$ with $u \circ  \psi^{-1}.$ 

Let $\rho$ denote a total boundary defining function. The operator $\frac{\rho}{x} P_a$ is an elliptic $\ie$ differential operator, so elliptic regularity \eqref{EllReg} yields $u \in x^{-1/2} H^1_{\iie}(X;\Iielaga(X)).$

We observe that, in the expression \eqref{PreDNearBdy},  $x(d^Y + \delta^Y_x)$ sends 
$x^{-1/2} H^1_{\iie}(X;\Iielaga(X))$ into $x^{1/2} L^2_{\iie}(X;\Iielaga(X)) $
and 
a similar observation is 
true for $s \eth_{\dR}^{\mathbb{R}^{b_0}}$, so using formulas \eqref{PreDNearBdy} and \eqref{NormalPa}, 
we get $P_a u - N_q (P_a) (u \circ \psi^{-1}) \in s^{1/2} L^2_{\iie},$ which proves the lemma.
\end{proof}

\subsection{Mapping properties of the normal operator} \label{sec:MapProp} $ $\newline 
We  shall  inductively assume that the signature operator on $Z$ has discrete spectrum.
Now assume that we also have:

\begin{equation}\label{Ass2}
	\boxed{ 
	\begin{array}{rcl}
	a)& \ &\Spec(\eth_{\dR}^{Z}) \cap (-1,1) \subseteq \{0\},  \\
	b)& \ &\mbox{If}\  k = \frac{f_0}2 \Mthen \mathcal{H}^k_{(2)}(Z)=0. 
	\end{array} }
\end{equation} 
Notice that $a)$  can be arranged by suitably scaling the metric on $Z$ while $b)$ is, by Theorem \ref{theo:cheeger}, 
a topological condition on $Z$.  

\begin{lemma}\label{lem:Bessel}
Let $a \in (0,1)$ and assume \eqref{Ass2} and that Theorem \ref{InductiveThm1} has been proven for $Z$, then $N(P_a)$ acting  on
$$
s^{-1/2} L^2_{\iie}(Z \times T_qY^+, \Iielaga (Z \times T_qY^+) )
$$ is injective on its maximal domain.
\end{lemma}

\begin{proof} 

Our plan is to take the first order system $N_q(P_a)w=0$ and reduce to a pair of uncoupled scalar, albeit fourth-order, equations.

The square of
\begin{equation*}
	N_q = N_q(P_a) 
	= \begin{pmatrix}
		\eth_{\dR}^{Z} + s\eth_{\dR}^{\bbR^b}  & -s\pa_s +\bN - f_0-a - 1/2 \\
		s\pa_s + \bN + a+1/2 & - \eth_{\dR}^{Z} - s\eth_{\dR}^{\bbR^b}
	\end{pmatrix},
\end{equation*}
is, with $\cL =(s\pa_s)^2 + (f_0+2a+1)s\pa_s - (\bN-f_0-a-1/2)(\bN+a+1/2)$,
\begin{equation*}
	N_q^2
	= \begin{pmatrix}
		\Delta^{Z} + s^2\Delta^{\bbR^b} -\cL  & -d^{Z} + \delta^{Z} \\
		d^{Z} - \delta^{Z} & \Delta^{Z} + s^2\Delta^{\bbR^b} - \cL
	\end{pmatrix}.
\end{equation*}
We take Fourier transform in the horizontal directions introducing the dual variables $\eta$ and we normalize these introducing the variables $t = s |\eta| $ and $\hat\eta =\frac{\eta}{|\eta|}$, then $\cL$ becomes
$\hat \cL = (t\pa_t)^2 + (f_0+2a+1)t\pa_t -(\bN-f_0-a-1/2)(\bN+a+1/2)$, and we get
\begin{equation*}
	N_q^2(\hat\eta)
	= \begin{pmatrix}
		\Delta^{Z} + t^2- \hat\cL  & -d^{Z} + \delta^{Z} \\
		d^{Z} - \delta^{Z} & \Delta^{Z} + t^2 - \hat\cL
	\end{pmatrix}.
\end{equation*}
Thus, if $(\alpha, \beta)$ are in the null space of $N_q^2(\hat\eta)$, they satisfy
\begin{equation}\label{Nq2System}
\begin{gathered}
	(\Delta^{Z} + t^2 - \hat\cL)\alpha = (d^{Z}-\delta^{Z})\beta \\
	(\Delta^{Z} + t^2 - \hat\cL)\beta = -(d^{Z}-\delta^{Z})\alpha.
\end{gathered}
\end{equation}

To analyze this system, we point out that $L^2$ forms on $Z$ satisfy a strong Kodaira decomposition, i.e., every $L^2$ form on $Z$ can be written in a unique way as a sum of a form in the image of $d^{Z}$, a form in the image of $\delta^{Z}$ and a form in the joint kernel of $d^{Z}$ and $\delta^{Z}$. 
As explained in \cite[\S2]{Hunsicker-Mazzeo}, {\em weak} Kodaira decompositions are a general feature of Hilbert complexes.
Inductively, we are assuming that $d+\delta$ is essentially self-adjoint and that its closed extension
has  closed range;  this implies, see  \cite[Proposition 4.6]{Hunsicker-Mazzeo}, that $d$ has a unique  
closed extension and that this extension
 has closed range (for instance, because $d$ coincides with $d+\delta$ on $(\ker d)^\perp$).
Hence the weak Kodaira decomposition is a strong Kodaira decomposition.

What is more, the operator $(\Delta^{Z} + t^2 - \hat\cL)$ preserves the Kodaira decomposition of forms on $Z$.
This claim is easy to check, since it is clearly true for $\Delta^{Z} + t^2 - (t\pa_t)^2 + (f_0+2a+1)t\pa_t$ and $(\bN-f_0-a-1/2)(\bN+a+1/2)$ satisfies
\begin{equation}\label{Commutation}
\begin{gathered}
	(\bN-f_0-a-1/2)(\bN+a+1/2) d = d (\bN-f_0-a+1/2)(\bN+a+3/2), \\
	(\bN-f_0-a-1/2)(\bN+a+1/2) \delta = \delta (\bN-f_0-a-3/2)(\bN+a-1/2)
\end{gathered}
\end{equation}
which implies that $\hat\cL$ induces maps
\begin{gather*}
	{\text{Image }d^{Z}}
	\to {\text{Image }d^{Z}}, \quad
	{\text{Image }\delta^{Z}}
	\to {\text{Image }\delta^{Z}}, \\
	\ker d^{Z} \cap \ker \delta^{Z}
	\to \ker d^{Z} \cap \ker \delta^{Z}.
\end{gather*}

Decomposing $(\alpha, \beta)$ in the null space of $N_q^2(\hat\eta)$ into
\begin{equation*}
	\alpha = d\alpha_1 + \delta \alpha_2 + \alpha_3, \quad \alpha_1 \in (\ker d)^\perp, \alpha_2 \in (\ker \delta)^\perp, \alpha_3 \in \ker d \cap \ker \delta
\end{equation*}
and similarly $\beta = d\beta_1 + \delta\beta_2 + \beta_3$, we find that the system \eqref{Nq2System} partially decouples into
\begin{gather*}
	\begin{cases}
	(\Delta^{Z} + t^2 - \hat\cL)d^{Z}\alpha_1 = d^{Z}\delta^{Z}\beta_2 \\
	(\Delta^{Z} + t^2 - \hat\cL)\delta^{Z}\beta_2 = \delta^{Z}d^{Z}\alpha_1
	\end{cases}
	\quad
	\begin{cases}
	(\Delta^{Z} + t^2 - \hat\cL)\delta^{Z}\alpha_2 = -\delta^{Z}d^{Z}\beta_1 \\
	(\Delta^{Z} + t^2 - \hat\cL)d^{Z}\beta_1 = -d^{Z}\delta^{Z}\alpha_2
	\end{cases}
	\\
	(t^2 - \hat\cL)\alpha_3 = 0, \quad
	(t^2 - \hat\cL)\beta_3 = 0.
\end{gather*}
These systems have the obvious symmetries $\alpha_3 \leftrightarrow \beta_3$ and $(\alpha_1, \beta_2) \leftrightarrow (\beta_1, -\alpha_2)$ so that we only need to solve half of the equations above.

The simplest is $(t^2 - \hat\cL)\alpha = 0$.
Notice that we can restrict to forms with vertical degree $k$, and also that the operator $(t^2 - \hat\cL)$ can be thought of as acting on the coefficients of $\alpha$ and so it suffices to consider this as a scalar equation.
Acting on the function $t^{-(f_0+2a+1)/2}h(t)$ we need to solve
\begin{gather*}
	t^{-(f+2a+1)}
	\lrpar{
	(t\pa_t)^2 - t^2 - \tfrac14(f_0 +2a+1)^2 - (k - f_0 - a -\tfrac12)(k+a+\tfrac12)}h =0 \\
	\iff
	t^{-(f+2a+1)}
	\lrpar{
	(t\pa_t)^2 - t^2 - \tfrac14(f_0 -2k)^2} h =0.
\end{gather*}
It follows, see [Lebedev, \S5.7], that $h$ is a linear combination of the modified Bessel functions of the first 
kind $I_\nu$ and MacDonald's function $K_\nu$ with $\nu$ given by
\begin{equation}\label{NuKer}
	\nu = \abs{ k - \frac{f_0}2 }.
\end{equation}
We point out that if condition ($b$) holds then $\nu \geq \frac12$.

Next we need to find the solutions to 
\begin{equation}\label{DecSystem}
	\begin{cases}
	(\Delta^{Z} + t^2 - \hat\cL)\alpha = d^{Z}\beta \\
	(\Delta^{Z} + t^2 - \hat\cL)\beta = \delta^{Z}\alpha
	\end{cases}
\end{equation}
when $\alpha \in { \text{Image } d^{Z}}$ and $\beta \in {\text{Image } \delta^{Z}}$.
Assuming $\beta = \delta^{Z}\beta_2$ with $\beta_2$ as above, the second equation implies that there is an element $\gamma \in \ker \delta^{Z}$ such that
\begin{equation*}
	(\Delta^{Z} + t^2 - \hat\cL_{\delta})\beta_2 = \alpha + \gamma
\end{equation*}
where $\hat \cL \delta^{Z} = \delta^{Z}\hat \cL_{\delta}$.
Combining this with the first equation we find (note that $d^{Z}\delta^{Z}\beta_2 = \Delta^{Z}\beta_2$ since $\beta_2 \in (\ker \delta)^{\perp}$)
\begin{gather*}
	(\Delta^{Z} + t^2 - \hat\cL_{\delta})(\Delta^{Z} + t^2 - \hat\cL)\alpha = \Delta^{Z}
	(\Delta^{Z} + t^2 - \hat\cL_{\delta})\beta_2
	= \Delta^{Z}(\alpha + \gamma) \\
	\iff
	\lrspar{
	(\Delta^{Z} + t^2 - \hat\cL_{\delta})(\Delta^{Z} + t^2 - \hat\cL) - \Delta^{Z} }\alpha = \Delta^{Z}\gamma = \delta^{Z}d^{Z}\gamma
\end{gather*}
The left-hand side of this equation is in ${\text{Image }d^{Z}} = (\ker \delta^{Z})^\perp$ while the right-hand side is in $\ker \delta^{Z}$, hence both sides must vanish.
Similar reasoning assuming $\alpha = d\alpha_1$ and with $\hat \cL_d$ defined through $\hat \cL d = d \hat \cL_d$ shows that
\begin{equation*}
	\lrspar{
	(\Delta^{Z} + t^2 - \hat\cL_{d})(\Delta^{Z} + t^2 - \hat\cL) - \Delta^{Z} }\beta = 0
\end{equation*}
hence solutions of \eqref{DecSystem} satisfy
\begin{equation}\label{DecSystemCon}
\begin{gathered}
	\lrspar{
	(\Delta^{Z} + t^2 - \hat\cL_{\delta})(\Delta^{Z} + t^2 - \hat\cL) - \Delta^{Z} }\alpha = 0, \\
	\lrspar{
	(\Delta^{Z} + t^2 - \hat\cL_{d})(\Delta^{Z} + t^2 - \hat\cL) - \Delta^{Z} }\beta = 0.
\end{gathered}
\end{equation}

To solve the first equation in \eqref{DecSystemCon}, notice that we can assume that $\alpha$ has vertical degree $k$ and that it is in an eigenspace of $\Delta^{Z}$, say with eigenvalue $\lambda^2$, which is not equal to zero since $\alpha$ is orthogonal to $\ker d^{Z} \cap \ker \delta^{Z}$. Then as before the equation is a scalar equation acting on the coefficients of $\alpha$ and so it suffices to consider its action on scalar functions.
From \eqref{Commutation}, we see that
\begin{equation*}
	\hat\cL_{\delta} = 
	(t\pa_t)^2 + (f_0+2a+1)t\pa_t -(\bN-f_0-a-3/2)(\bN+a-1/2)
\end{equation*}
and note that
\begin{equation*}
	\hat\cL_{\delta} t^{-(f_0+2a+1)/2}h(t)
	= t^{-(f_0+2a+1)/2} \lrpar{ (t\pa_t)^2 - (\bN-1-\tfrac{f_0}2)^2}h(t)
\end{equation*}
so the equation we need to solve is
\begin{equation*}
	\lrspar{
	\lrpar{ (t\pa_t)^2 -t^2 - \lambda^2 - (\bN-1-\tfrac{f_0}2)^2}
	\lrpar{ (t\pa_t)^2 -t^2 - \lambda^2 - (\bN-\tfrac{f_0}2)^2}
	-\lambda^2 } h(t) = 0.
\end{equation*}
We can write this as
\begin{equation*}
	\lrpar{ (t\pa_t)^2 - t^2  - \mu_+ }
	\lrpar{ (t\pa_t)^2 - t^2  - \mu_- } h(t) =0
\end{equation*}
where $\mu_{\pm}$ are given by
\begin{equation}\label{MuAlpha}
	\mu_{\pm} 
	=\lambda^2 + \tfrac12(\bN-1-\tfrac{f_0}2)^2 + \tfrac12(\bN-\tfrac{f_0}2)^2
	\pm  \sqrt{
	(\bN - \tfrac{f_0 + 1}2)^2 + \lambda^2 }
\end{equation}
and if $h$ is a solution then it 
can be written as a linear combination of $I_{\nu_+}$, $I_{\nu_-}$, $K_{\nu_+}$, $K_{\nu_-}$ where $\nu_\pm = \sqrt {\mu_\pm}$.

It will be useful to know how small $\mu_{\pm}$ can be, recalling that for these solutions we have $\lambda^2 \neq 0$.
Using condition ($a$), it is clear that $\mu_+ \geq 1$.
For a fixed $\lambda$, it is easy to see that the minimum of $\mu_-$ occurs when $\bN = \tfrac{f+1}2$ where we have
\begin{equation*}
	\mu_- \geq \lambda^2 + \tfrac14 - \lambda \geq \tfrac14
\end{equation*}
where we have again made use of condition ($a$).
In either case we have $\nu_{\pm} \geq \frac12$.

Finally, the second equation in \eqref{DecSystemCon} can be solved in the same fashion.
Here, since
\begin{equation*}
	\hat\cL_d 
	= (t\pa_t)^2 + (f_0+2a+1)t\pa_t - (\bN-f_0 - a +\tfrac12)(\bN+a+\tfrac32)
\end{equation*}
satisfies
\begin{equation*}
	\hat \cL_d t^{-(f_0 +2a+1)/2}h(t)
	= t^{-(f_0+2a+1)/2}\lrspar{ (t\pa_t)^2 -(\bN+1 - \tfrac {f_0}2)^2 }h
\end{equation*}
We ultimately need to solve the following equation
\begin{equation*}
	\lrspar{
	\lrpar{ (t\pa_t)^2 -t^2 - \lambda^2 - (\bN+1-\tfrac{f_0}2)^2}
	\lrpar{ (t\pa_t)^2 -t^2 - \lambda^2 - (\bN-\tfrac{f_0}2)^2}
	-\lambda^2 } h(t) = 0.
\end{equation*} We can rewrite this equation as:
\begin{equation} \label{eq:Null}
\bigl( (t\pa_t)^2 -t^2 -\wt \mu_+  \bigr) \bigl( (t\pa_t)^2 -t^2 -\wt \mu_-  \bigr) h(t)=0,
\end{equation} where
\begin{equation}\label{MuBeta}
	\wt\mu_{\pm}
	=\lambda^2 + \tfrac12(\bN+1-\tfrac{f_0}2)^2 + \tfrac12(\bN-\tfrac{f_0}2)^2
	\pm  \sqrt{
	(\bN - \tfrac{f_0 -1}2)^2 + \lambda^2 }.
\end{equation}
Since $\lambda \neq 0$, any solution of \eqref{eq:Null} can be written 
as a linear combination of $I_{\nu_+}$, $I_{\nu_-}$, $K_{\nu_+}$, $K_{\nu_-}$ where $\nu_\pm = \sqrt {\wt \mu_\pm}$.
As before, it is easy to check that $\nu_{\pm}$ is always at least $\frac12$.

Having solved all of the requisite equations, we next consider which of these solutions are in $t^{-1/2}L^2(t^{f_0} \; dt)$.
First recall that $I_\nu(t)$ always blows-up as $t \to \infty$, so none of these can be in this space.
On the other hand, $K_\nu(t)$ behaves like $e^{-t}$ for large $t$ and like $t^{-\nu}$ for $t$ near $0$.
Next, it is easy to see that a function of the form $t^s$ near $t=0$ is in $t^{-1/2}L^2_{\loc}(t^{f_0} \; dt)$ when
\begin{equation*}
	2s+1+f_0 > -1
	\iff 2s > -f_0 -2
\end{equation*}
So to {\em avoid} having $t^{-(f_0+2a+1)/2}K_\nu(t) \in t^{-1/2}L^2_{\loc}(t^{f_0} \; dt)$ it is enough to ensure
\begin{equation}\label{NecCond}
	- f_0 - 2a - 1 - 2\nu \leq - f_0 -2
	\iff
	2a+2\nu \geq 1.
\end{equation}
As pointed out above, conditions ($a)$ and ($b)$) guarantee that $\nu \geq \frac12$, hence $N_q(P_a)$ is injective whenever $a>0$.
\end{proof}

\subsection{Indicial roots}$ $\newline
Another model operator of $P_a$ is its {\bf indicial family}, defined using the action of $P_a$ on polyhomogeneous expansions.
The indicial family is a one parameter family of operators on $Y$, $I(P_a;\zeta)$ defined by
\begin{equation*}
	P_a ( x^\zeta f) = x^\zeta I(P_a; \zeta)f\rest{x=0} + \cO(x^{\zeta+1}).
\end{equation*}
The base variables at the boundary enter into the indicial family as parameters, so we can speak of the indicial family at the point $q \in Y$ by restricting not just to $x =0$ but to the fibre over $q$. We denote this refinement of the indicial family by $I_q(P_a;\zeta)$, from 
\eqref{DNearBdy} we see that it is given by
\begin{equation*}
	I_q(P_a; \zeta) 
	= \begin{pmatrix}
		\eth_{\dR}^{Z}   & -\zeta - (f_0-\bN+a+1/2) \\
		\zeta + \bN + a +1/2 & - \eth_{\dR}^{Z}.
	\end{pmatrix}
\end{equation*}
and we point out that it coincides with the indicial family of the normal operator at $q \in Y$.
The values of $\zeta$ for which $I_q(P_a;\zeta)$ fails to be invertible (on $L^2_{\iie}(Z)$) are known as the {\bf indicial roots} of $P_a$ at $q$, or the boundary spectrum of $P_a$ at $q$, and are denoted
\begin{equation*}
	\spec_b(N_q( P_a )).
\end{equation*}

A final model operator of $P_a$ is the {\bf indicial operator} of $P_a$, $I_q(P_a)$, given by
\begin{equation*}
	I_q(P_a) 
	= \begin{pmatrix}
		\eth_{\dR}^{Z}   & -t\pa_t - (f_0-\bN+a+1/2) \\
		t\pa_t + \bN + a +1/2 & - \eth_{\dR}^{Z}.
	\end{pmatrix}
\end{equation*}
It is related to the indicial family by the Mellin transform,
\begin{equation*}
	\cM( I_q(P_a) u)(\zeta)
	= I_q(P_a; -i\zeta) \cM(u)(\zeta).
\end{equation*}
Recall that this transform is defined, e.g., for $u \in \CIc( \bbR^+ )$ by
\begin{equation}\label{DefMellin}
	\cM u(\zeta) = \int_0^\infty u(x) x^{i\zeta - 1} \; dx,
\end{equation}
and extends to an isomorphism between weighted spaces
\begin{equation}\label{MellinIso}
	x^\alpha L^2 \left(\bbR^+, \frac{dx}x \right) \xrightarrow{\cong} L^2\lrpar{ \{ \eta = \alpha \}; d\xi }
\end{equation}
where $\eta = \Im\zeta$ and $\xi = \Re\zeta$.
The inverse of the Mellin transform as a map \eqref{MellinIso} is given by
\begin{equation*}
	\cM^{-1}(v)(x) = \frac1{2\pi} \int_{\eta = \alpha} v(\zeta) x^{-i\zeta} \; d\xi.
\end{equation*}

\begin{lemma}\label{lem:Indicial}
The indicial roots of $P_a$ are contained in 
\begin{equation}\label{IndSets}
\begin{gathered}
	\lrbrac{ -  \frac{f_0 + 2a +1 }2 \pm \sqrt{ \mu_{\pm} } } \bigcup
	\lrbrac{ -  \frac{f_0 + 2a +1 }2 \pm \sqrt{ \wt \mu_{\pm} } } \\
	\bigcup \lrbrac{ -  \frac{f_0 + 2a +1 }2 \pm \lrpar{ k - \frac{f_0}2 } }.
\end{gathered}
\end{equation}
where $k \in \{0, \ldots, f_0 \}$ and $\mu_{\pm}$, $\wt\mu_{\pm}$ are defined in \eqref{MuAlpha} and \eqref{MuBeta} respectively.

The indicial operator of $P_a$ is has a bounded inverse on the space $t^{-1/2}L^2_{\iie}(Z \times \bbR^+_t)$ for all $a \in (0,1)$. 
\end{lemma}

\begin{proof}
An analysis similar to -- but simpler than -- that above applies to the equation $I_q(P_a)u=0$. Indeed, it suffices to replace $(t\pa_t)^2 - t^2$ in the `equations to solve' by $(t\pa_t)^2$. 
Since the solutions to $(t\pa_t)^2u = \nu^2u$ are linear combinations of $t^{\nu}$ and $t^{-\nu}$, the solutions of $I_q(P_a)u=0$ are obtained from the solutions to $N_q(P_a)u =0$ by replacing each $I_v(t)$ by $t^\nu$ and each $K_\nu(t)$ by $t^{-\nu}$.
As before, asking for the solutions to be in $t^{-1/2}L^2(t^f \; dt)$ excludes those involving $t^\nu$, and then conditions ($a$) and ($b$) show that there are no solutions involving $t^{-\nu}$ for $a>0$.
Thus the proof of Lemma \ref{lem:Bessel} shows that the indicial operator $I_q(P_a)$ is injective on $t^{-1/2}L^2(Z \times \bbR^+)$ as long as $a>0$.

Similarly, the proof of Lemma \ref{lem:Bessel} shows that if there is a non-zero solution to
$I_q(P_a;\zeta)u=0$ then $\zeta$ must be in one of the sets in \eqref{IndSets}.
An advantage of the indicial family is that we can bring to bear our inductive hypotheses about $\eth_{dR}^Z$.
Indeed, decompose $I_q(P_a)$ as 
\begin{multline*}
	I_q( P_a )(\zeta)
	= \begin{pmatrix}
		\eth_{\dR}^{Z_0}   & -\zeta - (f_0-\bN+a+1/2) \\
		\zeta + \bN + a+1/2 & - \eth_{\dR}^{Z_0}
	\end{pmatrix} \\
	= \eth_{\dR}^{Z_0}
	\begin{pmatrix} \Id & 0 \\ 
	0 & - \Id \end{pmatrix}
	+ \begin{pmatrix}
		0   & -\zeta - (f_0-\bN+a+1/2) \\
		\zeta + \bN + a+1/2 & 0
	\end{pmatrix}
	= A + B.
\end{multline*}
Inductively we know that $A$ is essentially self-adjoint, has closed range, and its domain, $\cD(A)$, includes compactly into $L^2_{\iie}(Z)$.
It follows that the operator $B : \cD(A) \to L^2_{\iie}(Z)$ (where $\cD(A)$ is endowed with the graph norm) is compact, i.e., $B$ is relatively compact with respect to $A$, and so
$I_q(P_a;\zeta)$ has a unique closed extension, has closed range, and its domain is also $\cD(A)$.

Since $\eth_{\dR}^Z$ is essentially self-adjoint, the adjoint of $I_q( P_a )(\zeta)$ on $L^2(Z)$ is
\begin{equation*}
\begin{split}
	I_q( P_a; \zeta)^*
	&= \begin{pmatrix}
		\eth_{\dR}^{Z_0}   & \zeta + \bN + a+1/2 \\
		-\zeta - f + \bN -a-1/2) & - \eth_{\dR}^{Z_0}
	\end{pmatrix} \\
	&= I_q(P_a; -(\zeta + f + 2a+1) )
\end{split}
\end{equation*}
Notice that $\zeta$ is in one of the sets in \eqref{IndSets} if and only if $-(\zeta + f + 2a +1)$ is.
Thus we see that if $\zeta$ is not in one of these sets, then $I_q(P_a; \zeta)$ is in fact invertible with bounded inverse. In fact, since the domain of $I_q(P_a; \zeta)$ is $\cD(A)$, its inverse is a compact operator.
This proves that \eqref{IndSets} contains the indicial roots of $N_q(P_a)$.
Denote the inverse of $I_q(P_a; \zeta)$ by 
\begin{equation*}
	Q(\zeta) : L^2_{\iie}(Z) \to \cD(A) \hookrightarrow L^2_{\iie}(Z).
\end{equation*}
We obtain an inverse for $I_q(P_a)$ as an operator on $t^{-1/2}L^2_{\iie}(R^+ \times Z)$ by applying the inverse Mellin transform to $Q(\zeta)$ along the line $\eta = -\frac f2 -1$, which we can do as long as $-\frac f2 -1$ is not an indicial root.
If ($a$) and ($b$) hold, then this is true for all $a \in (0,1)$.
\end{proof}

\subsection{Closed range properties of $N_q( P_a)$.} $ $\newline 
Using the formal transpose, we shall obtain conditions for $N_q( P_a)$ to have dense range 
Since $\eth_{\dR}$ is symmetric on $L^2_{\iie}(X; \Iie\Lambda^*X)$, the operator
\begin{equation*}
	D_0 : x^{-1/2}L^2_{\iie}(X;\Iie\Lambda^*X) \to x^{1/2}L^2_{\iie}(X; \Iie\Lambda^*X)
\end{equation*}
also coincides with its formal adjoint.
A simple computation then shows that the formal adjoint of $P_0$ is
\begin{equation*}
	(P_0)^* = x^{-1/2} \eth_{\dR} x^{3/2} : 
	x^{-1/2}L^2_{\iie} (X;\Iie\Lambda^*X) 
	\to 
	x^{-1/2}L^2_{\iie} (X;\Iie\Lambda^*X), 
\end{equation*}
and similarly, 
\begin{equation}\label{PaTranspose}
	(P_a)^*
	= (x^{1/2-a} \eth_{\dR} x^{1/2+a})^*
	= x^{-1/2+a} \eth_{\dR} x^{3/2-a}
	=P_{1-a}.
\end{equation}
Observe that $N_q(P_a)$ is a closed unbounded operator and that its (maximal) domain is dense.
Thus from Lemma \ref{lem:Bessel} we see that, if conditions $a)$ and $b)$ hold, $N_q(P_a)$ will have dense range if $a <1$.
In the following lemma we will show that the range of $N_q(P_a)$ is closed for $a \in (0,1)$ and hence this operator is bijective.

\begin{lemma} \label{lem:NClosedRange}
The normal operator $N_q(P_a)$ has closed range as an operator on $t^{-1/2}L^2_{\iie}(Z \times T_qY^+)$ acting on its maximal domain, for all $a \in (0,1)$.
\end{lemma}

\begin{proof}
Recall that an operator is {\em essentially injective} if it has closed range and finite dimensional null space. An operator $L$ on a non-compact manifold is essentially injective if there is a compact set $K$ and a constant $C>0$ such that
\begin{equation}\label{GoalIneq}
	\norm{Lu} \geq C \norm{u}, \text{ for every $u$ supported outside } K,
\end{equation}
We will show that $N_q(P_a)$ has closed range for $a \in (0,1)$ by showing that it is essentially injective. 
For the duration of the proof, we simplify notation by omitting the bundle 
$\Iie\Lambda^* (Z \times T_qY^+)$.

As in the proof of \eqref{lem:Bessel}, it is useful to analyze $N_q(P_a)$ by carrying out a Fourier transform of the horizontal variables, introducing the new variable $\eta$ dual to $y$, and then making the change of variable $t = s|\eta|$, $\hat \eta = \eta/ |\eta|$. The resulting operator is given by
\begin{equation*}
	\wt N_q(\hat \eta)
	= \begin{pmatrix}
		\eth_{\dR}^{Z_0} + t\cl{\hat\eta}   & -t\pa_t - (f_0-\bN+a+1/2) \\
		t\pa_t + \bN + a+1/2 & - \eth_{\dR}^{Z_0} - t\cl{\hat\eta}
	\end{pmatrix}
\end{equation*}
To prove \eqref{GoalIneq} for $N_q(P_a)$ it suffices to prove that, for some compact set $K$ and constant $C>0$ (independent of $|\eta|$),
\begin{equation*}
	\norm{ \wt N_q(\hat \eta) u }_{L^2(Z_0 \times \bbR^+_t, t^{f_0} \; dt \; dvol_Z)} \geq C \norm{ u }_{L^2(Z_0 \times \bbR^+, t^{f_0} \; dt \; dvol_Z)}.
\end{equation*}

We will prove this separately for $u$ supported near $t =0$ and $u$ supported on $t \gg 0$.
For the former, we will show that $\wt N_q(\hat \eta)$ has a bounded left inverse on
\begin{equation*}
	L^2_{\Dir}(Z_0 \times (0,1]_t),
\end{equation*}
the closure of the set of $u$ in $L^2(Z_0 \times \bbR^+_t, t^{f_0} \; dt \; dvol_Z)$ with support in $\{ t <1 \}$.

From Lemma \ref{lem:Indicial}, we know that the indicial operator $I_q(P_a) = \wt N_q(\hat\eta)\rest{t=0}$ has a bounded inverse, $H_0$, on $t^{-1/2}L^2_{\iie}(Z \times \bbR^+)$.
On the other hand, the difference between $\wt N_q(\hat\eta)$ at $t=0$ and at $t>0$ is given by $t\cl{\hat\eta}\begin{pmatrix} \Id & 0 \\ 0 & -\Id \end{pmatrix}$.
Notice that, since $Q(\zeta)$ in the proof of Lemma \ref{lem:Indicial} is always a compact operator on $L^2_{\iie}(Z)$, the operator $H_0t$ is a compact operator on $L^2_{\Dir}(Z_0 \times (0,1]_t)$.
Thus $H_0 \wt N_q(\hat\eta)$ is equal to the identity plus a compact operator.
Since we know that $\wt N_q(\hat\eta)$ is injective, this shows that it has a bounded left inverse when acting on sections supported near $t =0$.

To handle the setting $t \gg 0$, note that since the Fourier transform is an isometry, it suffices to work with $N_q^2(\hat\eta)$.
For this operator, as in  \cite[Lemma 5.5]{Mazzeo:edge}, we consider the partial symbol 
\begin{equation*}
	\wt\sigma(N_q^2)
	=
	\begin{pmatrix}
	\Delta^{Z_0} + t^2 + t^2\tau^2 & 0 \\ 0 & -(\Delta^{Z_0} + t^2 + t^2\tau^2)
	\end{pmatrix}
\end{equation*}
where $\tau$ is a variable dual to $\pa_t$.
We clearly have
\begin{equation*}
	|\ang{ \wt \sigma(N_q^2)u, u }| \geq t^2(1+\tau^2)\norm{u}
\end{equation*}
where $\ang{\cdot, \cdot}$ and $\norm{\cdot}$ are taken in $t^{-1/2}L^2(Z \times \bbR^+_r \times \bbR_\tau; t^{f_0} \dvol_Z \; dt\; d\tau)$.
Hence the operator norm of $\wt \sigma(N_q^2)^{-1}$ is bounded by $t^{-2}(1+\tau)^{-2}$, and so the operator 
\begin{equation*}
	H_{\infty}(u) = \int e^{it\tau} \wt\sigma(N_q^2)^{-1} \hat u \; d\tau
\end{equation*}
serves as a bounded left inverse for $N_q^2(\hat \eta)$ for large $t$, and hence $H_{\infty}\wt N_q(\hat\eta)$ serves as a bounded left inverse for $\wt N_q(\hat\eta)$.
\end{proof}

\subsection{Integration by parts identity for $N_q(P_a)$} $ $\newline
In computing the indicial roots of $P_a$, we have made strong use of the symmetries of the normal operator of $P_a$, namely the translation invariance along horizontal directions (i.e., those tangent to $Y$) and dilation invariance in $T_qY^+$. 
In this section we exploit this invariance to establish an integration by parts identity, which will ultimately allow us to show that any `extra' vanishing of $N_q(P_a)u$ at $x=0$ translates to some degree of vanishing of $u$ at $x=0$, the latter degree bounded by the indicial roots of $N_q(P_a)$.

We will need the Sobolev spaces on $Z \times T_q Y^+$ analogous to those on $X$.
\begin{definition} \label{def:partial} Let $N \in \mathbb{N}.$  We define 
$H_{\pie}^{N}(Z \times T_qY^+;\Iie\Lambda^* )$ to be the set of 
$u \in L^2_{\iie}(Z \times T_qY^+;\Iie\Lambda^* )$ such that 
for any positive integer $p \leq N, $
$$
X_1 \ldots X_p u \in L^2_{\iie}(Z \times T_qY^+;\Iie\Lambda^* )
$$ where the $X_j$ are vector fields which are either of the form $s \partial_s, s\partial_{u_j} \, ( 1\leq j \leq b_0)$ 
or of the form $X(z,s,u)=X(z)$ for each $(z,s,u) \in Z \times T_qY^+,$ where $X(z)$ is an edge vector field of the fibre $Z=Z_{q}.$ Notice that these vectors fields $s \partial_s, s\partial_{u_j} \,X(z)$ generate 
a Lie algebra.
\end{definition}

As we have already used in \S\ref{sec:MapProp},
if a function in $L^2_{\iie}(X)$ is $\cO(x^\gamma)$ near $x=0$ then we must have $2\gamma+f_0 >-1$. As the $L^2$ cut-off will be very important below we introduce the function
\begin{equation}\label{Defdelta}
	\delta_0(\gamma) = \gamma - \frac{f_0+1}2,
\end{equation}
thus a function in $\cO(x^{\gamma})$ is in $x^aL^2(x^{f_0 }\; dx )$ precisely when $\gamma>\delta_0(a)$. 

Briefly, let us abbreviate $L^2_{\iie}( Z \times T_qY^+, \Iie \Lambda^*( Z \times T_qY^+))$ by $L^2_{\iie}(q)$.
Let $C$ be a fixed number in $[-1/2,1]$ and $\eps \in (0,1).$  Let now $R$ be the unbounded operator induced by $N_q(P_0)$ on $s^{C +\eps}L^2_{\iie}( q )$ with domain $C^\infty_c$; with a small abuse of notation
we denote also by $R$ the operator induced by $N_q(P_0)$ on $s^{C -\eps}L^2_{\iie}( q )$ (acting distributionally).
We consider the natural pairing 
$\ang{\cdot, \cdot}: s^{C +\eps}L^2_{\iie}( q )\times s^{C -\eps}L^2_{\iie}( q )\to \bbC$
between these two spaces
\footnote{Recall that this pairing is given by $<u,v>:= (u',v')_{s^C L^2}$ if $u=s^\epsilon u'$ and $v=s^{-\epsilon} v'$.}.
Let $R^t$ be the formal transpose of $R$ with respect to this pairing. $R^t$ is a differential operator and  we let it act, distributionally, 
on $s^{C-\eps} L^2_{\iie}( q )$.

We will establish that, if
\begin{equation*}
	u \in s^{C}L^2_{\iie}( q ), \quad 
	v \in s^{C-\eps}L^2_{\iie}( q ), \quad
	\Mand Ru, \;R^t v\; \in \; s^{C+\eps}L^2_{\iie}( q )
\end{equation*}
then, with respect to the natural pairing $\ang{\cdot, \cdot}$ above, we have 
$$\ang{v, Ru} = \ang{u, R^t v}.$$
Notice that, although both pairings make sense, this is not an instance of the definition of $R^t$, since both $u$ and $v$ are thought of as elements of $s^{C-\eps}L^2_{\iie}( q )$.

Assume inductively that we have shown $\cD_{\max}(\eth_{\dR}) = \cD_{\min}(\eth_{\dR})$ for stratifications of depth at most $m-1$ so that in particular
\begin{equation*}
	\ang{ \eth_{\dR}^{Z}u, v } = \ang{u, \eth_{\dR}^{Z}v}
\end{equation*}
for any two elements of $\cD_{\max}(\eth_{\dR}^{Z})$.

On the one hand we know that, for $u, v \in s^{C}L^2_{\iie}( q )$, the natural inner product is given by
\begin{equation*}
	\ang{u,v} = \int s^{-2C} u \wedge * v
\end{equation*}
and, on the other, the normal operator is given by
\begin{equation*}
	N_q = N_q(P_a) 
	= \begin{pmatrix}
		\eth_{\dR}^{Z} + s\eth_{\dR}^{\bbR^b}  & -s\pa_s +\bN - f- (a+1/2) \\
		s\pa_s + \bN + (a+1/2) & - \eth_{\dR}^{Z_0} - s\eth_{\dR}^{\bbR^b}
	\end{pmatrix},
\end{equation*}
so as anticipated we only have to justify integrating by parts the $s\pa_s$ and $s \eth_{\dR}^{\bbR^b}$.

We can  assume that we are working with sections compactly supported in a basic neighborhood $W.$

Our main tool is the Mellin transform \eqref{DefMellin}.
Using the inclusions $x^aL^2 \subset x^bL^2$ whenever $b<a$ it follows that the Mellin transform of a function in
$x^aL^2(\bbR^+, dx)$ is holomorphic in the half-plane $\{ \eta < a-1/2 \}$.
The Mellin transform is very useful for studying asymptotics. For instance, if $u$ is polyhomogeneous then $\cM u$ extends to a meromorphic function on the whole complex plane with poles at locations determined by the exponents occuring in the expansion of $u$.
Switching from $L^2(\bbR^+)$ to $L^2_{\iie}(X)$, assume that $\omega$ is supported in a basic neighborhood $W$ of $q \in Y$, then we have
\begin{equation*}
	\omega \in s^\alpha L^2_{\iie}(X)
	\iff \cM \omega \in L^2\lrpar{ \{ \eta = \delta_0(\alpha_0) \}, d\xi; L^2 (dy \; \dvol_{Z}) }
\end{equation*}
where $\cM$ denotes Mellin transform in $s$ (in the usual coordinates),  $d y$ denotes the Lebesgue 
measure of $\mathbb{R}^{b_0}$, and $\dvol_{Z}$ denote the volume form  associated to 
the iterated conic metric of $Z.$
Notice that $\cM\omega$ extends to a holomorphic function on the half-plane $\{ \eta < \delta_0(\alpha_0) \}$ with values in $L^2(dy \; \dvol_{Z})$.

Elliptic regularity (via the symbolic calculus) tells us that elements in the null space of an elliptic edge-operator are in $H^\infty_{\iie }(X;\Iie \Lambda^*)$, and hence smooth in the interior of the manifold. However, the derivatives of elements in $H^{\infty}_{\iie}(X;\cE)$ will typically blow-up at the boundary, which is just to say that knowing $\rho \pa_y u \in L^2_{\iie}(X; \Iie \Lambda^*)$ tells us that $\pa_y u \in \rho^{-1} L^2_{\iie}(X;\Iie \Lambda^*)$.
Using the Mellin transform we can turn this around: 
if $u$ is in the null space of an elliptic $\ie$-operator, $A$, as a map
\begin{equation*}
	A: \rho^{\alpha}L^2_{\iie}(X; \Iie \Lambda^*) \to
	\rho^{\alpha}L^2_{\iie}(X; \Iie \Lambda^*)
\end{equation*}
then, in the absence of indicial roots, we can view $u$ as an element of a space with a stronger weight at the cost of giving up tangential regularity at the boundary. We shall concentrate directly on the normal operator of $P_a$, even though
much of what we prove could be extended to more general differential operators.

\begin{lemma} \label{lem:integration}
Let $W$ be a basic neighborhood for the point $q \in Y$. Set $R= N_q(P_a)$ and assume 
 that, for some $\alpha \in \bbR$ and $\eps \in(0 , 1)$,
\begin{equation}\label{SymReq}
	\{ \Re \zeta + \tfrac f2 + \tfrac12: \zeta \in \spec_b(  R ) \} \cap [\alpha-\eps, \alpha+ \eps] \subseteq \{ \alpha \}.
\end{equation}

\begin{enumerate}
\item \label{lem:Asympt}
Assume $v \in s^\alpha L^2_{\iie}(Z \times T_qY^+ ; \Iie \Lambda^*)$ is supported in $W$ and
$R v \in s^{\alpha+\eps}L^2_{\iie}(Z \times T_qY^+ ; \Iie \Lambda^*)$
then
\begin{equation*}
\begin{split}
	v &\in s^{\alpha+\eps} L^2( s^{f_0} \;d s\;dvol_{Z},  H^{-1}(dy) \otimes \Iie \Lambda^*)  \\
	&= \{ s^{\alpha+\eps} u: u \in \Diff^{1}(Y) L_{\iie}^2(W ; \Iie \Lambda^*) \}
\end{split}
\end{equation*}
Moreover, as a map into $L^2( \dvol_Z, H^{-1}(dy) \otimes \Iie \Lambda^* )$ the Mellin transform of $v$ is holomorphic in the half-plane 
$\{\eta < \delta_0( \alpha+\eps) \}$.

\item \label{lem:IntByParts}
Assume that 
$u \in s^\alpha L^2_{\iie}(Z \times T_qY^+ ; \Iie \Lambda^*)$ and 
$w \in s^{\alpha-\eps}H^2_{\pie}(Z \times T_qY^+ ; \Iie \Lambda^*)$ (cf Definition \ref{def:partial}) are such that
\begin{equation*}
\begin{gathered}
	\supp u \subseteq W \\
Ru, R^tw \in s^{\alpha+\eps}L^2_{\iie}(Z\times T_qY^+ ; \Iie \Lambda^*),	
\end{gathered}
\end{equation*}
then with respect to the natural pairing
\begin{equation*}
	\ang{\cdot, \cdot}: s^{\alpha-\eps}L^2_{\iie}(Z\times T_qY^+ ; \Iie \Lambda^*) \times s^{\alpha+ \eps} L^2_{\iie}(Z \times T_qY^+ ; \Iie \Lambda^*) \to \bbC
\end{equation*}
we have $\ang{ w, Ru } = \ang{ u, R^t w }$.
\end{enumerate}
\end{lemma}

\begin{proof}
{\bf (1):}
Since $v$ is supported in a normal neighborhood of $q \in Y$, we can write 
\begin{equation*}
	I_q( R ) v = H v + h
\end{equation*}
where $I_q(R)$ is the indicial operator of $R$ and $H$ contains all of the `higher order terms' at the boundary, e.g., $s^2\pa_s$, $s\pa_u$. 

Passing to the Mellin transform, and using that $I_q(R;\zeta)$ depends polynomially on $\zeta$, we have an equality
\begin{multline}\label{Hol1}
	\cM v (\zeta)= I(R; i\zeta)^{-1} \lrpar{\cM (H v + h) (\zeta)} \\
	\text{ as meromorphic functions }
	\{ \eta < \delta_0(\alpha) \} \to L^2(dy \; \dvol_{Z} ;  \Lambda^*),
\end{multline}
of course since the left hand side is holomorphic on this half-plane so is the right hand side.
On the other hand, $\cM(h)$ is a holomorphic function into this space on the half plane $\{ \eta < \delta_0(\alpha +\eps) \}$, and, reasoning as in \cite{Mazzeo:edge},
$\cM( H v)$ extends holomorphically to this half plane but we have to give up tangential regularity,
\begin{equation*}
	\cM( H v) :
	\{ \eta < \delta_0(\alpha+\eps) \} \to L^2(\dvol_{Z}; H^{-1}(dy) \otimes  \Lambda^*) 
	\text{ holomorphically}.
\end{equation*}
This gives us an extension of \eqref{Hol1} to
\begin{multline}\label{Hol2}
	\cM v (\zeta)= I( R; i\zeta)^{-1} \lrpar{\cM ( H v + h) (\zeta)} \\
	\text{ as meromorphic functions }
	\{ \eta < \delta_0(\alpha+\eps) \} \to L^2(\dvol_{Z}; H^{-1}(dy) \otimes  \Lambda^*).
\end{multline}
The possible poles occur at indicial roots of $R$, so the first possibility would occur at $\zeta = \delta_0(\alpha)$, and by hypothesis this is the only potential indicial root with real part less than or equal to $\delta_0(\alpha+\eps)$. However we know that
\begin{equation*}
	v(s,y,z) = \frac1{2\pi} \int_{\eta = \delta_0(\alpha)} \cM v(\xi,y,z) s^{-i\zeta} \; d\xi
\end{equation*}
so in particular (as $1/\xi^2$ is not integrable) $\cM v$ does not have any poles on this line. Hence
\begin{multline*}
	\cM v (\zeta)= I ( R ; i \zeta )^{-1} \lrpar{\cM ( H v + h) (\zeta)} \\
	\text{ as holomorphic functions }
	\{ \eta < \delta_0(\alpha+\eps) \} \to L^2( \dvol_{Z};H^{-1}(dy) \otimes  \Lambda^* )
\end{multline*}
and we conclude that
\begin{equation*}
	v \in s^{\eps} L^2(s^{f_0} ds \dvol_{Z}; H^{-1}(dy) \otimes  \Lambda^*).
\end{equation*}

{\bf (2):}
This follows as in \cite[Corollary 7.19]{Mazzeo:edge} by analyzing the Mellin transform. 
Without loss of generality we can arrange, by conjugating $R$ with an appropriate power of $s$,  to work with the measure $\frac1s(ds dy \dvol_Z )$. We will assume, for the duration of the proof, that this has been done without reflecting it in the notation.
This has the advantage that the Parseval formula for the Mellin transform has the form\footnote{
For the measure $s^{f_0} \; ds$ the Parseval formula for the Mellin transform takes the form
\begin{equation*}
	\int_0^\infty g_1(s) g_2(s)  s^{f_0} \; ds 
	= \int_{ \eta = C } \cM g_1(\zeta) \;\; \cM g_2( -(f_0+1)i -\zeta ) \; d\xi
\end{equation*} }
\begin{equation*}
	\int_0^\infty g_1(s) g_2(s) \; \frac{ds}s 
	= \int_{ \eta = C } \cM g_1(\zeta) \;\; \cM g_2(-\zeta ) \; d\xi
\end{equation*}
with $C$ chosen so that the integral on the right makes sense.

Notice that from knowing $u,w \in s^{-\eps}L^2_{\iie}$ and $R(u), R^t(w) \in s^\eps L^2_{\iie}$ the respective Mellin transforms are defined on the half-planes
\begin{gather*}
	\cM(w)(\zeta) \text{ on } \{ \eta \leq -\eps \},
	\quad
	\cM(Ru)( - \zeta ) \text{ on } \{ \eta \geq -\eps \} \\
	\cM(R^tw)(\zeta) \text{ on } \{ \eta \leq \eps \},
	\quad
	\cM(u)( - \zeta ) \text{ on } \{\eta \geq \eps \}
\end{gather*}
so that {\em a priori} there is in each case only one choice for the constant $C$ appearing in Parseval's formula. 
More precisely, $C=-\eps$ for the first pair and $C=\eps$ for the second pair.

Using part 1) of this Lemma  we know we can extend 
\begin{equation*}
	\cM(u)( - \zeta) \text{ to } \{ \eta \geq -\eps \}
\end{equation*}
albeit with a loss in tangential regularity.
Fortunately this loss in tangential regularity is compensated by a gain in tangential regularity in $\cM(R^t w)$ in this same region. 
Indeed, since $w \in s^{-\eps}H^2_{\pie}$,
we know that $R^t w \in s^{\eps}H^1_{\pie}$ hence we have $\pa_y R^t w \in s^{-1+\eps}L^2_{\iie}$. It follows that the Mellin transform of $\pa_y R^t w$ is a holomorphic map from $\{ \eta < -1+\eps \}$ into $L^2(dy\; \dvol_Z ; \Lambda^*)$ and hence on this same half-plane $\cM(R^t w)$ maps holomorphically into $L^2(\dvol_Z, H^1(dy) \otimes  \Lambda^*)$.
Again applying Calderon's complex interpolation method, we conclude that
\begin{equation}\label{InterpolatePW}
	\cM(R^t w)(\zeta) \in L^2(dz, H^{\eps - \eta} )
	\text{ for } \eps -1 \leq \eta \leq \eps.
\end{equation}
The same reasoning applies to $w$.

Thus if we start out with $\ang { u, R^t w }$ which we can write as
\begin{equation*}
	\int \int_{\eta =\eps } \cM(R^t w)(\zeta) \;\; \cM(u)(-\zeta) \; d\xi \; dy \; \dvol_Z,
\end{equation*}
we can deform the contour from $\{\eta =\eps\}$ to $\{\eta =-\eps\}$ and throughout this deformation the integrand stays holomorphic with the loss in tangential regularity of $\cM(u)$ exactly compensated by a gain in regularity by $\cM(R^t w)$, i.e. the integrand makes sense as a pairing throughout the deformation. 
Moreover the integrand is holomorphic in this region and so the value of the integral does not change during the deformation.
Hence we can write $\ang{ u, R^t w }$ as
\begin{equation*}
	\int \int_{\eta =-\eps } \cM(R^t w)(\zeta) \;\; \cM(u)(-\zeta) \; d\xi \; dy \; \dvol_Z.
\end{equation*}
Now integrating each term by parts we write this as
\begin{equation*}
	\int \int_{\eta =-\eps } \cM(w)(\zeta) \;\; \cM(Ru)(-\zeta) \; d\xi \; dy \; \dvol_Z,
\end{equation*}
which by another application of Parseval's formula we recognize as $\ang{ w, Ru }$.
\end{proof}

\subsection{End of induction: $\eth_{\dR}$ is essentially self-adjoint and Fredholm } $ $\newline
Our next task is to use the information gleaned in the previous section to show that elements of the maximal domain of $\eth_{\dR}$ as an operator on 
$L^2_{\iie}(X; \Iie\Lambda^*X)$ are automatically in $\rho^{\eps}L^2_{\iie,\Gamma}(X; \Iie\Lambda^*X)$.
As explained in Lemma \ref{localize}, 
and also in Proposition \ref{prop:upshot},
we can prove this by showing that it holds  in any basic neighborhood at the boundary.

\begin{proposition}\label{DomVanBdy} Up to rescaling suitably the metric, the following is true.
{\item 1)} Let $u$ be in the maximal domain of $\eth_{\dR}$ as an operator on $L^2_{\iie}(X;\Iie\Lambda^*X)$ then 
for any $\eps\in (0,1)$, $u \in \rho^{\eps}H^1_{\iie}(X;\Iie\Lambda^*X)$. 
{\item 2)} The maximal domain $\cD_{\max}(\eth_{\dR})  $ is compactly embedded in $L^2_{\iie}.$  
\end{proposition}

\begin{proof}
We can write $u=u_1 + u_2$ where $u_2$ vanishes on a neighborhood of the singular locus 
and $u_1$ has compact support in the union of a finite number of distinguished 
neighborhoods $W_j, j\in J$ as in Section \ref{section:iterated}.  
Clearly $u_1$ and $u_2$ are both in the maximal domain of $\eth_{\dR}.$
By a partition of unity argument, 
we can write $u_1 = \sum_{j \in J} \chi_j u_1$ for a suitable choice of functions  
$\chi_j$ as in Lemma \ref{localize}. Lemma \ref{localize} shows that 
each $\chi_j u_1$ ($j \in J$) belongs to the maximal domain of $\eth_{\dR}.$ 
It is clear that it suffices to  show the  Proposition \ref{DomVanBdy} for each $\chi_j u_1.$ 
Therefore we can assume that $u$ has compact support in a basic neigborhood $W.$
We begin with the following intermediate result.
\begin{proposition} \label{prop:intermediate}
Let $u$, with compact support in $W$,  be in the maximal domain of $\eth_{\dR}$ as an operator on $L^2_{\iie}(X;\Iie\Lambda^*X).$ Then, for any $\eps\in (0,1)$, $u \in x^{\eps}L^2_{\iie}(X;\Iie\Lambda^*X)$.

\end{proposition}

\begin{proof} 
Fix $\eps_0 \in (0,1)$ small enough that 
\begin{equation*}
	\{ \Re \zeta + \tfrac f2 + \tfrac12: \zeta \in \spec_b(  R ) \} \cap [-1/2-\eps, -1/2+ \eps] \subseteq \{ -1/2 \}.
\end{equation*}
According to Proposition \ref{prop:upshot},  we only need to check the following. Let  $u \in s^{-1/2}L^2_{\iie}(Z \times T_qY^+;\Iie\Lambda^*)$ be  such that 
$N_q(P_0)(u) \in s^{1/2}L^2_{\iie}(Z \times T_qY^+;\Iie\Lambda^*)$ then $u \in s^{-1/2+\eps_0}L^2_{\iie}(Z \times T_qY^+;\Iie\Lambda^*)$. 
We fix such a $u$ in the rest of this proof.
We will abbreviate $L^2_{\iie}(Z_ \times T_qY^+;\Iie\Lambda^*)$ by $L^2(q)$.

Applying  Lemmas \ref{lem:Bessel} and \ref{lem:NClosedRange}, we know that
$R=N_q(P_0)$ is injective and has closed range as a map from $s^{-1/2 + \eps_0}L^2(q)$ to itself (when defined on its maximal domain).
It follows that $R^t$ is surjective from $s^{-1/2-\eps_0}L^2(q)$ to itself (when defined on its minimal domain) . 

Let $G$ be the bounded generalized inverse of $R^t$. 
That is, $G$ is a bounded map from $s^{-1/2-\eps_0}L^2(q)$ to itself, with image contained in the domain of $R^t$, and satisfies
\begin{equation*}
	R^t G = \Id_{s^{-1/2-\eps_0}L^2(q)}.
\end{equation*}

Let $\phi$ be any element of $s^{-1/2+\eps_0}H_{\pie}^{1}(Z \times T_qY^+;\Iie\Lambda^* )$. Then $v = G\phi$ satisfies
\begin{equation*}
	v \in s^{-1/2-\eps_0} L^2(a), \quad
	R^t v = R^t G\phi = \phi  \in s^{-1/2+\eps_0} L^2(q),
\end{equation*} 
the latter statement and elliptic regularity allows us to strengthen the former to $v \in s^{-1/2-\eps_0} H_{\pie}^{2}(Z \times T_qY^+;\Iie\Lambda^* )$.
On the other hand, we know that $R u \in s^{1/2} L^2(q)\subset s^{-1/2+\eps_0} L^2(q)$, 
so we can apply Lemma \ref{lem:integration} (2) with ($\alpha=-1/2$) to conclude that
\begin{equation*}
	\ang{R u,v} = \ang{R^t v, u}.
\end{equation*}

But then we also have
\begin{equation} \label{eq:1}
	\ang{R u, v} = \ang{ R u, G \phi} = \ang{G^t R u, \phi}
	\end{equation} 
	where we recall that $R u\in s^{1/2}L^2(q)  \subset s^{-1/2+\eps_0}L^2(q)$, $G\phi\in s^{-1/2-\eps_0}L^2(q)$
and where $G^t$ denotes the functional analytic transpose of the {\it bounded} operator $G$;
$G^t$ acts continuously on $s^{-1/2+\eps_0}L^2(q)$, so in fact $G^t R u \in s^{-1/2+\eps_0}L^2(q).$

Moreover, we have:
\begin{equation} \label{eq:2}
\ang{R u, v} =\ang{R^t v, u} = \ang{ R^t G \phi, u} = \ang{ \phi , u}.
\end{equation}	
	By comparing the last terms of  
	\eqref{eq:1}  and  \eqref{eq:2} we see that $\ang{u-G^t R u, \phi}=0$ and since $\phi$
	was arbitrary we finally get: $u= G^t R u.$
Therefore $u\in s^{-1/2+\eps_0}L^2(q)$.

Next, taking $\eps_1 \in (0,1)$ small enough that
\begin{equation*}
	\{ \Re \zeta + \tfrac f2 + \tfrac12: \zeta \in \spec_b(  R ) \} \cap [-1/2+\eps_0-\eps_1, -1/2+\eps_0+ \eps_1] = \emptyset \subseteq \{ -1/2+\eps_0 \},
\end{equation*}
we can repeat the argument above and conclude $u \in s^{-1/2 + \eps_0 + \eps_1}L^2(q)$ and indeed continuing in this way we can conclude that
$u \in s^{-1/2 + \eps}L^2(q)$ for any $\eps \in (0,1)$ as required.

\end{proof}

\noindent {\it Proof of Proposition \ref{DomVanBdy}.}

1) 
We proceed by induction   on the depth of a Witt space.  For Witt spaces of depth zero 
 there is nothing we need to prove. Let $m\in \mathbb N^*$ and assume that the result is true 
 for any Witt space of depth $\leq m-1.$  
With the notations of Section \ref{section:iterated},  we can assume that  
 $u \in \cD_{\max}(\eth_{\dR})  $ has compact support in 
$W=W_{\alpha_{j_0}, \ldots, \alpha_{j_l}}.$ Then  combining Proposition \ref{prop:intermediate}, \eqref{EllReg} and Lemma \ref{lem:iie}, one gets  
$u \in r_{j_1}^\epsilon H^1_{\iie} (W_{\alpha_{j_0}, \ldots, \alpha_{j_l}}, 
\Iie\Lambda^*
 W_{\alpha_{j_0}, \ldots, \alpha_{j_l}} ).$
  Now we use formula \eqref{DNearBdy} with $Z=L_{j-1}$  for the chart $(W_{\alpha_{j_0}, \ldots, \alpha_{j_l}} , \Theta)$  \eqref{eq:chart}.
 We get:
 $$
 (r_{j_1}, u^0_{\alpha_0}) \rightarrow \eth_{\dR}^{L_{j-1}} u(r_{j_1}, u^0_{\alpha_0}, z) \in r_{j_1}^\epsilon L^2_{\iie} ( \{( r_{j_1}, u^0_{\alpha_0})\} ;  L^2_{\iie} ( L_{j-1}, \Iie\Lambda^*) ).
 $$ Thus, for any any $\epsilon \in (0,1),$ we get that
 $$
 (r_{j_1}, u^0_{\alpha_0}) \rightarrow  u(r_{j_1}, u^0_{\alpha_0}, z) \in r_{j_1}^\epsilon L^2_{\iie} ( \{( r_{j_1}, u^0_{\alpha_0})\} ;  \cD_{\max}(\eth_{\dR}^{L_{j-1}})   ).
 $$
 Therefore, for any $\epsilon \in (0,1),$ the induction hypothesis  shows that:
$$
 (r_{j_1}, u^0_{\alpha_0}) \rightarrow  u(r_{j_1}, u^0_{\alpha_0}, z) \in r_{j_1}^\epsilon L^2_{\iie} ( \{( r_{j_1}, u^0_{\alpha_0})\} ;  (r_{j_2}\ldots r_{j_h} )^\epsilon L^2_{\iie} ( L_{j-1}, \Iie\Lambda^*) ).
 $$   From this we get that:
 $$
 u \in (r_{j_1} r_{j_2}\ldots r_{j_h})^\epsilon L^2_{\iie} (W_{\alpha_{j_0}, \ldots, \alpha_{j_l}}, \Iie\Lambda^* W_{\alpha_{j_0}, \ldots, \alpha_{j_l}} ).
 $$ Then, using \eqref{EllReg} and Lemma \ref{lem:iie}, we get that:
 $$
 u \in (r_{j_1} r_{j_2}\ldots r_{j_h})^\epsilon H^1_{iie} (W_{\alpha_{j_0}, \ldots, \alpha_{j_l}}, \Iie\Lambda^* W_{\alpha_{j_0}, \ldots, \alpha_{j_l}} ).
 $$ Therefore, 1) is proved.
 
 2) This is a consequence of the fact that $\rho^{\eps}H^1_{\iie}(X;\Iie\Lambda^*X)$ is compactly embedded in $L^2_{\iie}.$ 

\noindent {\it End of Proof of Proposition \ref{DomVanBdy}}
\end{proof}

With this proposition we know that elements of the maximal domain have some `extra' degree of vanishing, and we can apply the following argument of Gil-Mendoza \cite{Gil-Mendoza}.

\begin{proposition}[Gil-Mendoza] \label{GilMendoza}
If $\cD_{\max}(\eth_{\dR}) \subseteq \rho^C L^2_{\iie}(X;\Iie\Lambda^*X)$ for some $C >0$, then, as an operator on $L^2_{\iie}(X;\Iie\Lambda^*X)$,
\begin{equation*}
	\cD_{\max}(\eth_{\dR}) \cap \bigcap_{\eps>0} \rho^{1-\eps} L^2_{\iie}(X; \Iie\Lambda^*X) \subseteq \cD_{\min}(\eth_{\dR})
\end{equation*}
\end{proposition}

{\bf Remark.}
Since we have actually shown not only that 
\begin{equation*}
	\cD_{\max}(\eth_{\dR}) \subseteq \rho^C H^1_{\iie}(X;\Iie\Lambda^*X)
\end{equation*}
but in fact
\begin{equation*}
	\cD_{\max}(\eth_{\dR}) \subseteq \bigcap_{\eps>0} \rho^{1-\eps} H^1_{\iie}(X;\Iie\Lambda^*X),
\end{equation*}
this proposition implies $\cD_{\max}(\eth_{\dR}) = \cD_{\min}(\eth_{\dR})$.

\begin{proof}
We point out the following simple consequence of the formal self-adjointness of $\eth_{\dR}$ and the definitions of the minimal/maximal domains and weak derivatives:
\begin{lemma}
An element $u \in \cD_{\max}(\eth_{\dR})$ is in $\cD_{\min}(\eth_{\dR})$ if and only if
\begin{equation}\label{MinDomChar}
	(\eth_{\dR}u, v) = (u, \eth_{\dR}v), \Mforevery v \in \cD_{\max}(\eth_{\dR}).
\end{equation}
\end{lemma}
\begin{proof}
 For any operator $D$ with \underline{formal} adjoint $D^*$ one has, 
\begin{gather*}
	u \in \cD(D_{\min}) \iff
	u \in \cD\lrpar{ \lrpar{ (D^*)_{\max} }^* } \\
	\iff
	\ang{Du, v} = \ang{u, D^*v} \text{ for every } v \in \cD( (D^*)_{\max} )
\end{gather*}
If $D$ is symmetric so that $D^* = D$, then this is \eqref{MinDomChar}. 
\end{proof} 
Let $u \in \cD_{\max}(\eth_{\dR}) \cap \bigcap_{\eps>0} \rho^{1-\eps} L^2_{\iie}(X; \Iie\Lambda^*X),$ so 
\begin{equation*}
	u \in  \bigcap_{\eps>0} \rho^{1-\eps} H^1_{\iie}(X; \Iie\Lambda^*X).
\end{equation*}
Set $u_n = \rho^{1/n}u$ for $n \in \bbN$, so that for each $n,$ 
$u_n \subseteq \rho H_{\iie}^1(X;\Iie\Lambda^*),$ and, for every $\eps \in (0,1)$,
\begin{equation}\label{FirstConv}
	u_n \to u \text{ in } \rho^{1-\eps}H_{\iie}^1(X;\Iie\Lambda^*)
	\Mand
	\eth_{\dR} u_n \to \eth_{\dR} u \text{ in } \rho^{-\eps} L^2_{\iie}(X;\Iie\Lambda^*).
\end{equation}
Let $\eps \in (0,1)$  so that $\cD_{\max}(\eth_{\dR}) \subseteq \rho^\eps H_{iie}^1(X;\Iie\Lambda^*)$. 
Then, for any  $v \in \cD_{\max}(\eth_{\dR})$, \eqref{FirstConv} implies
\begin{gather*}
	(\eth_{\dR}u_n, v)_{L^2} 
	= (\rho^\eps \eth_{\dR} u_n, \rho^{-\eps}v)_{L^2}
	\to (\rho^\eps \eth_{\dR} u, \rho^{-\eps}v)_{L^2}
	= (\eth_{\dR}u, v)_{L^2}, \\
	\Mand 
	(u_n, \eth_{\dR}v) \to (u, \eth_{\dR}v).
\end{gather*}
Moreover, by the previous Lemma,  $u_n \in \cD_{\min}(\eth_{\dR})$ implies $(\eth_{\dR}u_n, v) = (u_n, \eth_{\dR}v)$.

It follows that $(\eth_{\dR}u, v) = (u, \eth_{\dR}v)$ for every $v \in \cD_{\max}(\eth_{\dR})$ and hence $u \in \cD_{\min}(\eth_{\dR})$.
\end{proof}

Altogether, we have shown the theorem we advertised. We recall the statement for the benefit of the reader.

\begin{theorem} \label{thm:fred} Up to a rescaling of the metric $g$, the following is true.
\begin{itemize}
\item[1)] If $\hat{X}$ satisfies \eqref{Ass2} for all strata, then the iterated incomplete edge de Rham operator $\eth_{\dR}$ is 
essentially self-adjoint as an operator on $L^2_{\iie}(X;\Iie\Lambda^*X)$; 
\item[ 2)] the maximal domain is contained in $\cap_{\eps>0} \rho^{1-\eps} H^1_{\iie}(X;\Iie\Lambda^*X)$ which is compactly included
in  $L^2_{\iie}(X;\Iie\Lambda^*X)$; 
\item[ 3)] $\eth_{\dR}$ is Fredholm when defined on its maximal domain $\cD_{\max}(\eth_{\dR})$ endowed with the graph norm;
\item[ 4)] $\eth_{\dR}$ has discrete $L^2$-spectrum of finite multiplicity.
\end{itemize}
\end{theorem} 

\begin{proof} The equality of the maximal and minimal domain as well as  2) are a direct consequence of the last Proposition.
Let us show that $\eth_{\dR}$ is self-adoint on its maximal domain. We denote by  $\eth_{\dR,{\rm max}}$ the operator 
$\eth_{\dR}$ on its maximal domain. If $v$ is in the domain of  $\eth_{\dR,{\rm max}}$ then integration by parts, which is allowed 
because of the extra vanishing, implies 
that $v$ is in the domain of  $(\eth_{\dR,{\rm max}})^*$ and that  $\eth_{\dR,{\rm max}}v=   (\eth_{\dR,{\rm max}})^* v$.
Viceversa let $v$ be  in the domain of  $(\eth_{\dR,{\rm max}})^*$. Observe that $\forall u\in C^\infty_c$ we have
$< \eth_{\dR} u,v>=<u,\eth_{\dR} v>$, with $\eth_{\dR}$ acting distributionally on $v$. But from the definition 
of adjoint we also know that $< \eth_{\dR} u,v>= <u, (\eth_{\dR,{\rm max}})^*v>$ and since this is true for
all $u\in C^\infty_c$ we infer that $\eth_{\dR} v$ is in $L^2_{\iie}$ (indeed, by definition,
  $(\eth_{\dR,{\rm max}})^*v\in L^2_{\iie}$). Thus
$v$ is in the domain of $\eth_{\dR,{\rm max}}$ and $\eth_{\dR,{\rm max}} v= (\eth_{\dR,{\rm max}})^*v$. 
This proves that $\eth_{\dR,{\rm max}}$ is self-adjoint. \\
3) Since $\eth_{\dR}$ is self-adjoint on its maximal domain, $(i \,{\rm Id} + \eth_{\dR})$ is invertible. Since 
$\cD_{\max}(\eth_{\dR})$ is compactly  embedded into $L^2_{\iie}(X;\Iie\Lambda^*X)$, $(i\, {\rm Id} + \eth_{\dR})^{-1}$  
defines a parametrix for $\cD_{\max}(\eth_{\dR})$  with compact reminder.\\
4) Since $\eth_{\dR}$ is Fredholm, there exists $\epsilon >0$ such that 
$(\epsilon\,{\rm Id} + \eth_{\dR})$ is invertible. Since the maximal domain is compactly embedded in 
$L^2_{\iie}$, $(\epsilon\,{\rm Id} + \eth_{\dR})^{-1}$  is compact and self-adjoint. Thus, 
the spectrum of $(\epsilon\,{\rm Id} + \eth_{\dR})^{-1}$ is discrete with finite multiplicity. Therefore, the spectrum 
of $\eth_{\dR}$ is discrete and has finite multiplicity.

\end{proof}

\section{The signature operator on Witt spaces} \label{sec:Sign} $\;$

We now turn from the de Rham operator to the signature operator, first on forms with scalar 
coefficients and then with $C^*$-algebra coefficients. We show first that these are Fredholm 
operators, but more importantly, that they define classes in the topological groups $K_*(\hat X)$ 
and $K_*(C^*\Gamma)$, respectively. The index of these operators is independent of the choice of
metric and defines a topological invariant.  In a sequel to this paper we show that this class is 
a Witt bordism invariant and thereby identify it with the so-called "symmetric signature". 

\subsection{The signature operator $\eth_{\sign}$}  $ $\newline
If $X$ is even-dimensional, the Hodge star induces a natural involution on the differential forms on $X$,
\begin{equation*}
	\cI: \Omega^*(X) \to \Omega^*(X), \quad \cI^2 = \Id
\end{equation*}
whose $+1$, $-1$ eigenspaces are known as the set of self-dual, respectively anti-self dual, forms and are denoted $\Omega^*_+$, $\Omega^*_-$.
The involution $\cI$ extends naturally to $\Iie \Omega^*(X)$ and with respect to the splitting 
$\Iie\Omega^*(X) = \Iie \Omega^*_+ \oplus \Iie \Omega^*_-$, the de Rham operator decomposes
$$
\eth_{\dR} 
	= \begin{pmatrix}
	0&  \eth_{\sign}^-  \\
	 \eth_{\sign}^+  & 0
	\end{pmatrix}
$$
where
$$
\eth_{\sign}^+=d+\delta: \Iie \Omega^*_+(X) \to \Iie\Omega^*_-(X),\;  \eth_{\sign}^-=( \eth_{\sign}^+)^*.
$$
If instead the manifold $X$ is odd-dimensional, the signature operator of an iterated conic metric is
\begin{equation*}
	\eth_{\sign} = -i(d \cI + \cI d)
	= -i\cI (d - \delta) = -i (d - \delta) \cI.
\end{equation*}
We point out for later use that in either case, given a continuous map $r: \hat X \to B\Gamma$, 
we also obtain a twisted Mishchenko-Fomenko 
signature operator $ \wt \eth_{\sign}$ acting on sections of the bundle $\Iie\Lambda_\Gamma^*(X)$.

\begin{theorem}  \label{thm:Fred}
Up to rescaling suitably the metric the following is true. If $\hat{X}$ satisfies \eqref{Ass2} for all strata, then 
the  iterated  incomplete edge signature operator $\eth_{\sign}$ is essentially self-adjoint  with maximal domain contained in 
$$
\bigcap_{\eps>0} \rho^{1-\eps} H^1_{\iie}(X;\Iie\Lambda^*X).
$$ 
Its  unique self-adjoint extension is Fredholm on its maximal domain endowed with the 
graph-norm; moreover it has discrete $L^2$-spectrum of finite multiplicity.
\end{theorem}
  
\begin{proof}
If $X$ is even-dimensional, it is immediate to see that 
\begin{gather*}
	\cD_{\min}(\eth_{\sign}^+) = \cD_{\min}(\eth_{\dR}) \cap L^2_{\iie}(X;\Iie \Lambda^*_+(X)), \\
	\cD_{\max}(\eth_{\sign}^+) = \cD_{\max}(\eth_{\dR}) \cap L^2_{\iie}(X;\Iie \Lambda^*_+(X))
\end{gather*}
so the result follows from the corresponding results for $\eth_{\dR}$.

For $X$ odd-dimensional, we point out that one can characterize the maximal domain of $d-\delta$ through the same analysis used for $d+\delta$.
Alternately, we can use the result for $d+\delta$ to deduce it for $d-\delta$ as follows.
As explained above, a byproduct of our results is the existence of a strong Kodaira decomposition
\begin{equation*}
	L^2_{\Iie} \Omega^* = 
	L^2\cH \oplus \Image d \oplus \Image \delta
\end{equation*}
where $L^2\cH$ is the intersection of the null spaces of $d$ and $\delta$.
The de Rham operator $d+\delta$ decomposes into
\begin{equation*}
	\lrpar{ d: \Image \delta \to \Image d } \oplus
	\lrpar{ \delta: \Image d \to \Image \delta },
\end{equation*}
hence $d$ and $\delta$ individually have closed range and 
\begin{gather*}
	\cD_{\max}(\eth_{\dR}) \cap \Image \delta = \cD_{\max}(d) \cap \Image \delta \\
	\cD_{\max}(\eth_{\dR}) \cap \Image d = \cD_{\max}(\delta) \cap \Image d 
\end{gather*}
hence $i(d-\delta)$ has closed range with domain contained in (hence, by symmetry, equal to) $\cD_{\max}(\eth_{\dR})$.
Applying Proposition \ref{GilMendoza} to $i(d-\delta)$ then shows that it too is essentially self-adjoint.
\end{proof}

\subsection{The K-homology class $ [\eth_{\sign}]\in K_* (\hat{X})$} $ $\newline
The results proved so far for the signature operator $\eth_{\sign}$ on a Witt space $\widehat{X}$ allow one to define the 
K-homology class  $[\eth_{\sign}]\in K_* (\widehat{X})=KK_* (C(\widehat{X}),\bbC)$. The K-homology signature class
already appears in the work of Moscovici-Wu \cite{Mosc-Wu-Witt}; the definition there is based
on the results of Cheeger.

Recall that an {\it even} unbounded Fredholm module for the $C^*$-algebra $C(\widehat{X})$ is a pair $(H,D)$
such that:
\begin{itemize}
\item $H$ is a Hilbert space endowed with a unitary $*$-representation of $C(\widehat{X})$; $D$ is a self-adjoint unbounded linear operator on $H$;
\item there is a dense $*$-subalgebra $\mathcal A\subset C(\widehat{X})$ such that $\forall a \in \mathcal A$ the domain
of $D$ is invariant by $a$ and $[D,a]$ extends to a bounded operator on $H$;
\item $(1+D^2)^{-1}$ is a compact operator on $H$;
\item $H$ is equipped with a grading $\tau=\tau^*$, $\tau^2=I$, such that $\tau f = f \tau$ and $\tau D= -D \tau$.
\end{itemize}
An {\it odd} unbounded Fredholm module is defined omitting the last condition.

An unbounded Fredhom module defines a Kasparov $(C(\widehat{X}),\bbC)$-bimodule and thus
an element in  $KK_* (C(\widehat{X}),\bbC)$. We refer to  \cite{Baaj-Julg} \cite{Blackadar} for more
on this foundational material.

\begin{theorem}\label{theo:k-homology}
The signature operator $\eth_{\sign}$ associated to a Witt space 
$\widehat{X}$ endowed with an iterated conic metric $g$ defines an unbounded Fredholm module
for $C(\widehat{X})$ and thus a class   $[\eth_{\sign}]\in KK_* (C(\widehat{X}),\bbC)$, $* \equiv\dim X \,{\rm mod} \,2$.
Moreover, the class  $[\eth_{\sign}]$ does not depend on the choice of iterated conic metric on $\widehat{X}$.
\end{theorem}

\begin{proof}
We take $H=L^2_{\iie}(X;\Iie\Lambda^*X)$, endowed with the natural representation of $C(\widehat{X})$
by multiplication operators. We take $D$ as the unique closed self-adjoint extension of $\eth_{\sign}$. 
These data depend of course on the choice of the iterated conic metric. We take
$\mathcal{A}$ equal to the space of Lipschitz functions on $\widehat{X}$ with respect to $g$;
$\mathcal{A}$ does not depend on the choice of $g$, since two iterated conic metric are quasi-isometric
(see Prop \ref{prop:adm}).
Finally, in the even dimensional case we take the involution defined by $\mathcal{I}$.
All the conditions defining an unbounded Kasparov module are easily proved using the results
of the previous section: indeed, if $f$ is Lipschitz then it is elementary to see that 
multiplication by $f$ sends the maximal domain of $\eth_{\sign}$
into itself; moreover $[f,\eth_{\sign}]$ is given  by Clifford multiplication by $df$
which exists almost everywhere and is an element
in $L^\infty (\widehat{X})$; in particular $[f,\eth_{\sign}]$ 
extends to a bounded operator on $H$; finally we know that $(1+D^2)^{-1}$ is a compact operator (indeed, we proved that
this is true for $(i+D)^{-1}$ and $(-i+D)^{-1}$).
Thus there is a well defined class $ KK_* (C(\widehat{X}),\bbC)$ which we denote simply by $[\eth_{\sign}]$;
this class depends a priori on the choice of the metric $g$. 
Recall however  that two iterated conic metric $g_0$ and $g_1$ are joined by a path of iterated conic metrics
$g_t$. See Proposition \ref{prop:homotopymet}. 
Let $\eth^0_{\sign}$ and $\eth^1_{\sign}$ the corresponding signature operators, with domains in $H_0$
and $H_1$.
Proceeding as in the work of Hilsum on Lipschitz manifolds \cite{Hilsum-LNM} one can prove that the 1-parameter family
$(H_t, \eth^t_{\sign})$ defines an {\it unbounded} operatorial  homotopy; using the homotopy invariance of
$KK$-theory one obtains
  $$[\eth^0_{\sign}]=[\eth^1_{\sign}] \;\,\text{ in }\;\;  KK_* (C(\widehat{X}),\bbC)\,.$$
  We omit the details since they are a repetition of the ones given in  \cite{Hilsum-LNM}. 
\end{proof}

\subsection{The index class of  the twisted signature operator $ \wt \eth_{\sign}$} $ $\newline
Let $\widehat{X}$ be a Witt space endowed with an iterated conic metric. 
Assume now that we are also given a continuous map $r:\widehat{X}\to B\Gamma$ and let
$\Gamma \to \widehat{X}^\prime \to \widehat{X}$ the Galois $\Gamma$-cover induced by $E\Gamma\to B\Gamma$. 
We consider the Mishchenko bundle
\begin{equation*}
	\wt{C^*_r}\Gamma: = C^*_r\Gamma \btimes_\Gamma \widehat{X}^\prime.
\end{equation*} 
and the signature operator with values in the restriction of $\wt{C^*_r}\Gamma$ to $X$, which we denote
by $ \wt \eth_{\sign}$. 

\begin{proposition}  \label{prop:esa-gamma}
The twisted signature operator $ \wt \eth_{\sign}$ is essentially self-adjoint as an operator on 
$L^2_{\iie, \Gamma}(X;\Iie\Lambda_\Gamma^*X)$, with maximal domain contained 
in \\$\cap_{\eps>0} \rho^{1-\eps} H^1_{\iie,\Gamma}(X;\Iie\Lambda_\Gamma^*X)$ which is in turn 
$C^*_r\Gamma$-compactly 
included in the Hilber $C^*_r\Gamma$-module $L^2_{\iie, \Gamma}(X;\Iie\Lambda_\Gamma^*X)$. 
\end{proposition}

\begin{proof}
We briefly point out how the proof given for $ \eth_{\dR}$  and $\eth_{\sign}$ extends to the case of  $\wt \eth_{\dR}$
and $\wt \eth_{\sign}$.
Recall that a $C^*_r \Gamma-$distribution on $\mbox{reg}\,(\widehat{X})$ is a $\mathbb{C}-$linear form 
$$
T: C^\infty_c ( \mbox{reg}\,(\widehat{X}), \Iie\Lambda_\Gamma^*X) \rightarrow C^*_r \Gamma
$$ satisfying the following 
property. For any compact $K \subset \mbox{reg}\,(\widehat{X}),$ there exists a finite set $S$
of elements of ${\rm Diff}^*_{ie, \Gamma}$ such that:
$$
\forall u \in C^\infty_K ( \mbox{reg}\,(\widehat{X}), \Iie\Lambda_\Gamma^*X), \;
|| < T ; u >||_{C^*_r \Gamma} \leq \sup_{  Q \in S}  || (Q u) ||_{L^2_{\iie,\Gamma}}.
$$ Of course, any element of  $L^2_{\iie, \Gamma}(X;\Iie\Lambda_\Gamma^*X)$ defines
a $C^*_r \Gamma-$distribution on $\mbox{reg}\,(\widehat{X}).$
It is clear that $\wt  \eth_{\dR}$ sends $L^2_{\iie, \Gamma}(X;\Iie\Lambda_\Gamma^*X)$ into
the space of $C^*_r \Gamma-$distributions. Therefore, the notion of maximal domain 
for $\wt  \eth_{\dR}$ is  defined. The notion of minimal domain is also well defined (this is simply the closure
of $C^\infty_c$ with respect to the norm $||u||+||\wt  \eth_{\dR} u||$ ). Notice that these two extensions
are closed.
Our first task is to show that these two extensions coincide. 
To this end we shall make use of the fundamental hypothesis that the reference map $r: X \rightarrow B \Gamma$ extends continously 
to the whole singular  space $\widehat X$. Therefore, for any distinguished neighborhood $W \simeq \bbR^b \times C(Z)$, 
the induced  $\Gamma$-coverings over  $W$ and  over $Z$ are trivial. 
This implies  that for any $q \in Y,$ 
$N_q(\wt  \eth_{\dR} ) $ is conjugate to  $N_q(\eth_{\dR}) \otimes {\rm Id}_{\wt C^*_r \Gamma}.$
Once this has been observed we have, immediately, that  Proposition \ref{prop:upshot} and Lemma \ref{lem:integration} extend 
to the case of $\wt\eth _{\dR}.$  Then, Proposition \ref{prop:intermediate} also extends easily to the present case 
 showing that the maximal domain of  $ \wt \eth_{\sign}$ is included in
$\cap_{\eps>0} \rho^{1-\eps} H^1_{\iie,\Gamma}(X;\Iie\Lambda_\Gamma^*X)$. 
Once the extra vanishing is obtained, we can apply the argument give in the proof of
Theorem \ref{thm:fred} in order to show that the maximal extension is in fact self-adjoint.
The argument of Gil-Mendoza can also be extended, showing the equality of the maximal and the minimal domain. 
The details of all this are easy and for the sake of brevity we omit them. Finally, proceeding as in \cite{LP-SMF},
one can prove that $\rho^\eps H^1_{\iie,\Gamma}(X;\Iie\Lambda_\Gamma^*X)$ is $C^*_r \Gamma$-compactly included into 
$ L^2_{\iie,\Gamma}(X;\Iie\Lambda_\Gamma^*X)$. The Proposition is proved for $\wt \eth_{\dR}$. The extra step needed
for the signature operator is proved as in Theorem \ref{thm:Fred}.
\end{proof}

From now on we shall only consider the closed unbounded self-adjoint  $C^*_r\Gamma$-operator  of Proposition \ref{prop:esa-gamma}
and with common abuse of notation we keep denoting it  by   $ \wt \eth_{\sign}$ .

We now proceed to show the following fundamental 

\begin{proposition}\label{prop:regularity}
The operator   $ \wt \eth_{\sign}$  is a {\it regular} operator. Consequently 
 $(i\pm \wt \eth_{\sign})$ and  $(1+\wt \eth_{\sign}^2)$ are  invertible.  
\end{proposition}

\begin{proof}
Recall that a closed unbounded self-adjoint operator $D$
on a Hilbert $C^*_r\Gamma$-module is
said to be regular if $1+D^2$ is surjective. One can show, see \cite{Lance}, that $D$ is regular if and only if
$1+D^2$ has dense image if and only if $(i\pm D)$ has dense image if and only if $(i\pm D)$ is surjective.
Moreover, if $D$ is regular then both  $(i\pm D)$  and $1+D^2$ have an inverse.
For a simple example of an unbounded self-ajoint operator on a Hilbert module  such that  $(i+ D)$ and  $(i- D)$ are
not invertible see \cite{Hilsum-K}, page 415.
 \\
We shall prove that our operator is regular by employing ideas of George Skandalis, explained in detail
in work of Rosenberg-Weinberger \cite{ros-weinberger}. We have seen in the previous subsection that $\eth_{\sign}$ defines an
unbounded Kasparov $(C(\widehat{X}),\bbC)$-bimodule and thus a  class
$[\eth_{\sign}]\in KK_* (C(\widehat{X}),\bbC)$. Consider now 
$$
\mathcal{E}:= L^2_{\iie}(X;\Iie\Lambda^*X)\otimes_{\bbC} C^*_r \Gamma \,;
$$
tensoring $\eth_{\sign}$ with $\Id_{C^*_r\Gamma}$ we obtain in an obvious way an unbounded 
Kasparov $(C(\widehat{X})\otimes C^*_r\Gamma, C^*_r\Gamma)$-bimodule that we will denote  by $(\mathcal{E},\mathcal{D})$.
For later use we denote the corresponding KK-class as 
\begin{equation}\label{tensor-class}
[[ \eth_{\sign} ]]\in KK_*(C(\widehat{X})\otimes C^*_r\Gamma, C^*_r\Gamma).
\end{equation}
Consider $A:= C(\widehat{X})\otimes C^*_r\Gamma$ and set
$$
\mathcal{A}:=\{a\in A : a(\Dom \mathcal{D})\subset \Dom \mathcal{D} \text{ and } [a,\mathcal{D}] \ 
\text{extends to an element of }\; \mathcal{L}( \mathcal{E}) \}.
$$
It is a non-trivial result, due to Skandalis, that $\mathcal{A}$ is a dense *-subalgebra of $A$ stable under holomorphic functional calculus.
Consider now the Mishchenko bundle $\wt{C^*_r}\Gamma$ and its continuous sections
$C^0 ( \widehat{X}; \wt{C^*_r}\Gamma)=:P$. It is obvious that $P$ is a finitely generated projective right $A$-module. 
The result cited above, together with Karoubi density theorem, implies that there exists a finitely generated projective 
right $\mathcal{A}$-module $\mathcal{P}$ such that $P=\mathcal{P}\otimes_\mathcal{A} A$. Consider for $\xi\in P$ the operator 
$T_\xi: \mathcal{E} \to P\hat{\otimes}_A \mathcal{E}$ defined by $T_\xi (\eta):= \xi\otimes \eta$. $T_\xi$ is a bounded
operator of $C^*_r\Gamma$  Hilbert modules with adjoint $T_\xi^*$. Recall now, following Skandalis, that a 
$\mathcal{D}$-connection in the present context is a  symmetric  $C^*_r\Gamma$-linear operator $\widetilde{\mathcal{D}}$
$$\widetilde{\mathcal{D}}: \mathcal{P}\otimes_\mathcal{A} \Dom (\mathcal{D}) \longrightarrow P\hat{\otimes}_A \mathcal{E}$$
such that $\forall \xi\in \mathcal{P}$ the following commutator, defined initially on $(\Dom (\mathcal{D}) )\oplus 
\mathcal{P}\otimes_\mathcal{A} \Dom (\mathcal{D}) $,  extends to a bounded operator on $ \mathcal{E}\oplus 
P\hat{\otimes}_A \mathcal{E}$: 
$$\left[
 \begin{pmatrix}
	\mathcal{D} &  0 \\
	0  & \widetilde{\mathcal{D}}
	\end{pmatrix},\begin{pmatrix}
	0 &  T_\xi^* \\
	T_\xi & 0
	\end{pmatrix} \right]
$$
Rosenberg and Weinberger have proved \cite{ros-weinberger}, following Skandalis, that every $\mathcal{D}$-connection is a 
self-adjoint regular operator. We can end the proof of the present Proposition as follows:
first we  observe that as  $C^*_r\Gamma$ Hilbert
modules  $P\hat{\otimes}_A \mathcal{E}= L^2_{\iie,\Gamma}(X;\Iie\Lambda_\Gamma^*X)$; next we consider 
$\wt \eth_{\sign}$ and prove the following.
\begin{lemma} The operator $\wt \eth_{\sign}$ defines a $\mathcal{D}$-connection. 
\end{lemma}
\begin{proof} (Sketch). 
It will suffice to prove the following. Let $U$ be an open
subset of ${X}$ over which $\widetilde{C^*_r\Gamma}$ is trivial. Then the restriction 
of $\xi \in \mathcal{P}$ to $U$ is a finite sum of terms of the form $\theta \otimes u$ where 
$\theta$ is a flat section and $u$ is a $C^1-$function. So we shall assume that 
$\xi= \theta \otimes u.$
Then for any $\eta \in L^2_{\iie}(U;\Iie\Lambda^*X_{| U}) \otimes C^*_r\Gamma $, one has:
$$
( \wt \eth_{\sign} \circ  T_\xi -T_\xi \circ \mathcal{D}) (\eta)=
\theta \otimes c(d u) \eta + \theta  \otimes u (\wt \eth_{\sign} - \mathcal{D})(\eta),
$$  where $ c ( d u)$ denote the Clifford multiplication. Recall that the restrictions to $U$ of $\wt \eth_{\sign}$ and   $\mathcal{D}$ 
are differential operators of order one having the same principal symbol. 
Therefore, $( \wt \eth_{\sign} \circ  T_\xi -T_\xi \circ \mathcal{D}) $ is bounded 
on $ L^2_{\iie}(U;\Iie\Lambda^*X_{| U}) \otimes C^*_r\Gamma. $
One then gets immediately the Lemma by using a partition of unity.
\end{proof}
Finally, we check easily that 
$\mathcal{P}\otimes_\mathcal{A} \Dom (\mathcal{D})\subset \Dom_{{\rm max}} (\wt \eth_{\sign})$. Since 
$(i+\wt \eth_{\sign})$ has dense image  with domain  $\mathcal{P}\otimes_\mathcal{A} \Dom (\mathcal{D})$,
we see that, a fortiori,
the image of $(i+\wt \eth_{\sign})$ with domain $ \Dom_{{\rm max}} (\wt \eth_{\sign})$ must also be dense.
\end{proof}

These two Propositions yield at once the following 

\begin{theorem}\label{theo:kk}
The twisted signature operator $\wt \eth_{\sign}$ and the $C^*_r\Gamma$-Hilbert
module $L^2_{\iie,\Gamma}(X;\Iie\Lambda_\Gamma^*X)$ define an unbounded 
Kasparov $(\bbC,C^*_r\Gamma)$-bimodule and thus a class in $KK_* (\bbC, C^*_r \Gamma)=K_* (C^*_r\Gamma)$.
We call this the index class associated to $\wt \eth_{\sign}$ and denote it by $\Ind (\wt \eth_{\sign})\in K_* (C^*_r\Gamma)$.\\
Moreover, if  as in \eqref{tensor-class} we denote by $[[\eth_{\sign}]]\in KK_*(C(\widehat{X})\otimes C^*_r\Gamma, C^*_r\Gamma)$ 
the class obtained from $[\eth_{\sign}]\in KK_*(C(\widehat{X}),\bbC)$ by tensoring with  $C^*_r\Gamma$, then
$\Ind (\wt \eth_{\sign})$ is equal to the Kasparov product of the class defined by Mishchenko bundle 
$[\widetilde{C^*_r}\Gamma]\in  KK_0(\bbC,C(\widehat{X})\otimes C^*_r\Gamma)$ with  $[[\eth_{\sign}]]$:
\begin{equation}\label{tensor}
\Ind (\wt \eth_{\sign})= [\widetilde{C^*_r}\Gamma]\otimes [[\eth_{\sign}]]
\end{equation}
In particular, the index class $\Ind (\wt \eth_{\sign})$ does not depend on the choice of the iterated conic metric.
\end{theorem}

\begin{proof}
We already know that $\wt \eth_{\sign}$ is self-adjoint regular and $\bbZ_2$-graded in the even dimensional case.
It remains to show that the inverse of $(1+\wt \eth_{\sign}^2)$ is a $C^*_r \Gamma$-compact operator.
However, the domain of $\wt \eth_{\sign}$ is compactly included in  $L^2_{\iie,\Gamma}(X;\Iie\Lambda_\Gamma^*X)$;
thus $(i+ \wt \eth_{\sign})^{-1}$ and $(-i+\wt \eth_{\sign})^{-1}$ are both compacts. It follows that  $(1+\wt \eth_{\sign}^2)^{-1}$ is
compact. Thus  $(\wt \eth_{\sign}, L^2_{\iie,\Gamma}(X;\Iie\Lambda_\Gamma^*X))$ define an unbounded 
Kasparov $(\bbC,C^*_r\Gamma)$-bimodule as required. The equality $\Ind (\wt \eth_{\sign})= [\widetilde{C^*_r}\Gamma]\otimes [[\eth_{\sign}]]$
is in fact part of the theorem of Skandalis on $\mathcal{D}$-connections. Finally, since we have proved that 
$ [\eth_{\sign}]$, and thus $ [[\eth_{\sign}]]$, is metric independent, and since $ [\widetilde{C^*_r}\Gamma]$
is obviously metric independent, we conclude that $\Ind (\wt \eth_{\sign})$ has this property too. The Theorem is proved.
\end{proof}

\begin{corollary}\label{cor:assembly}
Let $\beta:K_* (B\Gamma)\to K_* (C^*_r\Gamma)$ be the assembly map; let $r_* [\eth_{\sign}]\in K_* (B\Gamma)$
the push-forward of the signature K-homology class. Then 
\begin{equation}\label{assembly}
\beta(r_* [\eth_{\sign}])=\Ind (\wt \eth_{\sign}) \text{ in }  K_* (C^*_r\Gamma)
\end{equation}
\end{corollary}

\begin{proof}
Since $\Ind (\wt \eth_{\sign})= [\widetilde{C^*_r}\Gamma]\otimes [[\eth_{\sign}]]$, this follows immediately
from the very definition of the assembly map. See  \cite{Kasparov-inventiones} (for a survey, see 
\cite{Kasparov-contemporary}).
\end{proof}

\subsection{Concluding remarks and perspectives.}
The results proved in Theorem \ref{theo:k-homology},
Theorem \ref{theo:kk} and 
Corollary \ref{cor:assembly} establish on Witt spaces the first, the fifth and the sixth item of the signature package presented
in the Introduction.
In  the second part of this work we shall complete the formulation and the proof of the signature package
on Witt spaces. 
We shall begin by establishing the Witt bordism \cite{Se}
invariance of the index class $\Ind (\wt \eth_{\sign})$;  the symmetric signature will be introduced through the self-duality properties
of the intersection homology sheaf (see for example \cite{Banagl}); homotopy invariance will be replaced by {\it stratified}
homotopy invariance (see for example \cite{friedman}) ; the Witt higher signatures will be simply
defined as the collection of numbers
$\{ <\alpha,r_* L_*(\hat{X})>\,,\quad \alpha\in H^*(B\Gamma,\bbQ) \}\,,$
with $ L_*(\hat{X})\in H_* (\hat{X},\bbQ)$ the Goresky-MacPherson homology L-class; 
the rational equality of the index class and of the symmetric signature will be established, following an idea
already exploited in \cite{BCS},
through the rational surjectivity of the natural map $\Omega^{{SO}}_* (B\Gamma)\to \Omega^{{\rm Witt}}_* (B\Gamma)$; 
finally, the Chern character of the K-homology signature class $[\eth_{\sign}]\in K_* (\widehat{X})$
has already been proved
by Cheeger and Moscovici-Wu to be equal, rationally, to  the Goresky-MacPherson  L-class $L_* (\hat{X})\in  H_* (\hat{X},\bbQ)$.


\begin{thebibliography}{999}
\bibitem{ALN} 
	\script{Ammann, B;Lauter, R.; Nistor, V.}
	{\em Pseudodifferential operators on manifolds with a Lie structure at infinity.} 
	Annals of Mathematics, {\bf 165}, (2007), pages 717-747.

\bibitem{APS}
	\script{Atiyah, M. F.; Patodi, V. K.; Singer, I. M.}
	{\em Spectral asymmetry and Riemannian geometry. I.} 
	Math. Proc. Cambridge Philos. Soc. {\bf 77} (1975), 43--69  

\bibitem{ADS}
	\script{Atiyah, M. F.; Donnelly, H.; Singer, I. M.}
	{\em Eta invariants, signature defects of cusps, and values of $L$-functions.}
	Ann. of Math. (2)  {\bf 118}  (1983),  no. 1, 131--177.

\bibitem{atiyah-coverings}	
 \script{Atiyah, M. F.}
    {\em Elliptic operators, discrete groups and von Neumann
              algebras},
Colloque ``Analyse et Topologie'' en l'Honneur de
              Henri Cartan (Orsay, 1974),
    Ast\'erisque, No. 32-33, Soc. Math. France
   (1976).
   
   \bibitem{Baaj-Julg}
	\script{Baaj, S.; Julg, P.}
	{\em Th\'eorie bivariante de Kasparov et op\'erateurs non born\'es dans les $C\sp{\ast} $-modules hilbertiens.}
	C. R. Acad. Sci. Paris Sr. I Math.  {\bf 296}  (1983), no. 21, 875--878.

\bibitem{Banagl}
	\script{Banagl, M.}
	{\em Topological invariants of stratified spaces.}
	Springer Monographs in Mathematics. 
	Springer, Berlin, 2007. xii+259 pp. ISBN: 978-3-540-38585-1; 3-540-38585-1
	
\bibitem{BCS}
\script{	Banagl, M.;  Cappell S.E.  and Shaneson J. L.} {\em Computing twisted signatures and L-
classes of stratified spaces.} Math. Ann. {\bf 326} (2003), no. 3, 589-623.

\bibitem{Blackadar}
	\script{Blackadar, B.}
	{\em $K$-theory for operator algebras.}
	Second edition. Mathematical Sciences Research Institute Publications, 5. 
	Cambridge University Press, Cambridge, 1998. xx+300 pp. ISBN: 0-521-63532-2 

\bibitem{BHS}
	\script{Brasselet, J-P. ; Hector, G.; Saralegi, M.}
	 {\em Th\`eor\'eme de De Rham pour les Vari\'et\'es Stratifi\'ees.}
	Ann. Global Anal. Geom.  {\bf 9}  (1991),  no. 3, pages  211-243.

\bibitem{BL}
	\script{Brasselet, J-P. ; Legrand, A.}
	{\em Un complexe de formes differentielles a croissance bornee sur une variete stratifiee}.
	Annali della Scuola Normale Superiore di Pisa, Classe di Scienze 4e serie, tome 21, No {\bf 2} (1994), 213-234. 

\bibitem{BS}
	\script{Br\"uning, J.; Seeley, R.}
	{\em An index theorem for first order regular singular operators.}
	Amer. J. Math.  {\bf 110}  (1988),  no. 4, 659--714.

\bibitem{Cheeger-symp}
	\script{Cheeger, J.}
	{\em On the Hodge theory of Riemannian pseudomanifolds.}
	Geometry of the Laplace operator (Proc. Sympos. Pure Math., Univ. Hawaii, Honolulu, Hawaii, 1979),  pp. 91--146,
	Proc. Sympos. Pure Math., XXXVI, Amer. Math. Soc., Providence, R.I., 1980. 

\bibitem{Ch} 
	\script{Cheeger, J.}
	{\em Spectral geometry of singular Riemannian spaces. }
	J. Differential Geom.  {\bf 18}  (1983),  no. 4, 575-657.

\bibitem{Cheeger-Dai}
	\script{Cheeger, J.;Dai, X.}
	{\em $L^2$-cohomology of a non-isolated conical singularity and nonmultiplicativity of the signature.}
	preprint.

\bibitem{CM}
	\script{Connes, A.; Moscovici, H.}
	{\em Cyclic cohomology, the Novikov conjecture and hyperbolic groups.}
	Topology  {\bf 29}  (1990),  no. 3, 345--388.

\bibitem{oberwolfach} 
	\script{Ferry, S. C.; Ranicki, A.; Rosenberg, J.}
	{\em A history and survey of the Novikov conjecture.}
	Novikov conjectures, index theorems and rigidity, Vol. 1 (Oberwolfach, 1993),  7--66, 
	London Math. Soc. Lecture Note Ser., 226, Cambridge Univ. Press, Cambridge, 1995. 

\bibitem{friedman}
\script{Friedman G.} {\em Stratifed  
fibrations and the intersection homology of the regular neighborhoods of bottom strata.}
Topology Appl. {\bf 134} (2003), 69--109. 



\bibitem{Gil-Mendoza}
	\script{Gil, J. B.; Mendoza, G. A.}
	{\em Adjoints of elliptic cone operators.}
	Amer. J. Math.  {\bf 125}  (2003),  no. 2, 357--408.

\bibitem{GM}
	\script{Goresky, M.; MacPherson, R.}
	{\em Intersection homology theory.}
	Topology  {\bf 19}  (1980), no. 2, 135--162.

\bibitem{higson-roe} 
   \script{Higson, N. and Roe, J.}
     {\em Analytic {$K$}-homology},
    Oxford Mathematical Monographs,
      Oxford University Press (2000),
     pp xviii+405.
     
      
\bibitem{Hilsum-LNM}
	\script{Hilsum, M.}
	{\em Signature operator on Lipschitz manifolds and unbounded Kasparov bimodules.}
	Operator algebras and their connections with topology and ergodic theory (Buteni, 1983),  254--288, 
	Lecture Notes in Math., 1132, Springer, Berlin, 1985.

\bibitem{Hilsum-K} 
	\script{Hilsum, M.}
	{\em Fonctorialit\'e en $K$-th\'eorie bivariante pour les vari\'et\'es lipschitziennes.}
	$K$-Theory  {\bf 3}  (1989),  no. 5, 401--440.

\bibitem{Hughes-Weinberger}
	\script{Hughes, B.; Weinberger, S.}
	{\em Surgery and stratified spaces.}
	Surveys on surgery theory, Vol. 2,  319--352, 
	Ann. of Math. Stud., 149, Princeton Univ. Press, Princeton, NJ, 2001.

\bibitem{Hunsicker-Mazzeo}
	\script{Hunsicker, E.; Mazzeo, R.}
	{\em Harmonic forms on manifolds with edges.}
	Int. Math. Res. Not.  2005,  no. {\bf 52}, 3229--3272.

\bibitem{Kasparov-inventiones} 
	\script{Kasparov, G.}
	{\em Equivariant KK-theory and the Novikov conjecture.}
	Invent. Math. {\bf 91} (1988), 147-201.

\bibitem{Kasparov-contemporary} 
	\script{Kasparov, G.}
	{\em Novikov's conjecture on Higher Signatures: The Operator K-Theory Approach.}
	Contem. Math. Vol. {\bf 145} (1993)

\bibitem{Kirwan-Woolf}
	\script{Kirwan, F.; Woolf, J.}
	{\em An introduction to intersection homology theory.}
	Second edition. Chapman \& Hall/CRC, Boca Raton, FL, 2006. xiv+229 pp. ISBN: 978-1-58488-184-1; 1-58488-184-4

\bibitem{Kordyukov}
	\script{Kordyukov, Y.}
	{\em $L\sp p$-theory of elliptic differential operators on manifolds of bounded geometry.}  
	Acta Appl. Math.  {\bf 23}  (1991),  no. 3,  pages 223-260.

\bibitem{Lance}
	\script{Lance, E. C.}
	{\em Hilbert $C\sp *$-modules.}
	A toolkit for operator algebraists. London Mathematical Society Lecture Note Series, 210. 
	Cambridge University Press, Cambridge, 1995. x+130 pp. ISBN: 0-521-47910-X 

\bibitem{LLP}
	\script{Leichtnam, E.; Lott, J.; Piazza, P.}
	{\em On the homotopy invariance of higher signatures for manifolds with boundary.}
	J. Differential Geom.  {\bf 54}  (2000),  no. 3, 561--633.

\bibitem{LP-SMF}
	\script{Leichtnam, E.; Piazza, P.}
	{\em The $b$-pseudodifferential calculus on Galois coverings and a higher Atiyah-Patodi-Singer index theorem.}
	M\'em. Soc. Math. Fr. (N.S.)  No. 68  (1997), iv+121 pp.


\bibitem{LPFOURIER}
\script{Leichtnam, E.; Piazza, P.}
 {\em Elliptic operators and higher signatures},
   Ann. Inst. Fourier (Grenoble),
  {\bf 54}, no. 5  (2004),  1197--1277.
       
\bibitem{L}
	\script{Lesch, M.}
	{\em Operators of Fuchs type, conical singularities, and asymptotic methods.}
	Teubner-Texte zur Mathematik [Teubner Texts in Mathematics], 136. 
	B. G. Teubner Verlagsgesellschaft mbH, Stuttgart, 1997. 190 pp. ISBN: 3-8154-2097-0

\bibitem{Lott}
	\script{Lott, J.}
	{\em Superconnections and higher index theory.}
	Geom. Funct. Anal.  {\bf 2}  (1992),  no. 4, 421--454.

\bibitem{Mather}
	\script{Mather, J. N.}
	{\em Stratifications and mappings.}
	Dynamical systems (Proc. Sympos., Univ. Bahia, Salvador, 1971),  pp. 195--232. Academic Press, New York, 1973.

\bibitem{Mazzeo:edge}
	\script{Mazzeo, R.}
	{\em Elliptic theory of differential edge operators. I.}
	Comm. Partial Differential Equations  {\bf 16}  (1991),  no. 10, 1615--1664.

\bibitem{Meladze-Shubin} 
	\script{Meladze, G.;Shubin, M.}
	{\em Algebras of pseudodifferential operators on unimodular Lie groups.} 
	Dokl. Akad. Nauk SSSR  {\bf 279}  (1984),  no. 3, pages 542-545.

\bibitem{APS Book} 
	\script{Melrose, R. B.}
	{\em The Atiyah-Patodi-Singer index theorem.} 
	Research Notes in Mathematics, 4. A K Peters, Ltd.,
	Wellesley, MA, 1993. xiv+377 pp. ISBN 1-56881-002-4

\bibitem{Melrose:Kyoto}
	\script{Melrose, R. B.}
	{\em Pseudodifferential operators, corners and singular limits.}
	Proceedings of the International Congress of Mathematicians, Vol. I, II (Kyoto, 1990),  217--234, Math. Soc. Japan, Tokyo, 1991.
	
\bibitem{Mich-Fomenko} 
	\script{Mischenko, A.; Fomenko, A.}
	{\em The index of elliptic operators over $C^*-$algebras.}
	Izv. Akad. Nauk SSSR, Ser. Mat. {\bf 43}, (1979), pages 831-859.

\bibitem{Mosc-Wu-Witt} 
	\script{Moscovici, H; Wu, F.}
	{\em Straight Chern Character for Witt spaces.}
	Fields Institute Communications Vol 17 (1997).

\bibitem{Mue}
	\script{M\"uller, W.}
	{\em Signature defects of cusps of Hilbert modular varieties and values of $L$-series at $s=1$.}
	J. Differential Geom.  {\bf 20}  (1984),  no. 1, 55--119.

\bibitem{NSS1}
	\script{Nazaikinskii, V. E.; Savin, A. Yu.; Sternin, B. Yu.}
        {\em Pseudodifferential operators on stratified manifolds. I} Differ. Uravn. {\bf 43} (2007), no. 4, 519-532.
	
\bibitem{NSS2}
	\script{Nazaikinskii, V. E.; Savin, A. Yu.; Sternin, B. Yu.}
        {\em Pseudodifferential operators on stratified manifolds. II} Differ. Uravn. {\bf 43} (2007), no. 5, 685-696.

\bibitem{Pfl}
	\script{Pflaum, M. J.}
	\em{ Analytic and geometric study of stratified spaces.}
	Lecture Notes in Mathematics, 1768. Springer-Verlag, Berlin, 2001. viii+230 pp. ISBN: 3-540-42626-4

\bibitem{rosenberg-anft}
        \script{Rosenberg, J.},
     \em{Analytic Novikov for topologists},
     Novikov conjectures, index theorems and rigidity, Vol. 1 (Oberwolfach, 1993),  338-372, 
	London Math. Soc. Lecture Note Ser., 226, Cambridge Univ. Press, Cambridge, 1995. 
		
\bibitem{ros-weinberger} 
	\script{Rosenberg, J; Weinberger, S.}
	{\em Higher $G$-signatures for Lipschitz manifolds.}
	$K$-Theory  {\bf 7}  (1993),  no. 2, 101--132.

\bibitem{Sch-MSRI}
	\script{Schulze, B.-W.}
	{\em The iterative Structure of Corner Operators.}
	preprint, 2009, arxiv.org/abs/0905.0977

\bibitem{Se}
	\script{Siegel, P. H.}
	{\em Witt spaces: a geometric cycle theory for $K{\rm O}$-homology at odd primes.}
	Amer. J. Math.  {\bf 105}  (1983),  no. 5, 1067--1105.

\bibitem{ST} 
	\script{Solovyev, Y. P.; Troitsky, E. V.}
	{\em $C^*$-algebras and elliptic operators in differential topology} 
	AMS, 2001, Translations of mathematical monographs, vol 192.

\bibitem{Vai}
	\script{Vaillant, B.}
	{\em Index and spectral theory for manifolds with generalized fibred cusps.}
	Ph.D. dissertation, Bonner Math. Schriften 344, 
	Univ. Bonn, Mathematisches Institut, Bonn, 2001.



\bibitem{vaillant-diploma}
\script{Vaillant, B.}
   {\em Indextheorie f\"ur \"Uberlagerungen},
Universit\"at Bonn,
 1997, Diplomarbeit. Available in English on  arXiv:0806.4043.
		
		
\bibitem{Ver} 
	\script{Verona, A.}
	{\em Stratified mappings-structure and triangulability.} 
	Lecture Notes in Mathematics, Springer-Verlag, 1102, (1984).

\end{thebibliography}
\end{document}